\documentclass{scrbook}
\usepackage[checktabu={false}, hints={false}, ,babel={ngerman,english}]{hudiss}
\usepackage{theorem}
\usepackage{amsmath}
\usepackage{amsfonts}
\usepackage{amssymb}
\usepackage{oldgerm}
\usepackage{dsfont}
\usepackage[all]{xy}
\entrymodifiers={!!<0pt,.7ex>+}
\usepackage{splitidx}
\usepackage{bookmark}
\usepackage{type1ec}
\usepackage{type1cm}

\newtheorem{theorem}{Theorem}[chapter]
\newenvironment{proof}{\textbf{Proof:} }{\text{q.e.d.}}
\newtheorem{assumption}[theorem]{Assumption}
\newtheorem{proposition}[theorem]{Proposition}
{\theorembodyfont{\rmfamily}
\newtheorem{notation}[theorem]{Notation}
\newtheorem{convention}[theorem]{Convention}

\newtheorem{definition}[theorem]{Definition}

\newtheorem{lemma}[theorem]{Lemma}
\newtheorem{corollary}[theorem]{Corollary}
\newtheorem{remark}[theorem]{Remark}

\newtheorem{example}[theorem]{Example}}
\newcommand{\ncmd}{\newcommand}
\ncmd{\be}{\begin{enumerate}}
\ncmd{\ee}{\end{enumerate}}
\ncmd{\bi}{\begin{itemize}}
\ncmd{\ei}{\end{itemize}}
\ncmd{\beq}{\begin{equation}}
\ncmd{\eeq}{\end{equation}}
\ncmd{\bdfn}{\begin{definition}}
\ncmd{\edfn}{\end{definition}}
\ncmd{\bprop}{\begin{proposition}}
\ncmd{\eprop}{\end{proposition}}
\ncmd{\bthm}{\begin{theorem}}
\ncmd{\ethm}{\end{theorem}}
\ncmd{\brem}{\begin{remark}}
\ncmd{\erem}{\end{remark}}
\ncmd{\bcor}{\begin{corollary}}
\ncmd{\ecor}{\end{corollary}}
\ncmd{\bex}{\begin{example}}
\ncmd{\eex}{\end{example}}
\ncmd{\bnot}{\begin{notation}}
\ncmd{\enot}{\end{notation}}
\ncmd{\bass}{\begin{assumption}}
\ncmd{\eass}{\end{assumption}}
\ncmd{\blem}{\begin{lemma}}
\ncmd{\elem}{\end{lemma}}
\ncmd{\bproof}{\begin{proof}}
\ncmd{\eproof}{\end{proof}}

\ncmd{\idxS}[1]{\sindex[indexsymb]{#1}}
\ncmd{\idxD}[1]{\sindex[indexdef]{#1}}

\ncmd{\lra}{\mbox{$\longrightarrow$}}
\ncmd{\Lra}{\mbox{$\Longrightarrow$}}
\ncmd{\ura}{\underrightarrow}
\ncmd{\ua}{\mbox{$\uparrow$}}
\ncmd{\da}{\mbox{$\downarrow$}}
\ncmd{\hra}{\mbox{$\hookrightarrow$}}
\ncmd{\ra}{\mbox{$\rightarrow$}}
\ncmd{\Ra}{\mbox{$\Rightarrow$}}
\ncmd{\eqa}{\mbox{$\leftrightarrow$}}
\ncmd{\Eqa}{\mbox{$\Leftrightarrow$}}
\ncmd{\leqa}{\mbox{$\longleftrightarrow$}}
\ncmd{\Leqa}{\mbox{$\Longleftrightarrow$}}
\ncmd{\iso}{\mbox{$\tilde{\ra}$}}
\ncmd{\liso}{\mbox{$\tilde{\lra}$}}

\ncmd{\Reinfty}{\mbox{$\mathbb{R}\cup\{\infty\}$}}
\ncmd{\bbN}{\mathbb{N}}
\ncmd{\bbZ}{\mbox{$\mathbb{Z}$}}
\ncmd{\bbR}{\mbox{$\mathbb{R}$}}
\ncmd{\ggT}{\operatorname{ggT}}
\ncmd{\Gcd}[2]{\mathop{\text{gcd}(#1,#2)}}

\ncmd{\tf}[1]{\textfrak{#1}}
\ncmd{\calli}[1]{\mathcal{#1}}
\ncmd{\gfrak}{\textfrak{g}}
\ncmd{\hfrak}{\textfrak{h}}

\ncmd{\ti}[1]{\tilde{#1}}
   \ncmd{\tiV}{\ti{V}}\ncmd{\tiD}{\ti{D}}\ncmd{\tisigma}{\ti{\sigma}}
   \ncmd{\tij}{\ti{j}}\ncmd{\tii}{\ti{i}}
   \ncmd{\titfh}{\ti{\textfrak{h}}}\ncmd{\titfgt}{\ti{\textfrak{g}}}
   \ncmd{\tiJ}{\ti{J}}\ncmd{\tiJplus}{\ti{J}^+}\ncmd{\tiG}{\ti{G}}
   \ncmd{\tibG}{\mathbf{\ti{G}}} \ncmd{\tibH}{\mathbf{\ti{H}}}
\ncmd{\leftexp}[2]{{\vphantom{#2}}^{#1}{#2}}

\ncmd{\graded}[2]{\mbox{$\tf{gr}_{#1}(#2):=\oplus_i({#1}^i{#2}/{#1}^{i+1}{#2})$}}

\ncmd{\umin}[1]{\underset{#1}{\min}}
\ncmd{\uinf}[1]{\underset{#1}{\inf}}
\ncmd{\usup}[1]{\underset{#1}{\sup}}
\ncmd{\umax}[1]{\underset{#1}{\max}}
\ncmd{\uomin}[2]{\underset{#1}{\overset{#2}{\min}}}
\ncmd{\uoinf}[2]{\underset{#1}{\overset{#2}{\inf}}}
\ncmd{\uosup}[2]{\underset{#1}{\overset{#2}{\sup}}}
\ncmd{\uomax}[2]{\underset{#1}{\overset{#2}{\max}}}

\ncmd{\sequence}[1]{\mbox{$({#1}_n)_\mathbb{N}$}}

\ncmd{\pF}{\tf{p}_F}
\ncmd{\oK}{\mbox{$o_K$}}
\ncmd{\pK}{\mbox{$\textfrak{p}_K$}}
\ncmd{\pD}{\mbox{$\textfrak{p}_D$}}
\ncmd{\pFpow}[1]{\mbox{$\textfrak{p}_F^{#1}$}}
\ncmd{\pDpow}[1]{\mbox{$\textfrak{p}_D^{#1}$}}
\ncmd{\oDelta}{\mbox{$o_{\Delta}$}}
\ncmd{\pDelta}{\mbox{$\textfrak{p}_{\Delta}$}}
\ncmd{\vr}[1]{\mbox{$o_{#1}$}}
\ncmd{\vi}[1]{\mbox{$\textfrak{p}_{#1}$}}
\ncmd{\vipower}[2]{\mbox{$\tf{p}_{#1}^{#2}$}}
\ncmd{\Char}{\operatorname{char}}

\ncmd{\Building}{\tf{B}^1}
\ncmd{\building}{\tf{B}}
\ncmd{\Ical}{\mathcal{I}}
\ncmd{\IAE}{\mbox{$\mathcal{I}^{E^{\times}}$}}
\ncmd{\IB}{\mbox{$\mathcal{I}_E$}}
\ncmd{\Om}[1]{\mbox{$\Omega_{#1}$}}
\ncmd{\OmE}[1]{\mbox{$\Omega_{E,{#1}}$}}
\ncmd{\bary}{\mathop{bary}}
\ncmd{\typ}{\operatorname{typ}}
\ncmd{\Stern}{\operatorname{Stern}}
\ncmd{\diam}{\operatorname{diam}}

\ncmd{\Norm}[2]{\operatorname{Norm}_{#1}(#2)}
\ncmd{\Normone}[2]{\operatorname{Norm}^1_{#1}(#2)}
\ncmd{\Normtwo}[2]{\operatorname{Norm}^2_{#1}(#2)}

\ncmd{\ord}{\oparatorname{ord}}
\ncmd{\Ord}{\operatorname{Ord}}
\ncmd{\lattices}[1]{\operatorname{lattices}({#1})}
\ncmd{\rad}[1]{\operatorname{rad}({#1})}
\ncmd{\Her}[1]{\operatorname{Her}({#1})}
\ncmd{\Latt}[2]{\operatorname{Latt}_{#1}(#2)}
\ncmd{\Gitter}[2]{\operatorname{Latt}(#2,#1)}
\ncmd{\LC}{\operatorname{LC}}
\ncmd{\Lattone}[2]{\operatorname{Latt}^1_{#1}(#2)}
\ncmd{\Latttwo}[2]{\operatorname{Latt}^2_{#1}(#2)}

\ncmd{\Inn}{\operatorname{Inn}}
\ncmd{\Centr}{\operatorname{Z}}
\ncmd{\Z}{\operatorname{Z}}
\ncmd{\lmult}{\operatorname{l}}
\ncmd{\rmult}{\operatorname{r}}
\ncmd{\ind}{\operatorname{ind}}
\ncmd{\Gal}{\mathop{\text{Gal}}}

\ncmd{\LietiG}{\mbox{$\mathop{Lie}(\tilde{G})$}}
\ncmd{\LF}{\operatorname{LF}}
\ncmd{\gyr}{\textfrak{g}_{x,r}}
\ncmd{\hxr}{\textfrak{h}_{x,r}}

\ncmd{\Matr}{\operatorname{M}}
\ncmd{\trd}{\operatorname{trd}}
\ncmd{\Nrd}{\operatorname{Nrd}}
\ncmd{\trace}{\operatorname{tr}}
\ncmd{\Gram}{\operatorname{Gram}}
\ncmd{\gr}{\operatorname{\tf{gr}}}
\ncmd{\degree}{\operatorname{deg}}
\ncmd{\Skew}{\operatorname{Skew}}
\ncmd{\Sym}{\operatorname{Sym}}
\ncmd{\diag}{\operatorname{diag}}
\ncmd{\antidiag}{\operatorname{antidiag}}
\DeclareMathOperator{\row}{row}
\DeclareMathOperator{\MopRow}{Row}
\ncmd{\EMatr}[1]{\mathds{1}_{#1}}
\ncmd{\Mint}[3]{\mathop{\text{M}_{#1,#2}(#3)}}
\ncmd{\trans}[1]{{#1}^{\ensuremath{\mathsf{T}}}} 

\ncmd{\Lie}{\operatorname{Lie}}
\ncmd{\Abb}{\operatorname{Abb}}
\ncmd{\Top}{\operatorname{Top}}
\ncmd{\Hom}{\operatorname{Hom}}
\ncmd{\End}{\operatorname{End}}
\ncmd{\Aut}{\operatorname{Aut}}
\ncmd{\Iso}{\operatorname{Iso}}

\ncmd{\im}{\operatorname{im}}
\ncmd{\rang}{\operatorname{rang}}
\DeclareMathOperator{\image}{im}
\DeclareMathOperator{\rank}{rank}
\ncmd{\id}{\operatorname{id}}

\ncmd{\Span}{\operatorname{span}}
\ncmd{\tens}[3]{{#1}\otimes_{#2}{#3}}
\ncmd{\proj}{\operatorname{proj}}

\ncmd{\PO}{\operatorname{PO}}
\ncmd{\SL}{\text{SL}}
\ncmd{\GL}{\text{GL}}
\ncmd{\Sp}{\operatorname{Sp}}
\ncmd{\SO}{\operatorname{SO}}
\ncmd{\Ogp}{\operatorname{O}}
\ncmd{\Rad}{\operatorname{Rad}}
\ncmd{\nr}{\operatorname{nr}}
\ncmd{\GLDV}{\mbox{$\GL_DV$}}
\ncmd{\SLDV}{\mbox{$\SL_DV$}}
\ncmd{\U}{\operatorname{U}}
\ncmd{\SU}{\operatorname{SU}}
\ncmd{\AfSp}[2]{\operatorname{\mathbf{A}}^{#1}_{K}}
\ncmd{\Res}{\operatorname{Res}}
\ncmd{\bA}{\operatorname{\mathbf{A}}}
\ncmd{\bGm}{\operatorname{\mathbf{G}_{\mathbf{m}}}}
\ncmd{\bG}{\operatorname{\mathbf{G}}}
\ncmd{\bH}{\operatorname{\mathbf{H}}}
\ncmd{\bMatr}{\operatorname{\mathbf{M}}}
\ncmd{\bU}{\operatorname{\mathbf{U}}}
\ncmd{\bSU}{\operatorname{\mathbf{SU}}}
\ncmd{\bSL}{\operatorname{\mathbf{SL}}}
\ncmd{\bO}{\operatorname{\mathbf{O}}}
\ncmd{\bSO}{\operatorname{\mathbf{SO}}}
\ncmd{\bSp}{\operatorname{\mathbf{Sp}}}
\ncmd{\bGL}{\operatorname{\mathbf{GL}}}
\ncmd{\bT}{\operatorname{\mathbf{T}}}
\ncmd{\bAfSp}[1]{\operatorname{\mathbf{A}^{#1}}}
\ncmd{\Weylgroup}{\mbox{$\leftexp{v}{W}$}}
\ncmd{\Osheaf}[1]{\mbox{$\textfrak{O}_{#1}$}}
\ncmd{\RingedSpace}[1]{\mbox{$(#1,\Osheaf{#1})$}}
\ncmd{\maximalIdeal}[1]{\mbox{$\textfrak{m}_{#1}$}}

\ncmd{\Invariante}[2]{\mathop{\text{Iv}({#1},{#2})}}
\ncmd{\pairs}{\mathop{\text{pairs}}}

\ncmd{\zzmatrix}[4]{\left( \begin{array}{cc}#1&#2\\#3&#4 \end{array}\right)}

\ncmd{\Eb}{Euclidean building}
\ncmd{\Gr}{geometric realisation}
\ncmd{\vs}{vector space}
\ncmd{\nd}{non-degenerated}
\ncmd{\na}{non-archimedean}
\ncmd{\fd}{finite dimensional}
\hudissmetadata{
 authorprefix={Dipl. math.},
 authorfirstname={Daniel},
 authorsurname={Skodlerack},
 authorsuffix={geb. Dunker}, 
 authoradd={geb. 26.03.1978 in Berlin}, 
 doctitle={Embedding types and canonical affine maps between Bruhat-Tits buildings of classical groups}, 
 approvala={Herr Prof. Dr. Ernst-Wilhelm Zink}, 
 approvalb={Herr Dr. habil. Paul Broussous},
 approvalc={Herr Dr. habil. Bertrand Lemaire},
 degree={doctor rerum naturalium (Dr. Rer. Nat.)},
 subject={Mathematik}, 
 faculty={Mathematisch-Naturwissenschaftlichen Fakult\"at II}, 
 university={Humboldt-Universit\"at zu Berlin}, 
 dean={Herr Prof. Dr. Peter Frensch}, 
 president={Herr Prof. Dr. Dr. h.c. Christoph Markschies},
 datesubmitted={13.04.2010}, 
 dateexam={29.06.2010}, 
 keywordsen={Tits building, embedding types, classical groups, unitary groups}, 
 keywordsde={Tits-Geb\"aude, Einbettungstypen, klassische Gruppen, Unit\"are Gruppen} 
}
\makeindex
\begin{document}
\newindex[Index of definitions]{indexdef}
\newindex[Index of notation]{indexsymb}
\makeatletter
\begin{titlepage}
\begin{spacing}{1}
\newlength{\parindentbak}
  \setlength{\parindentbak}{\parindent}
\newlength{\parskipbak}
  \setlength{\parskipbak}{\parskip}

\setlength{\parindent}{0pt}
  \setlength{\parskip}{\baselineskip}

\thispagestyle{empty}

\expandafter\Hu@titlepagefont\expandafter

\begin{center}
{\LARGE \textbf{\ifx\Hu@doctitle\empty\default{doctitle}\else\Hu@doctitle\fi}}

\ifx\Hu@docsubtitle\empty\default{}\else\Hu@docsubtitle\fi

\so{DISSERTATION}

zur Erlangung des akademischen Grades

\ifx\Hu@degree\empty\default{degree}\else\Hu@degree\fi\\
im Fach
\ifx\Hu@subject\empty\default{subject}\else\Hu@subject\fi

eingereicht an der\\
\ifx\Hu@faculty\empty\default{faculty}\else\Hu@faculty\fi\\
\ifx\Hu@university\empty\default{university}\else\Hu@university\fi

von\\
\textbf{\ifx\Hu@authorprefix\empty\default{authorprefix}\else\Hu@authorprefix\fi\ \ifx\Hu@authorfirstname\empty\default{authorfirstname}\else\Hu@authorfirstname\fi\ \ifx\Hu@authorsurname\empty\default{authorsurname}\else\Hu@authorsurname\fi\ \ifx\Hu@authorsuffix\empty\default{authorsuffix}\else\Hu@authorsuffix\fi}\\
\ifx\Hu@authoradd\empty\default{authoradd}\else\Hu@authoradd\fi
\end{center}

\vfill

Pr\"asident der \ifx\Hu@university\empty\default{university}\else\Hu@university\fi:\\
\ifx\Hu@president\empty\default{president}\else\Hu@president\fi

Dekan der \ifx\Hu@faculty\empty\default{faculty}\else\Hu@faculty\fi:\\
\ifx\Hu@dean\empty\default{dean}\else\Hu@dean\fi

Gutachter:
\begin{tplist} 
  \ifx\Hu@approvala\empty\item \default{approvala}\else\item \Hu@approvala\fi
  \ifx\Hu@approvalb\empty\else\item \Hu@approvalb\fi
  \ifx\Hu@approvalc\empty\else\item \Hu@approvalc\fi
  \ifx\Hu@approvald\empty\else\item \Hu@approvald\fi
  \ifx\Hu@approvale\empty\else\item \Hu@approvale\fi
\end{tplist}

\begin{exlist}
  \item[eingereicht am:] \ifx\Hu@datesubmitted\empty\default{datesubmitted}\else\Hu@datesubmitted\fi 
  \item[Tag der m\"undlichen Pr\"ufung:] \ifx\Hu@dateexam\empty\default{dateexam}\else\Hu@dateexam\fi
\end{exlist}

\setlength{\parindent}{\parindentbak}
\setlength{\parskip}{\parskipbak}
\end{spacing}
\end{titlepage}
\makeatother
\tableofcontents
\newpage
\begin{abstract}
\section*{Introduction} 
This thesis is devoted to the description and characterisation of affine maps between enlarged Bruhat-Tits buildings of certain reductive groups 
over non-Archime\-dean local fields. More precisely we consider subgroups of classical groups which arise 
as the centraliser of a rational Lie algebra element which generate a semisimple algebra over the base field, and 
we study affine embeddings of the corresponding Bruhat-Tits buildings. Our approach is based on the fact that the building can be described in terms of lattice functions. 

In part two we consider unit groups of local central simple algebras or in other words general linear groups with coeffitients in a local division algebra under no assumption on the characteristic. Here we use the affine embeddings of Bruhat-Tits buildings of centraliser subgroups in order to recover the embedding data from the work of Broussous and Grabitz. Embedding data play an important role in the construction of simple types. It is to underline the usefulness of such affine embeddings and that we may expect further results in the future. 

\subsection*{Part 1}
For the construction of types for p-adic unitary groups S. Stevens applied a result of his paper with P. Broussous \cite{broussousStevens:09}. He used a map between the enlarged buildings of a centraliser and the group to apply an induction. The important property of this map is the compatibility with the Lie algebra filtrations (\textbf{CLF}) which correspond to the Moy-Prasad filtrations \cite{moyPrasad:94}. In that paper the quaternion algebra case is missing and the authors proposed a uniqueness and generalization conjecture to the reader. 
Functoriality questions for affine maps between Bruhat-Tits buildings have been studied before.
One work to mention is Landvogt's paper \cite{landvogt:00} and the paper of P. Broussous with B.Lemaire \cite{broussousLemaire:02} and with S.Stevens \cite{broussousStevens:09} 
which I am going to explain in more detail below. For example E. Landvogt already proved in \cite[2.1.1.]{landvogt:00} that for an inclusion $H\subseteq G$ of connected reductive $K$-groups there is toral and $H(L)$- and $\Gal(L|K)$-equivariant map from the enlarged Bruhat-Tits building of $H(L)$ into that of $G(L)$ such that after a normalisation of the metric of the latter building the map is isometrical. In his work he assumed $L|K$ to be a quasi-local extension. 

We consider a p-adic field $k_0$ of residue characteristic not two, a $k_0$-form $\bG:=\bU(h)$ of $\bGL_n,\ \bSp_n,$ or $\bO_n$ where $h$ is a 
signed hermitian form, a Lie algebra element $\beta\in\Lie(\bG)(k_0)$ such that $k_0[\beta]$ is semisimple and its centraliser $\bH:=\bU(h)_{\beta}.$ P. Broussous and S.Stevens give a model in terms of lattice functions for the enlarged Bruhat-Tits building $\Building(\bH,k_0)$ if $\beta$ is separable. This model leads them to the definition of $\Building(\bH,k_0)$ if $\beta$ is not separable. We embed $\Building(\bH,k_0)$ into $\Building(\bG,k_0)$ by an affine, $\bH(k_0)$-equivariant CLF-map $j.$  In 
\cite{broussousLemaire:02} such a map was fully studied in the other case of $\bGL_n(D)$ instead of $\bU(h)$ and in 
\cite{broussousStevens:09} the authors considered the case where the image of $h$ is a field and $k_0$ has an odd residual characteristic. In the latter paper the authors showed that if $\beta$ is non-zero and generates a field then the CLF-property determines $j.$ 
In this thesis we consider the general case, more precisely we include the quaternion algebra case and we analyse uniqueness without any further restriction on $\beta.$ The group $\bH$ decomposes under $\beta$ into classical groups $\bH_i.$ We construct the map $j$ such that it has the above properties, see theorem \ref{thmExistence}, and we prove at first in theorem 
\ref{thmUniquenessForJGLIsEmpty} that there is no other CLF-map from $\Building(\bH,k_0)$ to $\Building(\bG,k_0)$ if the groups $\bH_i$ are unitary, i.e. of the form $\bU(h_i),$ and not $k_0$-isomorphic to the isotropic $\bO_2.$  Secondly we show that in general a $\Centr(\bH^0(k_0))$-equivariant, affine CLF-map from $\Building(\bH,k_0)$ to $\Building(\bG,k_0)$ has to be unique up to a translation of the building $\Building(\bH,k_0),$ see \ref{thmCLFUnitaryCase}. A summary of the theorems of part one is given in chapter 
\ref{chapterSummaryOfTheTheorems}.
For the buildings we use the model with lattice functions which are introduced in \cite{broussousLemaire:02} and \cite{broussousStevens:09}.

The aim of chapter 1 is to give the exact definition of $\bGL_D(V)$ and $\bU(h).$ We also repeat the notion of a signed hermitian form and a Witt decomposition.

Chapter 2 relies heavily on results which are summarised in the appendix. The second aim of this work is to give a complete definition of the Bruhat-Tits building for $\bGL_D(V)$ over a p-adic field and for    
$\bU(h)$ over a p-adic field of residue characteristic not two. The way of construction is taken from 
the articles of Bruhat and Tits. We give the definition of several kinds of lattice functions and shortly introduce the Lie algebra filtrations. 

For the next two chapters we fix a separable Lie algebra element until section 
\ref{secUniquenessInTheGeneralCase}.
In chapter 3 the section \ref{secCLFProperty} is devoted to the definition of the CLF-property. After recalling results of \cite{broussousLemaire:02} we prove the existence of a CLF-map in the case of $\bU(h).$ The proof of the torality of the constructed map is given in chapter  \ref{chapterTorality}.

Chapter 4 provides the proof of the uniqueness results stated above and in section \ref{secGeneralisationToTheNonseparableCase} we show how the preceding results of chapter 3 and 4 generalise to the case of a non-separable Lie algebra element.

\subsection*{Part 2}
In the whole part 2 we consider a finite dimensional skewfield $D$ with centre a p-adic field $F.$
Embedding types were introduced in the paper of Broussous and Grabitz \cite{broussousGrabitz:00}. They considered one step on the way to construct the smooth dual of $G:=\GL_m(D)$ using Bushnell and Kutzko's strategy \cite{bushnellKutzko:93} for $\GL_n(F).$ The aim is to produce a list of possible candidates for simple types, i.e. a list of pairs $(J,\lambda)$ consisting of a compact mod center subgroup $J$ and a smooth irreducible representation of $J$ with two properties. The second property states that if two paires are contained in the same irreducible representation of $G$ then they are conjugate under the action of $G.$ The idea is to construct the list of $(J,\lambda)$ by an inductive procedure using simple strata. A simple stratum is a quadruple $[\tf{a},n,q,\beta],$ especially consisting of a hereditary order $\tf{a}$ normalised by an element $\beta$ of $A:=\Matr_m(D)$ which generates a field $E.$ To prove the second property above they needed a rigidity for simple strata relying on a description of the way $E|F$ is embedded in $A.$ This was done by introducing numerical invariants.

Let $E_D|F$ be the maximal unramified subextension of $E|F$ which can be embedded in $D.$ 
In part 2 we show how to obtain the embedding type of $(E,\tf{a})$ if one applies  
\[j_{E_D}:\building(\bGL_m(D),F)^{E_D^{\times}}\ra\building(Z_{\bGL_m(D)}(E_D),E_D)\] on the barycenter $M_{\tf{a}}$ of $\tf{a}$ (see \ref{thmconnection}). The inverse of $j_{E_D}$ is the unique CLF-map from the latter building into $\building(\bGL_m(D),F).$ The map was constructed and analysed by Broussous and Lemaire in \cite{broussousLemaire:02}.

In chapter 8 we recall the definition of embedding type and we introduce the easy numerical tools which enables us to decode the embedding type from $y:=j_{E_D}(M_{\tf{a}}).$

The aim of chapter 9 is to describe the simplicial structure of $\building(\bGL_m(D),F)$ in terms of lattice chains as it has been done in \cite{broussousLemaire:02}.

In chapter 10 we state the connection between the oriented barycentric coordinates of $y,$ i.e. the so called local type
of $y,$ and the embedding type of $(E,\tf{a}).$

At the end of the whole introduction I want to thank P. Broussous and Prof. Zink for giving me the first and the second topic respectively and for the whole and patient support. Broussous mainly gave me the hint to use roots for proving lemma \ref{lemBeta=0} and he helpfully pointed out mistakes and proofread my notes several times.
Prof. Zink proofread the second part and introduced me into its background.
I thank S. Stevens for stimulating discussions about or around the topic in Norwich. 
At the end I thank the DFG for supporting my doctorial between January 2006 and December 2008.

\section*{General notation}

\be
\item All rings we consider in this thesis are unital.
\item The set of natural numbers starts with 1 and the set of the first $r$ natural numbers is denoted by ${\bbN}_r.$\idxS{${\bbN}_r$} For the set of non-negative integers we use the symbol $\mathbb{N}_0$\idxS{$\mathbb{N}_0$} and the set of its first $r+1$ elements is writen as $\mathbb{N}_0^r.$\idxS{$\mathbb{N}_0^r$} 
\item If $k$ is a non-Archimedean local field with valuation $\nu$ we 
denote by
\bi
\item $o_k$ the valuation ring of $k,$\idxS{$o_k$}
\item $\tf{p}_k$\idxS{$\tf{p}_k$} the valuation ideal of $k$ and by
\item $\kappa_k$\idxS{$\kappa_k$} the residue field of $k.$
\ei
\item For an arbitrary field $k$ we fix an algebraic closure $\bar{k}$ and the maximal inseparable (resp. separable) field extension of $k$ in $\bar{k}$ is denoted by $k^{isep}$ (resp. $k^{sep}$).
\item If we have fixed a local field $(k,\nu)$ we also write $\nu$ for the unique extension of $\nu$ to $D$ for any finite dimensional skewfield extension $D|k.$  We also use the notation $o_D,\tf{p}_D$ and $\kappa_D.$ We write $e(D|k)$\idxS{$e(D\mid k)$} for the ramification index and 
$f(D|k)$\idxS{$f(D\mid k)$} for the inertia degree.\idxS{$o_D$}\idxS{$\tf{p}_D$}\idxS{$\kappa_D$}
\item The symbol $\Centr (N)$ denotes the center of $N$ and we write $\Centr_N(M)$ for the centraliser of $M$ in $N$ for a set $N$ with multiplication and a subset $M$ of $N.$\idxS{$\Centr_N(M),$ centraliser of $M$ in $N$}
\idxS{$\Centr(G),$ centre of $G$}
\item The $m\times m$ identity matrix is denoted by $\EMatr{m}.$\idxS{$\EMatr{m}$}
\ee  
\end{abstract}

\cleardoublepage
\part{Canonical maps, buildings and centralisers}\label{part1}
We fix a field $k.$ We have the following conventions on $k.$
\bi
\item In section \ref{subsecAlgebrasWithInvolution} and \ref{secFormsOfClassicalGroups} the characteristic of $k$ is not two.
\item In this part from chapter \ref{chapBruhatTitsBuildingOfAClassicalGroup} on we assume $k$ to be a non-Archimedian local 
field with discrete valuation $\nu.$ 
\item In this part from section \ref{secSelfDualLatticeFunctions} on we further assume $k$ to 
have residue characteristic not two.
\ei
 
\chapter{Classical groups}
\section{Algebraic preliminaries}


\subsection{Semisimple algebras}

In this subsection we do not need any restriction on the characteristic of $k.$
A finite dimensional $k$-algebra $A$ is \textit{simple}\idxD{simple algebra} if there is no ideal of $A$ which is different from $\{0\}$ and $A,$ and it is semisimple
if $A$ has no nilpotent ideal except the zero-ideal.

\begin{theorem}(Wedderburn)
If $A$ is a finite dimensional semisimple $k$-algebra there is a unique natural number $m$
and an $m$-tuple 
\[(p_1,\ldots,p_m)\]
of pairs
\[p_i=([D_i],n_i)\]
consisting of a $k$-algebra isomorphism class of a skewfield $D_i$ and a natural number $n_i$ such that there is a $k$-algebra isomorphism
\[A\cong\prod_{i=1}^m\Matr_{n_i}(D_i).\]\idxS{$\Matr_{n}(D)$ set of matrices of size $n\times n$ with entries in $D$}
Up to permutation the $m$-tuple $(p_1,\ldots,p_m)$ is uniquely determined by $A.$
\end{theorem}

\begin{remark}
\be
\item If $A$ in the theorem is commutative it is $k$-isomorphic to a product of fields.
\item A finite dimensional simple $k$-algebra is $k$-isomorphic to a matrix ring, because $A$ is unital by our general notation.  
\ee
\end{remark}

\begin{definition}
A finite dimensional $k$-algebra is \textit{separable}\idxD{separable algebra} if 
for every field extension $L|k$ the $L$-algebra
$\tens{A}{k}{L}$ is semisimple.
\end{definition}

\begin{remark}
A commutative finite dimensional $k$-algebra is separable 
if and only if it is $k$-isomorphic to a product of separable 
extension fields of $k.$
\end{remark}

\begin{definition}
An element $\beta$ of a finite dimensional $k$-algebra $A$ is 
 \textit{separable}\idxD{separable} if the $k$-algebra
$k[\beta]$ is separable.
\end{definition}

\begin{definition}\label{defSimpleDatum}
A triple $(A,V,D)$ consisting of:
\bi
\item a skewfield $D$ which is a finite dimensional $k$-algebra
\item a finite dimensional right $D$-vector space $V$ 
\item and $A:=\End_D(V)$
\ei
is called a \textit{simple $k$-datum}.\/\idxD{simple $k$-datum}
Such a datum is \textit{central}\/\idxD{central simple $k$-datum} if the center of $D$ is $k.$
If we want to emphasize $m:=\dim_DV$ and $d:=\deg(D)$ then we
write $(A,V,m,D,d)$ instead of $(A,V,D)$ and if in addition we need a 
maximal Galois extension $L$ of $k$ in $D$ we write
$(A,V,m,D,d,L|k).$\idxS{$(A,V,m,D,d,L\mid k)$ simple $k$-datum}
A simple datum is \textit{local}\idxD{local simple datum} if $k$ is a non-Archimedean local 
field.
\end{definition}


\subsection{Algebras with involution}\label{subsecAlgebrasWithInvolution}
For references we recommand the books \cite{knus:91},\\ \cite{knus:98} and \cite{scharlau:85}. We assume $\Char(k)\neq 2.$

An \textit{involution}\/\idxD{involution} of a ring is an antimultiplicative ring automorphism of order 1 or 2.

\begin{notation}
\be
\item If $\sigma$ is an involution on the ring $R,$ we introduce the following standard notation.
\[\Sym(R,\sigma):=\{r\in R\mid \sigma(r)=r\}\]\idxS{$\Sym(R,\sigma)$}
\[\Skew(R,\sigma):=\{r\in R\mid \sigma(r)=-r\}\]\idxS{$\Skew(R,\sigma)$} 
for the set of symmetric and the set of skewsymmetric elements of $R.$
\item For a semisimple finite dimensional $k$-algebra $A$ with involution $\sigma$
we denote by 
\[\U(\sigma):=\{a\in A^\times|\ \sigma(a)a=1\}\]\idxS{$\U(\sigma)$}
the \textit{unitary group}\/\idxD{unitary group} of $(A,\sigma)$
and by 
\[\SU(\sigma):=\{a\in \U(\sigma)|\ \Nrd(a)=1\}\]\idxS{$\SU(\sigma)$}
the \textit{special unitary}\/\idxD{special unitary group} group of $(A,\sigma).$
\ee
The symbol $\Nrd$ denotes the reduced norm of $A$ over $k.$\idxS{$\Nrd$}
\end{notation}

\begin{definition}
Let $A$ be a simple central finite dimensional $k$-algebra.
An involution $\sigma$ of $A$ is \textit{of the first kind}\/\idxD{involution of the first kind}
if it fixes every element of $\Centr(A)$ and 
\textit{of the second kind}\/\idxD{involution of the second kind} otherwise.
We call an involution of the first kind on $A$ \textit{orthogonal}\/\idxD{orthogonal involution}
if 
\[\dim_k\Sym(A,\sigma)>\dim_k \Skew(A,\sigma)\]
and \textit{symplectic}\/\idxD{symplectic involution} otherwise.
An involution of the second kind is also called \textit{unitary}\idxD{unitary involution}.
\end{definition}

\begin{remark}
For a symplectic involution on $A$ we have 
\[\dim_k\Sym(A,\sigma)<\dim_k \Skew(A,\sigma)\]
and for a unitary involution we get an
equality.
\end{remark}

\begin{assumption}
For this section we fix a central simple $k$-datum $(A,V,m,D,d).$
We assume that there is an involution $\rho$ on $D.$
\end{assumption}

One way to obtain an involution on $A$ is to take the adjoint involution of an $\epsilon$-hermitian form.

\begin{definition}
Fix an $\epsilon$ which is $1$ or $-1.$
An \textit{$\epsilon$-hermitian form on $V$}\idxD{hermitian form} is a biadditive map $h$ from $V\times V$ to $D$ such that
\be
\item $h(v,w)=\epsilon\rho(h(w,v))$ for all $v,w\in V$ and 
\item $h$ is sesquilinear in the first coordinate and linear in the second coordinate, i.e. 
\[ h(wd_1,vd_2)=\rho(d_1)h(w,v)d_2,\]
and 
\item it is non-degenerate.
\ee
The pair $(V,h)$ is called an \textit{$\epsilon$-hermitian space.}\idxD{hermitian space}
An $\epsilon$-hermitian form $h_1$ and a $\delta$-hermitian form $h_2$ are equivalent if there is an element $b\in k^\times$ such that $bh_1=h_2.$  
An involution $\sigma$ of $\End_D(V)$ is called \textit{the adjoint involution}\/\idxD{adjoint involution} of $h,$ and it is denoted by $\sigma_h,$ if for every $a\in\End_D(V)$ and for every $v,w\in V$ we have \[h(a(v),w)=h(v,\sigma(a)(w)).\]  
\end{definition}

If we do not want to state the $\epsilon$ of an $\epsilon$-hermitian form we write \textit{signed} \idxD{signed hermitian form} hermitian form.  
There is a general notion of quadratic forms given in \cite[7.3.3]{scharlau:85} which includes the notion of signed hermitian forms for the case of characteristic different from 2.

\begin{proposition}\label{propHermFInv}
There is a bijection from the set of equivalence classes of signed hermitian forms to the set of involutions of 
$\End_D(V)$ whose restriction to $\Centr(D)$ is $\rho.$ 
\end{proposition}

For a given signed hermitian form $h$ we have the map
\[\hat{h}:\ V\ra V^*,\ \hat{h}(v)(w):=h(v,w),\]
where $V^*$ is the dual vector space of $V.$
It is an isomorphism of $D$-left vector spaces where $D$ acts on $V$ on the left via $dv:=v\rho(d).$

\begin{proof}
Equivalent signed hermitian forms have the same adjoint involution. 
To prove the injectivity assume $\sigma_{h_1}=\sigma_{h_2}$ and we call this involution $\sigma.$ It implies that 
$\psi:=\hat{h}_2^{-1}\circ \hat{h}_1$ is in the center of $\End_D(V),$ because for $a\in\End_D(V)$ and $v,w\in V$ we have 
\begin{align*}
h_2(\psi(a(v)),w) &= h_1(a(v),w)\\
&= h_1(v,\sigma(a)(w))\\
&= h_2(\psi(v),\sigma(a)(w))\\
&= h_2(a(\psi(v)),w).
\end{align*} 
Thus $h_1$ and $h_2$ are equivalent. 
If one chooses a $D$-basis $(v_i)$ of $V$ the map 
\[h_{(v_i)}:\ (\sum_iv_i\lambda_i,\sum_iv_i\mu_i)\mapsto \sum_i\rho(\lambda_i)\mu_i\] 
is a hermitian form whose adjoint involution $\sigma_{(e_i)}$ equals to $\rho$ on $\Centr(D).$
We identify $A$ with $\Matr_m(D)$ and we have $\sigma_{(e_i)}(C)=\rho(C)^T$ where $\rho(C)$ is meant to be the matrix obtained after applying $\rho$ to every entry of $C.$
The surjectivity is now given by the Skolem-Noether theorem, more precisely: we take an involution $\sigma$ whose restriction to $\Centr(D)$ is $\rho.$ By the Skolem-Noether theorem there is a $B\in\GL_m(D)$ which satisfies
\[\sigma(C)=B\rho(C)^TB^{-1}\] for all matrices $C.$ By $\sigma^2=\id$ we get that 
there is a $\lambda\in\Centr(D)$ such that $B\rho(B^{-1})^T=\lambda \EMatr{m}$ and $\rho(\lambda)\lambda$ is $1$ by $\sigma_{(e_i)}^2=\id.$

\textbf{Case 1:} If $\rho$ fixes $\lambda,$ the matrix $B^{-1}$ is a Gram matrix of a $\lambda$-hermitian form with adjoint involution $\sigma.$

\textbf{Case 2:} If $\lambda$ is not a fixed point of $\rho,$ Hilbert's 90th theorem \cite[12.3]{kersten:90} implies the existence of an element $\alpha\in\Centr (D)^{\times}$ such that 
\[\alpha\rho(\alpha)^{-1}=\lambda.\] 
Thus $\alpha B^{-1}$ is the Gram matrix of a $1$-hermitian form whose adjoint involution is $\sigma.$
\end{proof}
\vspace{1em}

To study groups $\U(\sigma)$ we need Witt's theorem. 

\begin{definition}\idxS{$\antidiag(\ldots)$}
An $r\times r$-matrix $M$ is \textit{antidiagonal}\/\idxD{antidiagonal} if all entries $m_{ij}$ with $i+j\neq r+1$ are zero. We denote an antidiagonal matrix $M$ by 
\[\antidiag(m_{r,1},m_{r-1,2},\ldots,m_{1,r}).\] 
\end{definition}

The following theorem uses $\Char(k)\neq 2.$

\begin{theorem}{(Witt) \cite[7.9.2 (iii)]{scharlau:85}}\label{thmWitt}
Let $h$ be an $\epsilon$-hermitian form of $V.$ Then there is a basis $(v_i)$ of $V$ such that the Gram matrix 
$\Gram_{(v_i)}(h)$ of $h$ over $(v_i)$ has the form
\beq\label{eqGramMatr}
\left(\begin{array}{ccc}
0 & M & 0 \\
\epsilon M & 0 & 0 \\
0 & 0 & B
\end{array}\right)
\eeq
such that $M=\antidiag(1,\ldots,1)$ and that the $\epsilon$-hermitian form corresponding to $B$ is anisotropic. 
\end{theorem}

\begin{definition}
Under the assumptions of the theorem the size $r$ of $M$ does not 
depend on the basis $(v_i).$ We call $r$ the \textit{Witt index of $h.$}\idxD{Witt index}
\end{definition}

\begin{definition}\label{defWittDecomp}
Let $r$ be the Witt index of a signed hermitian form $h.$
A set of 1-dimensional vector subspaces 
\[\{V_1,V_{-1},V_2,V_{-2},\ldots V_{r},V_{-r}\}\]
together with an anistropic $D$-subvector space $V_0$ of $V$
such that 
\[\bigoplus_i V_i=V\]
is called a \textit{Witt decomposition}\idxD{Witt decomposition} of $h$ if for all 
\[i,j\in\{1,-1,2,-2,\ldots,r,-r\}\]
we have
\be
\item $h(V_i,V_j)=0$ if $i\neq -j$ and
\item $h(V_i,V_0)=0.$
\ee
\end{definition}

\begin{definition}
If $\rho$ is the identity of $D,$ implying $D$ equals $k,$
then a $1$-hermitian form is called \textit{symmetric bilinear 
form}\/\idxD{symmetric bilinear form} and a $-1$-hermitian form is said to be a \textit{skew symmetric}\/\idxD{skew symmetric bilinear form} bilinear form.
\end{definition}

\begin{theorem}\cite[7.6.3]{scharlau:85}\label{thmOrthogonalSum}
If we exclude the skew symmetric case from the assumptions of the last theorem, i.e. the case where $\epsilon=-1$ and $\rho=\id_D,$ the vector space $V$ has an orthogonal basis.
\end{theorem}

\begin{definition}
A $D$-basis of $V$ such that the Gram matrix of an $\epsilon$-hermitian form $h$ is of the form in theorem \ref{thmWitt} such that 
the matrix $B$ is diagonal is called a \textit{Witt basis}\idxD{Witt basis} of $h.$
\end{definition}

\begin{corollary}\label{corExistOfWittBasis}
Under the assumptions of theorem \ref{thmWitt} for every 
Witt de\-com\-position of $h$ there is a Witt basis of $h$ such that the isotropic vectors of the basis 
span the isotropic lines and the other vectors together span the 
anisotropic vector space.  
\end{corollary}

\begin{notation}\idxS{$\bU(h)$}
\idxS{$\bSU(h)$}
We set 
\[\U(h):=\U(\sigma_h) \text{ and } \SU(h):=\SU(\sigma_h)\] if $h$ is an $\epsilon$-hermitian form on $V.$ 
\end{notation}

\begin{definition}\label{defLocalHermkDatum}
A \textit{hermitian $k$-datum}\idxD{hermitian $k$-datum} is a tuple 
\[((A,V,D),\rho,k_0,h,\epsilon,\sigma)\]\idxS{$((A,V,D),\rho,k_0,h,\epsilon,\sigma)$}
\bi
\item $(A,V,D)$ is a central simple $k$-datum,
\item $\rho$ is an involution of $D$ whose set of central fixed points is $k_0,$
\item $h$ is an $\epsilon$-hermitian form on $V$ with adjoint involution $\sigma.$
\ei
A hermitian datum is \textit{local}\idxD{local hermitian $k$-datum} if $k$ is a non-Archimedean local field and $\rho$ is continuous under the valuation of $k.$ In this case $k_0$ is local too.
We write $(k,\nu)$-datum instead of $k$-datum to emphasize the valuation $\nu$ on $k.$ 
\end{definition}


\section{Forms of classical groups}\label{secFormsOfClassicalGroups}

A good introduction in the theory of classical groups can be found in 
\cite{platonovRapinchuk:94}. We only consider $\Char(k)\neq 2.$ 
Let us fix a natural number $n.$ 

\begin{notation}\idxS{$\bA^n$}\idxS{$\bGL_n$}\idxS{$\bSL_n$}
$\bA^n$ denotes the affine space of dimension $n.$ The groups $\bGL_n$ (resp. $\bSL_n$) are the general linear group (resp. the special linear group). All are considered as affine algebraic group schemes defined over the prime field of $k.$
\end{notation}

\begin{definition}\label{defOrthSymplGroup} We consider the transposition $( )^T$\idxS{$( )^T$} on $\bGL_n(\bar{k}).$ We denote by $\bO_n$\idxS{$\bO_n$} the \textit{orthogonal group}\idxD{orthogonal group}, i.e the subscheme of $\bGL_n$ which is defined by the equation $g^Tg=1,$  and by $\bSp_{2n}$ the \textit{symplectic group}\idxS{$\bSp_{2n}$}, i.e. the subscheme of $\bGL_{2n}$ given by the equation
\[g^TJg=J,\] where $J$ is 
\beq
\left(\begin{array}{cc}
0 & M \\
-M & 0 
\end{array}\right)
\eeq
and $M:=\antidiag(1,\ldots,1).$ 
The \textit{special orthogonal group}\idxD{special orthogonal group} $\bSO_n$\idxS{$\bSO_n$} is the intersection of $\bO_n$ with 
$\bSL_n.$
\end{definition}

In this section we use the notion of a $k$-form and therefore we give a general definition here.

\begin{definition}
An algebraic group defined over $k$ is called a \textit{$k$-group}\idxD{$k$-group}. Two $k$-groups 
are \textit{$k$-isomorphic to each other}\/\idxD{$k$-isomorphic} if there is an isomorphism of algebraic groups defined over $k$ between them. Let $L|k$ be a field extension in $\bar{k}|k$ and let $G$ be an $L$-group. A $k$-group $H$ is a {$k$-form of $G$} if $H$ and $G$ are $\bar{k}$-isomorphic. A $k$-form $H$ of $G$ is an {$L|k$-form of $G$} if $G$ and $H$ are $L$-isomorphic to each other.
\end{definition}

\begin{convention}
Instead of ``isomorphic as algebraic groups'' we only write ``isomorphic''.
\end{convention}

\begin{definition}
A \textit{classical group in the strict sense}\idxD{classical group in the strict sense} is an algebraic group which is $\bar{k}$-isomorphic to $\bSL_n(\bar{k}),\ \bSO_n(\bar{k})$ or $\bSp_{2n}(\bar{k}).$  
\end{definition}

\begin{proposition}\label{propSymAndSkewsymForms}
Let $V$ be an $n$-dimensional $\bar{k}$-vector space equipped with a non-de\-ge\-ne\-rate symmetric or alternate bilinear form $h.$ 
We assume that $V$ has a $k$-structure $V_k,$ i.e. $\tens{V_k}{k}{\bar{k}}=V,$ and that
\[\sigma_h=\tens{\sigma}{k}{\bar{k}}\] 
for an involution $\sigma$ of $\End_k(V_k)$ of the first kind. Then the following holds.
\be
\item If $h$ is alternate then the group $\U(h)$ equals $\SU(h)$ and is a $k|k$-form of $\bSp_{n}(\bar{k}).$
\item If $h$ is symmetric then the group $\SU(h)$ (resp. $\U(h)$) is a  $k^{sep}|k$-form of $\bSO_n(\bar{k})$ (resp. $\bO_n(\bar{k})$).
\ee
\end{proposition}

\begin{remark}
The set $\End_{\bar{k}}(V)$ is made to an affine space defined over $k$ in taking a $k$-basis of $\End_k(V_k)$ and introducing coordinates, i.e. we have
\[\End_{\bar{k}}(V)\cong\bA^{n^2}(\bar{k})\]
Every $k$-linear isomorphism of $\End_k(V_k)$ induces a $\bar{k}$-linear isomorphism of $\bA^{n^2}(\bar{k})$ defined over $k.$
The composition of maps in $\End_{\bar{k}}(V)$ coinsides with a $k$-morphism \[\phi:\ \bA^{n^2}(\bar{k})\times\bA^{n^2}(\bar{k})\ra\bA^{n^2}(\bar{k})\] 
and $(\bA^{n^2}(\bar{k}),\phi)$ is $k$-isomorphic to 
$(\Matr_n(\bar{k}),\circ).$
We identify the groups $\U(h)$ and $\SU(h)$ of the proposition with the corresponding subsets of $\bA^{n^2}(\bar{k}).$
The assertions of the proposition are valid for any choice of the basis of $\End_k(V_k).$ 
\end{remark}

\begin{proof} (of proposition \ref{propSymAndSkewsymForms})
By proposition \ref{propHermFInv} there is a $\lambda\in\bar{k}$ such that $\lambda h$ maps $V_k\times V_k$ to $k.$ Without loss of generality we assume that $\lambda$ is one. 
\be
\item By theorem \ref{thmWitt} there is a $k$-basis of $V_k$ such that the Gram matrix $G$ of $h$ is $J$ from definition \ref{defOrthSymplGroup}. Thus all elements of $\U(h)$ have determinant $1$ and there is a $k$-isomorphism from $\bA^{n^2}(\bar{k})$ to $\Matr_n(\bar{k})$ which maps $\U(h)$  onto $\bSp_n(\bar{k}).$ Thus $\U(h)$ and $\SU(h)$ equal and are $k|k$-forms of $\bSp_n(\bar{k}).$
\item By theorem \ref{thmOrthogonalSum} there is an orthogonal basis of $V_k$ with respect to $h|_{V_k\times V_k},$ i.e. the corresponding Gram matrix $G$ of $h$ is diagonal and has entries in $k.$ It implies that there is a $k$-isomorphism from $\bA^{n^2}(\bar{k})$ to $\Matr_n(\bar{k})$
which maps $\U(h)$ (resp. $\SU(h)$) to $\U(\tilde{h})$ (resp. $\SU(\tilde{h})$) where 
\[\tilde{h}:\bar{k}^n\times \bar{k}^n\ra\bar{k}\]
is a bilinear form whose Gram matrix under the standard basis of $\bar{k}^n$ is $G.$ The characteristic of $k$ is not $2$ and thus the roots 
of the diagonal elements of 
$G$ are separable over $k,$ i.e. there is a diagonal matrix $X$ in $\Matr_n(k^{sep})$ such that $XX=G.$ 
The inner $\bar{k}$-algebra automorphism $\Inn(X)$ maps $\U(\tilde{h})$ (resp. $\SU(\tilde{h})$) onto $\bO_n(\bar{k})$ (resp. $\bSO_n(\bar{k})$).
 This map is a $k^{sep}$-isomorphism. $\bO_n$ and $\bSO_n$ are defined over the prime field and therefore defined over $k^{sep}.$ Thus $\U(h)$ 
and $\SU(h)$ are defined over $k^{sep}.$ Both groups are $k$-closed too, and since $k^{sep}|k$ 
is a separable field extension they are defined over $k.$
\ee
\end{proof}

\bass
We fix a central simple $k$-datum $(A,V,m,D,d,L|k)$ and we assume $L$ to be a subfield $\bar{k}.$
\eass

We now give forms of $\bSL_n(\bar{k}),$ $\bSO_n(\bar{k})$ and $\bSp_n(\bar{k}).$
We also write $\GL_D(V)$ for $\Aut_D(V)$ and $\SL_D(V)$ for the set of elements of $\End_D(V)$ with reduced norm one.

\begin{proposition}\label{propGLmSLm}
The group $\GL_D(V)$ (resp. $\SL_D(V)$) is the set of $k$-rational points of an $L|k$-form of 
$\bGL_{md}(\bar{k})$ (resp. $\bSL_{md}(\bar{k})$) denoted by $\bGL_D(V)$\idxS{$\bGL_D(V)$}(resp. $\bSL_D(V)$\idxS{$\bSL_D(V)$}). 
\end{proposition}

\begin{proof}
We fix a $k$-basis of $\End_D(V)$ to introduce coordinates. We identify the tensor product \mbox{$\tens{\End_D(V)}{k}{\bar{k}}$} with $\bA^{m^2d^2}(\bar{k}),$ i.e. the set of $k$-rational points of $\bA^{m^2d^2}$ is identified with $\End_D(V).$ $\tens{\End_D(V)}{k}{\bar{k}}$ is $L$-isomorphic to $\Matr_{md}(\bar{k})$ as a $\bar{k}$-algebra, because $L$ is a splitting field of $D.$ The reduced norm corresponds to the determinant on $\Matr_{md}(\bar{k})$ and thus there is an $L$-isomorphism 
from $\bA^{m^2d^2+1}(\bar{k})$ to $\Matr_{md}(\bar{k})\times \bar{k}$ which maps 
\[\bGL_D(V):=\{(g,y)\in (\tens{\End_D(V)}{k}{\bar{k}})\times\bar{k}|\ \Nrd(g)y=1\}\]
onto $\bGL_{md}(\bar{k})$ and 
\[\bSL_D(V):=\{g\in\bGL_DV|\ \Nrd(g)=1\}\]
onto $\bSL_{md}(\bar{k}).$ 

The reduced norm is a homogeneous polynomial in $k[X_1,\ldots,X_{(md)^2}]$ and therefore $\bGL_D(V)$ and $\bSL_D(V)$ are $k$-closed. $\bGL_{md}$ and $\bSL_{md}$ are defined over the prime field and therefore defined over $L.$ Thus $\bGL_D(V)$ and $\bSL_D(V)$ are defined over $L.$ The separability of $L|k$ implies that both groups are defined over $k.$ 
\end{proof}

\begin{remark}
If $F$ is an intermediate field between $k$ and $\bar{k}$ then the set of $F$-rational points of 
$\bGL_D(V)$ is given by $(\tens{\End_D(V)}{k}{F})^\times.$
\end{remark}

\bass
We now extend our assumption and fix a hermitian $k$-datum
\[((A,V,m,D,d,L|k),\rho,k_0,h,\epsilon,\sigma).\]
\eass

\begin{proposition}\cite[2.15]{platonovRapinchuk:94}\label{propHermForms}
There is an algebraic group $\bU(h)$\idxS{$\bU(h)$} (resp. $\bSU(h)$)\idxS{$\bSU(h)$} 
which is 
\be
\item an $L|k$-form of $\bSp_{md}(\bar{k})$ if $\sigma$ is symplectic,
\item a $k^{sep}|k$-form of $\bO_{md}(\bar{k})$ (resp. $\bSO_{md}(\bar{k})$) if $\sigma$ is orthogonal and
\item an $L|k_0$-form of $\bGL_{md}(\bar{k})$ (resp. $\bSL_{md}(\bar{k})$) if $\sigma$ is unitary
\ee
such that its set of $k$-rational points is $\U(h)$ (resp. $\SU(h)$). 
\end{proposition}

\begin{definition}\label{defbUsigmah}
We also denote $\bU(h)$ and $\bSU(h)$ by $\bU(\sigma_h)$
 and $\bSU(\sigma_h)$ respectively.
\end{definition}

For the orthogonal and the unitary case in the proposition we needed that the characteristic of $k$ is different from two. In the unitary case $k|k_0$ has degree two. 

\begin{proof}
At first we assume that $\sigma$ is of the first kind and we define
\[\bU(h):=\{g\in\bGL_DV|\ g^{(\tens{\sigma}{k}{\id})}g=1\}\]
and
\[\bSU(h):=\{g\in\bU(h)|\ \Nrd(g)=1\}.\]
The assertions (1) and (2) follow now from proposition \ref{propSymAndSkewsymForms},
because $\tens{\End_D(V)}{k}{\bar{k}}$ is $L$-isomorphic to $\Matr_{md}(\bar{k}).$

At second we assume that $\sigma$ is of the second kind. Here we use the notion of the Weil-restriction, see appendix chapter \ref{appWeilRestriction}.
We have 
\[\Res_{k|k_0}(\tens{\End_D(V)}{k}{\bar{k}})=\tens{\End_D(V)}{k_0}{\bar{k}}\]
and a commutative diagram
\[\begin{matrix}
\tens{\End_D(V)}{k_0}{\bar{k}} & \cong & (\tens{\End_D(V)}{k}{\bar{k}})\times (\tens{\End_D(V)}{k}{\bar{k}})^{\rho}\\
\cup & & \cup\\
\Res_{k|k_0}(\bA^1)(\bar{k})& \cong & \bA^1(\bar{k})\times\bA^1(\bar{k})\\
\end{matrix}
\]
where the $k$-isomorphism on the top is induced by \[\tens{\lambda}{k_0}{\mu}\in\tens{k}{k_0}{\bar{k}}\mapsto(\lambda\mu,\rho(\lambda)\mu).\]
We now explain $(\tens{\End_D(V)}{k}{\bar{k}})^{\rho}.$
Under the choice of a $k$-basis in $\End_D(V)$ the composition of endomorphisms is given by a morphism 
\[\phi:\ \bA^{(md)^2}\times\bA^{(md)^2}\ra\bA^{(md)^2}\] defined over $k.$
The multiplication on $(\tens{\End_D(V)}{k}{\bar{k}})^{\rho}$ is given by 
$\phi^{\rho}.$ 
The involution $\tens{\sigma}{k_0}{\id}$ defines an involution 
$\tilde{\sigma}$ on the right side of the diagram 
which permutes the coordinates of $\bar{k}\times\bar{k}.$
Thus there is an $L$-isomorphism from $(\tens{\End_D(V)}{k}{\bar{k}})\times (\tens{\End_D(V)}{k}{\bar{k}})^{\rho}$ to 
$\Matr_{md}(\bar{k})\times\Matr_{md}(\bar{k})$
such that the involution $\tilde{\sigma}$ corresponds to the involution
\[\sigma':(B_1,B_2)\mapsto (B_2^T,B_1^T).\]
The group $\U(\sigma')$ is a $k|k$-form of $\bGL_{md}$ and thus 
\[\bU(h):=\{g\in\Res_{k|k_0}(\bGL_DV)|\ g^{\Res_{k|k_0}(\tens{\sigma}{k}{\id})}g=1\}\]  
is an $L|k_0$-form of $\bGL_D(V).$ 
The Weil-restriction of the reduced norm from $k$ to $k_0$ 
corresponds to $\det\times\det$ on $\Matr_{md}(\bar{k})\times\Matr_{md}(\bar{k})$ and we conclude that 
under \[\bU(h)\cong\bGL_{md}\]
the group 
\[\bSU(h):=\{g\in\bU(h)|\ \Res_{k|k_0}(\Nrd)(g)=1\}\]
is mapped onto $\bSL_{md}(\bar{k}).$ 
The groups fulfill
\[\bU(h)(k_0)=\U(h) \text{ and }\bSU(h)(k_0)=\SU(h)\]
by definition.
\end{proof}

\begin{remark}
For a version of a converse of this theorem see \cite[26.9, 26.12, 26.14, 26.15]{knus:98}.
\end{remark}

Later we only consider non-Archimedean local fields and there
we have the following restricted possibility for the involution.

\begin{theorem}\cite[10.2.2]{scharlau:85}\label{thmDScharlau}
If $k$ is a non-Archimedean local field and $\rho$ is an involution of $D$ then if $\rho$ is of the second kind we have $D=k$ and if $\rho$ is of the first kind, the degree of $D$ is not bigger than two.
\end{theorem}

\chapter{Bruhat-Tits building of a Classical group}\label{chapBruhatTitsBuildingOfAClassicalGroup}
For an introduction to the theory of buildings, I recommand \cite{abramenkoBrown:08} or the previous version \cite{brown:89}.

\bass\label{asstiG}
In this chapter we fix a local simple $k$-datum
\[(A,V,m,D,d,L|k),\] especially we assume that $k$ is a non-Archimedean local field with valuation $\nu.$ 
The unique extension of $\nu$ to a finite dimesional skewfield over $k$ is also denoted by $\nu.$ 
We denote $\bGL_D(V)$ by $\tibG$\idxS{$\tibG$} and its set of $k$-rational points by $\tiG,$\idxS{$\tiG$} i.e.
$\tiG:=\GL_D(V).$
\eass


\section{Norms and lattice functions}
This is a collection of definitions and results of \cite{broussousLemaire:02} and \cite{bruhatTitsIII:84}.


\subsection{First definitions}

\begin{definition}
A \textit{$D$-norm}\idxD{norm} on $V$ is a map $\alpha:\ V\ra \Reinfty$ such that for all $t\in D$ and $v,v'\in V$
we have
\be
\item $\alpha(v)=\infty\ \Lra\ v=0,$
\item $\alpha(tv)=\alpha(v)+\nu(t)$ and
\item $\alpha(v+v')\geq \min(\alpha(v),\alpha(v')).$
\ee
The set of $D$-norms on $V$ is denoted by $\Normone{D}{V}.$\idxS{$\Normone{D}{V}$}
\end{definition}

Given a norm the family of balls around $0$ is a decreasing function of lattices. We recall the definition below.

\begin{definition}\label{defODLattice}
A finitely generated $o_{D}$-submodule $\Gamma$ of $V$ is called a \textit{(full) $o_{D}$-lattice of} $V$ if $\Span_{D}(\Gamma)=V.$ 
We denote by $\Gitter{o_{D}}{V}$\idxS{$\Gitter{o_{D}}{V}$} the set of all full $o_D$-lattices of $V.$
\end{definition}

\begin{definition}\cite[2.1]{broussousLemaire:02} \label{defLatticeFunction}
A map $\Lambda :\ \bbR\ra \Gitter{o_D}{V}$ is called an \textit{$o_D$-lattice function,}\idxD{lattice function} if for all 
reals $r$ and $s$ with $r\leq s$ we have
\be
\item $\Lambda (r+\nu(\pi_D))$ equals $\Lambda (r)\pi_D,$
\item $\Lambda (r)$ contains $\Lambda (s)$ and
\item the lattice $\Lambda(r)$ is the intersection of the $\Lambda(r-\epsilon)$ where $\epsilon$ runs over all positive real numbers, i.e. $\Lambda$ is left continuous, when $\Gitter{o_D}{V}$ is endowed with the discrete topology.
\ee 
The set of $o_D$-lattice functions is denoted by $\Lattone{o_D}{V}.$\idxS{$\Lattone{o_D}{V}$}
For $\Lambda\in\Lattone{o_{D}}{V}$ the number of elements in $\Lambda([0,\nu(\pi_D)[)$ is the \textit{(simplicial) rank} \idxD{rank of a lattice function} of $\Lambda.$
\end{definition}

\begin{remark}\cite[I.2.4]{broussousLemaire:02}\label{remBLI.2.4}
The map $\alpha\mapsto\Lambda_{\alpha}$ with 
\[\Lambda_{\alpha}(r):=\{x\in V|\ \alpha(x)\geq r\}\]\idxS{$\Lambda_{\alpha}$}
is a $\tilde{G}$-set isomorphism from $\Normone{D}{V}$ to $\Lattone{o_D}{V}$
with actions
\[g.\alpha:=\alpha\circ g^{-1}\text{ and } (g.\Lambda)(r):=g(\Lambda(r)).\]
\end{remark}

In certain proofs it is useful to reduce the $m$-dimensional case to lower dimensional cases. This is done with the concept of a splitting vector space decomposition.

\begin{definition}\cite[1.4]{bruhatTitsIII:84}
A family $(V^1,\ldots,V^l)$ of $D$-vector subspaces is a 
\textit{splitting decomposition of $V$} \idxD{splitting decomposition for a norm} for $\alpha\in \Normone{D}{V},$ if 
\bi
\item $V=\oplus_iV^i$ and
\item for all $(v_i)\in\prod V^i$ we have 
\[\alpha(\sum_iv_i)=\min_i\alpha(v_i).\]
\ei
A family $(V^i)$ is a \textit{splitting decomposition of $V$} for\idxD{splitting decomposition for a \\ lattice function} an $o_D$-lattice function $\Lambda$ on $V$ if 
\[\Lambda(t)=\oplus_i(\Lambda(t)\cap V^i)\] holds for all $t.$\/
We also say that the norm or the lattice function is split
by $(V^i)_i.$
If all $V^i$ are one dimensional and $(b_i)\in\prod_iV^i$
is a $D$-basis of $V$ we call it a \textit{splitting basis}\idxD{splitting basis}
for norms and lattice functions which are split by $(V^i).$
\end{definition}

\begin{definition}\cite[1.11 (17)]{bruhatTitsIII:84}\label{defDualNormVStar}
The \textit{dual of $\alpha$}\idxD{dual of a norm}\/ is the $o_D$-norm $\alpha^*$ on 
\[V^*:=\Hom_D(V,D)\]\idxS{$V^*$} defined by
\[\alpha^*(f):=\inf\{\nu(f(v))-\alpha(v)|\ v\in V\setminus \{0\}\}.\]\idxS{$\alpha^*$}
\end{definition}

\begin{proposition}\cite[1.11 (18), 1.26]{bruhatTitsIII:84}\label{propDualNorm}
\be
\item The dual basis of a splitting basis $(v_i)_i$ of $\alpha$ is
a splitting basis of $\alpha^*$ and the equation $\alpha^*(v^*_i)=-v_i$ holds.
\item Any two $o_D$-norms on $V$ have a common splitting basis.
\ee
\end{proposition}

The following definition generalises the definition
\cite[1.1.1]{bruhatTitsI:72}.

\begin{definition}\label{defAffineStructure}{(Affine structure)}
An \textit{affine structure}\idxD{affine structure} on a set $S$ is a function
\[a:S\times S\times [0,1]\ra S.\] We write
\[ts+(1-t)s':=a(s,s',t).\]
\end{definition}

\begin{definition}
For a real number $x$ we denote by $[x]+$ the smallest integer which is not smaller than $x.$ 
\end{definition}

\begin{remark}\label{remAffineStructure}
We have an affine structure on $\Lattone{o_D}{V}.$
If $\Lambda$ and $\Lambda'$ are two elements of $\Lattone{o_D}{V}$ with a common splitting basis $(v_i),$ i.e. there are $m$-tupels $(\alpha_i)$ and $(\beta_i)$ of real numbers such that
\[\Lambda(x)=\bigoplus_i v_i\tf{p}_D^{[(x-\alpha_i)d]+}\]
and
\[\Lambda'(x)=\bigoplus_i v_i\tf{p}_D^{[(x-\beta_i)d]+},\]
then for $\lambda\in [0,1]$ we define a new element of 
$\Lattone{o_D}{V}$ by 
\[(\lambda\Lambda+(1-\lambda)\Lambda')(x):
=\bigoplus_iv_i\tf{p}_D^{[(x-\lambda\alpha_i-(1-\lambda)\beta_i)d]+}.\] 
This definition does not depend on the choice of the basis $(v_i).$
\end{remark}

\begin{remark}\label{remDSOLF}
Let $V'$ be another finite dimensional right $D$-vector space.
The map
\[\Lattone{o_D}{V}\times\Lattone{o_D}{V'}\ra\Lattone{o_D}{V\oplus V'} \]
given by 
\[(\Lambda,\Lambda')\mapsto \Lambda\oplus\Lambda'\]
with 
\[(\Lambda\oplus\Lambda')(t):=\Lambda(t)\oplus\Lambda'(t)\]
is affine and $\tilde{G}\times\GL_D(V')$-equivariant. 
\end{remark}

\begin{definition}\label{defDSOLF}
The lattice function $\Lambda\oplus\Lambda'$ is called the 
\textit{direct sum}\idxD{direct sum of lattice functions} of $\Lambda$ and $\Lambda'$
\end{definition}  

\begin{remark}
Under $\alpha\mapsto\Lambda_\alpha$ the affine structure of $\Lattone{o_{D}}{V}$ defines the following affine structure on 
$\Normone{D}{V}.$ For $\lambda\in [0,1],\ \alpha,\ \alpha'\in\Normone{o_D}{V}$
and a common splitting basis 
$(v_1,\ldots,v_n)$ we have that 
\[(\lambda\alpha+(1-\lambda)\alpha')\]
is the norm with splitting basis $(v_i)_i$
such that
\[v_i\mapsto \lambda\alpha(v_i)+(1-\lambda)\alpha'(v_i).\]
\end{remark}


\subsection{Square lattice functions}\label{subsecSquarelatticeFct}

\begin{assumption}
Let us now assume that $k$ is the center of $D.$
\end{assumption}

We also have $k$-norms and $k$-lattice functions on $A.$ We recall that for an $o_D$-lattice function $\Lambda$ (resp. $D$-norm $\alpha$) on $V$ the $o_k$-lattice function 
\[r\mapsto \End(\Lambda)(r):=\{a\in A |\ a(\Lambda(s))\subseteq\Lambda(s+r)\ \forall s\in\mathbb{R}\}\]\idxS{$\End(\Lambda)$}
(resp. $k$-norm
\[a\mapsto\End(\alpha)(a):=\inf\{\alpha(a(x))-\alpha(x)|\ x\in V\setminus \{0\}\})\]\idxS{$\End(\alpha)$}
on $A$ is called \textit{square lattice function}\idxD{square lattice function} (resp. \textit{square norm}\idxD{square norm}) on $A.$ The set of square lattice functions, square norms on $A$ is denoted by $\Latttwo{o_k}{A},\ \Normtwo{k}{A}$\idxS{$\Latttwo{o_k}{A}$}\idxS{$\Normtwo{k}{A}$}
respectively. There is a $\tilde{G}$-actions on $\Normtwo{k}{A}$, $\Latttwo{o_k}{A}$ which is given by
\[g\beta:=\beta\circ \Inn(g^{-1}),\ (g\Gamma)(r) :=\Inn(g)(\Gamma(r))\]
respectively where $\Inn$ denotes the adjoint action of $\tilde{G}$ on $A.$ 

\begin{remark}{\cite[4.10]{broussousLemaire:02}}
The map $\End(\alpha)\mapsto \End(\Lambda_{\alpha})$ is a $\tilde{G}$-set isomorphism from $\Normtwo{k}{A}$ to $\Latttwo{o_k}{A}.$
\end{remark}

A square lattice functions encodes the $o_D$-lattice function up to translation.

\begin{definition}\label{defLattoDV}
The \textit{translation of an $o_D$-lattice function $\Lambda$ }\idxD{translation of a lattice function} 
by a real number $s$ is defined as
\[(\Lambda+s)(t):=\Lambda(t-s).\]
Two $o_D$-lattice functions $\Lambda$ and $\Lambda'$ are \textit{equivalent}\idxD{equivalent lattice functions} if $\Lambda$ is a translation of $\Lambda'.$
The set of equivalence classes of $o_D$-lattice functions of 
$V$ is denoted by $\Latt{o_D}{V}.$\idxS{$\Latt{o_D}{V}$} Taking classes in remark \ref{remAffineStructure} one obtains an affine structure for 
$\Latt{o_D}{V},$ i.e. 
\[\lambda[\Lambda]+(1-\lambda)[\Lambda']:=[\lambda\Lambda+(1-\lambda)\Lambda'],\ \lambda\in [0,1].\]
The \textit{translation of an $o_D$-norm $\alpha$} of $V$ by an element $s$ of $\bbR$ is defined as
\[(\alpha+s)(v):=\alpha(v)+s.\]
Two norms are \textit{equivalent} if one is a translation of the other and the set of all equivalence classes is denoted by
$\Norm{D}{V}.$\idxS{$\Norm{o_D}{V}$}
\end{definition}

\begin{theorem}\cite[I.4]{broussousLemaire:02}\label{thmBLI.4}
The following is a commutative diagram of $\tilde{G}$-set isomorphisms.
\begin{equation}\label{eqThmCD}
\begin{array}{ccc}
\Norm{D}{V} & \ra &\Latt{o_D}{V}\\
\downarrow &&\downarrow\\
\Normtwo{k}{A} &\ra &\Latttwo{o_k}{A}\\
\end{array}
\end{equation}
The maps are defined as follows.
\bi
\item on the top: $[\alpha]\mapsto[\Lambda_\alpha],$
\item on the bottom: $\beta\mapsto\Lambda_\beta,$
\item on the left: $[\alpha]\mapsto\End(\alpha),$
\item on the right: $[\Lambda]\mapsto\End(\Lambda).$
\ei
\end{theorem}

For sake of completeness we give a second diagram.

\begin{remark}
The map given in remark \ref{remBLI.2.4}
induces a commutative diagram of $\tilde{G}$-set isomorphisms.
\begin{equation}\label{eqThmCD2}
\begin{array}{ccc}
\Normone{D}{V} & \ra &\Lattone{o_D}{V}\\
\downarrow &&\downarrow\\
\Norm{D}{V} &\ra &\Latt{o_D}{V}\\
\end{array}
\end{equation}
The maps downwards send an element to its equivalence class.
\end{remark}

We give a last remark describing the behavior of 
square lattice functions under direct sum.
We use the assumptions of remark \ref{remDSOLF}.

\begin{definition}
For $a\in A$ and $a'\in \End_D(V')$ the direct sum of
$a$ and $a'$ in $\End_D(V\oplus V')$ is defined by
\[(a\oplus a')(v,v'):=(a(v),a'(v')).\]
\end{definition}

\begin{proposition}\label{propCLFunderDS}
\[\End(\Lambda\oplus\Lambda')(t)\cap(\End_D(V)\oplus\End_D(V'))=
\End(\Lambda)(t)\oplus\End(\Lambda')(t).\]
\end{proposition}


\section{\texorpdfstring{The Bruhat-Tits building of $\bGL_D(V)$ over $k$}{The Bruhat Tits building of GLD(V) over k}} \label{secBTBuildingsGLDV}

We consider the building of the following valuated root datum mentioned in
\cite{bruhatTitsIII:84}. We briefly repeat the construction.
As usual  $X^*(?)_k$\idxS{$X^*(?)_k$} and $X_*(?)_k$\idxS{$X_*(?)_k$} denote the set of $k$-rational characters and cocharacters respectively.

\begin{assumption}
In this section we assume that $k=\Centr(D).$
\end{assumption}

We take a $D$-basis $(v_i)$ of $V$ and consider the maximal $k$-split torus $T$ of $\tibG$ whose set of $k$-rational points is  \[\{t\in \tiG\mid\ tv_i\in kv_i,\text{ for all }i\}.\]
With the basis $\GL_D(V)$ identifies with $\GL_m(D)$ and the $k$-rational points of $T$ are diagonal matrices.
The torus acts on the Lie algebra by conjugation and the $k$-rational roots of $T$ are the characters
\[t\mapsto a_{i,j}(t):=t_it^{-1}_j,\ i,j\in\bbN_m\text{ for } i\neq j.\]
The root system 
\[\Phi:=\{a_{i,j}\mid\ i,j\in\bbN_m,\ i\neq j\}\]
of $\tens{X^*(T/\Centr(T))}{\bbZ}{\bbR}$ is of type $A_{m-1}.$
We denote by $u_{i,j}(x)$ the matrix of the homomorphism
\[v_k\mapsto v_k+v_j\delta_{i,k}x\]
The set of $k$-rational points of $\Centr_{\bGL_D(V)}(T)$ together with the root groups 
\[U_{i,j}:=\{u_{i,j}(x)\mid\ x\in D\},\ i\neq j,\ i,j\in\bbN_m,\] form a valuated root datum using the valuation
\begin{equation}\label{eqValuatedRootDatum}
\phi_{a_{i,j}}(u_{i,j}(x)):=\nu(x).
\end{equation}
For the definition of a valuated root datum see \cite[6.2]{bruhatTitsI:72} or \ref{defValuationOfARootDatum}. For the example see \cite[10]{bruhatTitsI:72}. A short introduction of the steps for the construction of the building of a valuated root datum can be found in the appendix \ref{appChBuildingOfAValuatedRootDatum}.

The vec\-tor space $W:=\tens{X_*(T/\Centr(T))_k}{\bbZ}{\bbR}$ is iden\-ti\-fied with the dual of \\
$\tens{X^*(T/\Centr(T))_k}{\bbZ}{\bbR}$ via the na\-tu\-ral pairing
\[X^*(T/\Centr(T))_k\times X_*(T/\Centr(T))_k\ra \bbZ\]
and we de\-note there\-fore $\tens{X^*(T/\Centr(T))_k}{\bbZ}{\bbR}$ by $W^*.$
The standard apartment is the set $\Delta$ of all valuations of \[(\Centr_{\GL_D(V)}(T(k)),(U_{i,j})_{i,j})\] which are equipollent to $\phi,$ see section \ref{secBuildingOfAValuatedRootDatum}. $\Delta$ is an affine space over $W.$ A vector $w$ of $W$ acts on $A$ by
\[\psi\mapsto (u\in U_a\mapsto \psi_a(u)+a(w)).\]
The group $T(k)$ acts on $\Delta$ by translation via
\[(t,\psi)\mapsto \psi+w(t)\]
where $w(t)\in W$ is defined by 
\[\forall a\in\Phi:\ a(w(t))=\nu(a(t)).\]
The set $N(T)(k)$ of $k$-rational pionts of the normaliser $N(T)$ of $T$ in $\tibG$ precisely consists of the monomial matrixes with entries in $D$ and the above action extends for an element $n\in N(T)(k)$ via
\[ n.(\phi+w)=\phi+n.w \]
where  
\[ a_{\tau(i)\tau(j)}(n.w)=a_{ij}(w)+\nu(n_{\tau(j),j})-\nu(n_{\tau(i),i})\]
and $\tau$ is the involution defined by $n_{\tau(i),i}\neq 0.$
  
\begin{definition}
The building $\building(\tibG,k)$\idxS{$\building(\bGL_D(V),k)$} of the valuated root datum defined in (\ref{eqValuatedRootDatum}) is the set of equivalence classes of $\tilde{G}\times \Delta$ under the relation: 
\[(g,x)\sim (h,y)\] if and only if there exists a monomial matrix $n$ such that 
\[n(x)=y\text{ and } hn\in gP_x.\] 
The set of all apartments is given by the sets of the form $g\Delta,\ g\in \tilde{G}$ using the $\tilde{G}$-action on the first coordinate. 
The definition of $P_{x}$ is given in section \ref{secBuildingOfAValuatedRootDatum} in appendix \ref{appChBuildingOfAValuatedRootDatum}. 
We denote $\building(\tibG,k)$ the \textit{Bruhat-Tits building of }\idxD{Bruhat-Tits building of $\tibG$} $\tibG$ over $k.$
\end{definition}

\begin{remark}\idxS{$\building(\bSL_D(V),k)$}
The Bruhat-Tits building of $\bSL_D(V)$ over $k$ is constructed in the same way and it is canonically identified with the Bruhat-Tits building of $\tibG$ over $k.$
\end{remark}

\begin{remark}
In chapter \ref{chapEnlargedBuildingOfAReductiveGroup} we recall the notion of an enlarged building from \cite[4.2.16]{bruhatTitsII:84}. 
If $\building(G,k)$ is the Bruhat-Tits building of a reductive group over a local field $k$ we denote the enlarged building by 
$\Building(G,k).$\idxS{$\Building(\tibG,k)$}
The group $X^*(\bSL_D(V))_k$ is trivial and $X^*(\tibG)_k$ is isomorphic to $\bbZ.$ Thus 
$\Building(\bSL_D(V),k)$ and $\building(\bSL_D(V),k)$ coincide and $\tibG$ has a proper enlarged building over $k.$
\end{remark}

\brem
The apartments are in one to one correspondence with the maximal $k$-split tori of $\tibG.$
\erem

As described in \cite[2.11]{bruhatTitsIII:84} and \cite{broussousLemaire:02} 
one can associate to a point $x$ of the enlarged apartment $\Delta^1:=\tens{X_*(T)_k}{\bbZ}{\bbR}$ a $D$-norm.
 in the following way.
\beq\label{eqBTBL}
\alpha_x(\sum_id_iv_i):=\inf_i(\nu(d_i)-a_i(x))
\eeq
where $a_i\in X^*(T)$ is the projection to the $i$th coordinate and 
\[a_i(x):=<x,a_i>.\]
The $o_D$-lattice function corresponding to $\alpha_x$ is denoted by $\Lambda_x.$
The maps from $\Delta^1$ to $\Normone{D}{V}$, $\Lattone{o_D}{V}$ resp. induced by (\ref{eqBTBL})
can be extended to the whole enlarged building and if one asks for some further properties then this extension is possible in a unique way. More precisely by Bruhat, Tits, Broussous and Lemaire we have the following theorem. Here we use that the dual $W^1$ of $\tens{X_*(\tibG)}{\bbZ}{\bbR}$ acts on the enlarged building by translations. For further description see \ref{chapEnlargedBuildingOfAReductiveGroup}. One defines an action of $W^1$ on the set of $o_D$-lattice functions by 
\[\Lambda+\lambda nd\check{a}_0:=\Lambda+\lambda \]
where \[\check{a}_0:=\frac{1}{n}\sum_i\check{a}_i\]
and where $(\check{a}_i)$ is the dual basis of $(a_i)$ in $X_*(T).$ 

\begin{theorem}\label{thmBLI1.4}\cite[I.1.4.,I.2.4, II.1.1. for $F=E$]{broussousLemaire:02}
There is a unique $A^{\times}$- and $W^1$- equivariant affine bijection 
\[\Building(\tibG,k)\ra \Lattone{o_D}{V}\]
extending the map
\[x\in \Delta^1\mapsto \Lambda_x.\]
This bijection induces the unique $\tilde{G}$-equivariant and affine map from $\building(\tibG,k)$ to the set of lattice function classes $\Latt{o_D}{V}.$ It is bijective and an extension of 
\[x\in \Delta\mapsto [\Lambda_x].\]
\end{theorem}

\begin{definition}\label{defLieAlgebraFiltrationGLDV}
The set of $k$-rational points of the Lie algebra of $\tibG$ is $A.$
For a point $x\in\Building(\tibG,k)$ we denote the square lattice function corresponding to $x$ by $\LF(x,\tibG,k).$\idxS{$\LF(x,\bGL_D(V),k)$}
This sequence is called the \textit{Lie algebra filtration}\idxD{Lie algebra filtration} of 
$x$ in $A.$
\end{definition}

The last proposition of this section is not used in this part of the thesis, but in the 
next part. We shortly explain the simplicial structure of $\building(\bGL_D(V),k).$
For this we explain the structure for $\Delta$ and apply the action of 
$\GL_D(V).$ The hyperplanes of $\Delta$ given by the equations
\[ a_{i,j}=\frac{k}{d}\]
for $i,j\in\bbN_m$ with $i\neq j$ and $k\in\bbZ$ cut out a cell decomposition of $\Delta,$ see for example \cite[1.3.3]{bruhatTitsI:72}
for the simplicial structure given by an affine root system or \cite[VI.1.B]{brown:89} or \cite[12.1]{garrett:97}. In the next we consider the last map of the above theorem, i.e.
the correspondence with $\Latt{o_D}{V}.$ The ideas of the following proposition are taken form \cite[2.16]{bruhatTitsIII:84}. 

\begin{proposition}\cite{bruhatTitsIII:84}\label{propBTII2.16}
\be
\item
The apartment $\Delta$ is mapped to the set of classes of $o_D$-lattice functions which are split by $(v_i).$
\item
An element $x$ of $\building(\bGL_D(V),k)$ lies on a face of rank $k$ if and only if $\Lambda_x$ has rank $k.$ (We only consider faces which are open in their affine span, i.e. we consider cells.)
\ee
\end{proposition}

\begin{proof}
\be
\item This follows from the definition.
\item Because of the $\SL_D(V)$-equivariance we only have to consider the subset of $\Delta$ given by the inequalities 
\[\alpha_{i+1,i}(x)\geq 0\text{ for }i\in\bbN_{m-1}\text{ and }\alpha_{m,1}(x)\leq \frac{1}{d}.\]
It is the closure of the chamber $C$ which is defined by the strict inequalities. The image of $\bar{C}$ is the set 
\[\{[\Lambda_x]\mid\ x\in\Delta^1,\ 0\leq a_1(x)\leq a_2(x)\leq\ldots\leq a_m(x)= \frac{1}{d}\}.\]
The proof now  is an easy counting of jumps in the sequence \[(a_1(x),a_2(x),a_3(x),\ldots,a_{m-1}(x),\frac{1}{d}).\]
\ee
\end{proof}


\section{Self-dual lattice functions}\label{secSelfDualLatticeFunctions}
This is a collection of results of \cite{broussousStevens:09} and \cite{bruhatTitsIV:87} and we slightly generalise the definition of self-dual objects and propositions of \cite{broussousStevens:09}.
For this section we make the following assumption.

\begin{assumption}\label{assSDLF}
We fix a local hermitian $(k,\nu)$-datum
\[((A,V,m,D,d,L|k),\rho,k_0,h,\epsilon,\sigma)\] and we assume $k$ to have residue characteristic not two
(see definition \ref{defLocalHermkDatum}).
\end{assumption}

\subsection{Duality}

We explain how $h$ defines a map of order two on all spaces of lattice functions and norms which we have considered. 
At first we define this map for $\Lattone{o_D}{V}$ and then for all sets of the diagrams 
(\ref{eqThmCD}) and (\ref{eqThmCD2}).

\begin{definition}
For a lattice function $\Lambda\in \Lattone{o_D}{V}$ and a real number $r$ we define
\[\Lambda(r+):=\cup_{s>r}\Lambda(s). \]
\end{definition}

\begin{definition}[\cite{broussousStevens:09} after Prop. 3.2]
Given a lattice $M\in \Gitter{o_{D}}{V}$ and a lattice function $\Lambda\in \Lattone{o_D}{V}$ the duals are defined by
\[M^{\#}:=\{x\in V|\ h(x,M)\subseteq\vi{D} \}\]\idxS{$(\ )^{\#}$} and
\[\Lambda^{\#}(r):=[\Lambda((-r)+)]^{\#}.\] \idxS{$\Lambda^{\#}$}
\end{definition}

\begin{remark}\label{remEndoDualLatt}
A $D$-endomorphism $a$ of $V$ behaves under the dualisation of a lattice in the following way:
\[(a^{\sigma}(M))^{\#}=a^{-1}(M^{\#}).\]
\end{remark}

\begin{proposition}\label{propDualLatt}
\be
\item For all $M\in \Gitter{o_D}{V}$ the set $M^{\#}$ is a full $o_D$-lattice and $(M^{\#})^{\#}=M.$
\item For all $\Lambda\in \Lattone{o_D}{V}$ we have $\Lambda^{\#}\in \Lattone{o_D}{V}$ and $(\Lambda^{\#})^{\#}=\Lambda.$
\ee
\end{proposition}

\begin{proof}
\be
\item Let $(v_i)_i$ be a Witt basis of $h$ (see corollary \ref{corExistOfWittBasis}), and let $f$ be a $D$-linear automorphism of $V$ which maps 
\[M':=\bigoplus_{i=1}^mv_i\vr{D}\] 
onto $M$. Assertion 1 is true for $M'$ because the following two equations hold 
\[(M')^{\#}=\bigoplus_{i=1}^mv_i\vi{D} \text{ and }((M')^{\#})^{\#}=(M'\tf{p}_D)^{\#}=(M')^{\#}\vipower{D}{-1}=M'.\]
Further we have 
\[f^{\sigma}(M^{\#})=(M')^{\#}.\]
Therefore $M^{\#}$ is a full lattice and 
\[(M^{\#})^{\#}\ =\ (f^{\sigma})^{\sigma}(((M')^{\#})^{\#})\ =\ f(M')\ =\ M\]
as required.
\item For the first assertion we only show (3) of definition \ref{defLatticeFunction}. 
\begin{align*}
\cap_{\epsilon>0}\Lambda^{\#}(r-\epsilon) & = \{v\in V|\ h(v,\cup_{\epsilon>0}\Lambda((-r+\epsilon)+))\subseteq\vi{D}\}\\ & = \{v\in V|\ h(v,\Lambda(-r)+)\subseteq\vi{D}\} \\ & = \Lambda^{\#}(r).\\
\end{align*}
The second assertion is true in pairs $r$, $-r$ of continuity points of $\Lambda$ because of 
\[(\Lambda^{\#})^{\#}(r)=(\Lambda(r)^{\#})^{\#}=\Lambda(r).\]
The density of the set of these $r$ in $\mathbb{R}$ and the left continuity of $\Lambda$ and $(\Lambda^{\#})^{\#}$ extend the equality to all real nunbers. 
\ee
\end{proof}
\vspace{1em}

Before we transfer $( )^\#$ to other spaces, we introduce 
the analogous definition for the dual of a norm. This was introduced by Bruhat and Tits.

\begin{definition}
The \textit{dual}\idxD{dual of a norm with respect to an hermitian form} of a $D$-norm $\alpha$ on $V$ with respect to $h$
is the $D$-norm given by
\[\bar{\alpha}(v):=\inf_{w\in V,w\neq 0}(\nu(h(v,w))-\alpha(w)).\]\idxS{$\bar{\alpha}$}
We skip the notion "with respect to $h$" after the following lemma because definition \ref{defDualNormVStar} is not used after the proof.  
\end{definition}

\begin{lemma}\cite[2.5]{bruhatTitsIV:87}\label{lemBarAlpha}
The dual of a norm with respect to $h$ is well defined and if $(v_i)$ is a splitting basis of $\alpha$ and 
$(w_i)$ is a $D$-basis of $V$ such that 
\[h(w_i,v_j)=\delta_{i,j}\]
then $(w_i)$ is a splitting basis of $\bar{\alpha},$
and the value of $\bar{\alpha}$ in $w_i$ is $-\alpha(v_i).$
\end{lemma}

A basis $(w_i)$ as in the above lemma exists, because 
$h$ is non-degenerate.

\begin{proof}
Under 
\[(\hat{h})^\#:\Norm{D}{V^*}\ra\Abb(V,\bbR),\]
\[\beta\mapsto\beta\circ\hat{h},\]
the image of $\alpha^*$ is $\bar{\alpha}.$ 
See proposition \ref{propHermFInv} for the definition of 
$\hat{h}.$ Therefore $\bar{\alpha}$ is an $o_D$-norm of $V.$
We now prove that $(w_i)$ splits $\bar{\alpha}.$ The norm $\alpha^*$ has $(v_i^*)$ as a splitting basis which is the image of $(w_i)$ under $\hat{h}.$ Thus $(w_i)$ is a splitting basis of 
$\bar{\alpha},$ and
\[\bar{\alpha}(w_i)=\alpha^*(v_i)=-\alpha(v_i).\]
\end{proof}

\begin{proposition}\label{propMapsInducedByInv}
Under the diagrams (\ref{eqThmCD}) and (\ref{eqThmCD2}) the map $()^\#$ corresponds to the following maps:
\be
\item on $\Normone{D}{V}:$ $\alpha\mapsto\bar{\alpha}$
\item on $\Norm{D}{V}:$ $[\alpha]^{\sigma}:=[\bar{\alpha}],$\idxS{$[\alpha]^{\sigma}$}
\item on $\Latt{o_D}{V}:$ $[\Lambda]^{\sigma}:=[\Lambda^\#].$\idxS{$[\Lambda]^{\sigma}$}
\item on $\Latttwo{o_k}{A}:$ $\tf{a}^{\sigma}(t):=(\tf{a}(t))^\sigma,$\idxS{$\tf{a}^{\sigma}$}
\item on $\Normtwo{k}{A}:$ $\beta^{\sigma}(f):=\beta(f^\sigma),$
\ee 
\end{proposition}

\begin{proof}
1. The proof is similar to \cite[3.3.]{broussousStevens:09}.

2. and 3. The maps are well defined which follows immediately from the definitions.

4. We prove for all $\Lambda\in \Lattone{o_D}{V}:$
\[\End(\Lambda^\#)= \End(\Lambda)^\sigma.\]
For an element $a$ of $A$ the following statements are equivalent.
\bi
\item $a^\sigma\in \End(\Lambda^\#)(r)$
\item The lattice $a^\sigma([\Lambda((-s)+)]^\#)$ is contained in $[\Lambda((-s-r)+)]^\#$ for all real numbers $s$. 
\item The lattice $a\Lambda((-s-r)+)$ is contained in $\Lambda((-s)+)$ for all real numbers $s$ by remark \ref{remEndoDualLatt}.
\item $a$ is an element of $\End(\Lambda)(r).$
\ei

5. The bijection 
\[\Normone{k}{A}\cong\Lattone{o_k}{A}\]
maps the norm $\End(\alpha)\circ\sigma$ to $\End(\Lambda_\alpha)^\sigma$ which is $\End(\Lambda_\alpha^\#)$ by assertion 4.
The commutativity of diagram (\ref{eqThmCD}) implies that 
$\End(\bar{\alpha})$ is mapped to $\End(\Lambda_{\bar{\alpha}})$ which is $\End(\Lambda_\alpha^\#)$ by assertion 1.
The equality 
\[\End(\alpha)\circ\sigma=\End(\bar{\alpha})\]
follows now from the injectivety of the above bijection.
\end{proof} 

\begin{remark}\label{remEquivariance}
\be
\item All maps given in \ref{propMapsInducedByInv} have order two 
by proposition \ref{propDualLatt}. 
\item Let $\Lambda$ be an $o_D$-lattice function. For $g\in\GL_DV$ we have 
\[(g\Lambda)^\#=(g^\sigma)^{-1}\Lambda^\#.\]
\ee
\end{remark}

\begin{proposition}\label{propAffDL}
The map 
\[( )^\#:\Lattone{o_D}{V}\ra\Lattone{o_D}{V}\]
is affine and $\U(h)$-equivariant. 
\end{proposition}

\begin{proof}
The equivariance follows from \ref{remEquivariance}[2.]. 
We prove the affineness with norms and lemma \ref{lemBarAlpha}. Let $(v_i)$ be a splitting basis of two $D$-norms $\alpha$ and $\alpha'.$ We choose an element $\lambda\in [0,1].$ The basis $(v_i)$ also splits $\gamma:=\lambda\alpha+(1-\lambda)\alpha'.$
By lemma \ref{lemBarAlpha} the $D$-basis
$(w_i)$ fulfilling 
\[h(w_i,v_j)=\delta_{i,j}\]
splits $\bar{\alpha},\ \bar{\alpha'}$ and $\bar{\gamma}$ and the values at $w_i$ are 
\[\bar{\alpha}(w_i)=-\alpha(v_i),\ \bar{\alpha'}(w_i)=-\alpha'(v_i)\]
and
\[\bar{\gamma}(w_i)=-\gamma(v_i)=-\lambda\alpha(v_i)-(1-\lambda)\alpha'(v_i).\]
Thus
\[\bar{\gamma}(w_i)=\lambda\bar{\alpha}(w_i)+(1-\lambda)\bar{\alpha'}(w_i),\]
which proves the affiness of the map $\alpha\mapsto \bar{\alpha}.$
\end{proof}


\subsection{MM-norms}

We are interested in the sets of self-dual objects. For the self-dual norms Bruhat and Tits gave another definition, the definition of an MM-norm.

\begin{definition}[\cite{bruhatTitsIV:87}[2.1]]
One says that $\alpha\in \Normone{D}{V}$ is \textit{dominated  by $h$}\idxD{norm dominated by an hermitian form}, if for all
$v,v'\in V$ we have
\beq \label{defDominated}
\alpha(v)+\alpha(v')\leq \nu(h(v,v')).
\eeq
\end{definition}

\begin{remark}\label{remSplitDominated}
If $(v_i)_i$ is a splitting basis for $\alpha$ then $\alpha$ is dominated by $h$ if and only if for all $i,\ j$ we have 
\[\alpha(v_i)+\alpha(v_j)\leq \nu(h(v_i,v_j)).\]
\end{remark}

We make $\Normone{D}{V}$ to a poset by defining $\alpha\leq\beta$ if $\alpha(v)\leq\beta(v)$
for all $v\in V.$ 

\begin{definition}[\cite{bruhatTitsIV:87} 2.1]
A maximal element of the set of $\alpha\in \Normone{D}{V}$ dominated by $h$ is called a \textit{MM-norm for $h$}\idxD{MM-norm} (maximinorante in French).
\end{definition}

\begin{lemma}\label{lemAlphaAlphaBar}
A $D$-norm $\alpha$ satisfies the following three properties.
\be
\item For all $v,v'\in V$ we have 
\beq\label{eqAlphaAlphaBarDominated}
\alpha(v)+\bar\alpha(v')\leq \nu(h(v,v')).
\eeq
\item $\bary(\alpha):=\frac{1}{2}\alpha+\frac{1}{2}\bar\alpha$ is dominated by $h.$
\item If $\alpha$ is dominated by $h$ then $\bary(\alpha)\geq\alpha.$
\ee
\end{lemma}

\begin{proof}
The first assertion follows from the definition of $\bar\alpha$ and it implies the 
second assertion because $\overline{\bary(\alpha)}=\bary(\alpha)$ 
Point 3 follows because to be dominated by $h$ is equivalent to $\alpha\leq\bar{\alpha}.$
\end{proof}

A part of \cite[2.5]{bruhatTitsIV:87} is the following
proposition.

\begin{proposition}[F. Bruhat, J. Tits]
For $\alpha\in \Normone{D}{V}$ the following statements are equivalent.
\be
\item $\alpha=\bar\alpha.$
\item $\alpha$ is a MM-norm.
\ee
\end{proposition}

\begin{proof}
$1.\Ra 2.:$ $\alpha$ is dominated by $h,$ since $\alpha\leq\bar\alpha.$ If $\gamma\geq\alpha$ and $\gamma$ is dominated by $h,$ then $\gamma\leq\bar\gamma\leq\bar\alpha=\alpha,$ thus 2.
$2.\Ra 1.:$ By remark \ref{lemAlphaAlphaBar} (2 and 3) we get $\alpha=\bary(\alpha).$ Thus $\alpha=\bar\alpha.$ 
\end{proof}


\subsection{Self-duality}

We obtain two diagrams with sets of self-dual objects from the 
diagrams (\ref{eqThmCD}) and (\ref{eqThmCD2}). An element of 
one of the sets given in these dia\-grams is called \textit{self-dual}\idxD{self-dual}
if it is a fixed point of the corresponding map given in proposition \ref{propMapsInducedByInv}.

\begin{notation}
We denote the set of self-dual objects as follows:
\bi
\item for $\Normone{D}{V},\ \Norm{D}{V}:$ $\Normone{h}{V},\ \Norm{h}{V},$\idxS{$\Normone{h}{V}$}\idxS{$\Norm{h}{V}$}
\item for $\Lattone{o_D}{V},\ \Latt{o_D}{V}:$ $\Lattone{h}{V},\ \Latt{h}{V}$\idxS{$\Lattone{h}{V}$}\idxS{$\Latt{h}{V}$}
\item for $\Normtwo{k}{A}:$ $\Normtwo{\sigma}{A},$ \idxS{$\Normtwo{\sigma}{A}$}
\item for $\Latttwo{o_k}{A}:$ $\Latttwo{\sigma}{A}.$\idxS{$\Latttwo{\sigma}{A}$}
\ei
\end{notation} 

The next proposition is a corollary of proposition \ref{propAffDL}.

\begin{proposition}\label{propAffStructureOnLatt1hV}
The sets $\Normone{h}{V}$ and $\Lattone{h}{V}$ are closed under the affine structure of $\Normone{h}{V}$ and $\Lattone{h}{V}$ respectively.
\end{proposition}

\begin{proposition}
We get two commutative diagrams of $\U(h)$-equivariant maps.
\beq
\begin{array}{ccc}\label{diagNormSgigma}
\Norm{h}{V} & \ra &\Latt{h}{V}\\
\downarrow &&\downarrow\\
\Normtwo{\sigma}{A} &\ra &\Latttwo{\sigma}{A}\\
\end{array}
\eeq
\beq
\begin{array}{ccc}\label{diagNormh}
\Normone{h}{V} & \ra &\Lattone{h}{V}\\
\downarrow &&\downarrow\\
\Norm{h}{V} &\ra &\Latt{h}{V}\\
\end{array}
\eeq
The maps in the second diagram are affine.
\end{proposition}

The vertical maps of diagram (\ref{diagNormh}) are surjective.
Indead if a $D$-norm $\alpha$ satisfies $[\alpha]^{\sigma}=[\alpha]$ 
there is a real number $s$ such that
\[\bar{\alpha}=\alpha+s.\]
It follows that the norm $\alpha+\frac{s}{2}$ lies in
$\Normone{h}{V},$ because 
\[\overline{\alpha+\frac{s}{2}}=\bar{\alpha}-\frac{s}{2}=\alpha+\frac{s}{2}.\]

\begin{remark}
An analogous argument shows that two self-dual $D$-norms are
equivalent if and only if they equal, i.e. that all maps of diagram (\ref{diagNormh}) are bijective.
\end{remark}

We consider the direct sum of self-dual lattice functions.

\begin{remark}\label{remDSOSDLF}
Let $V'$ be another finite dimensional $D$-right vector space with 
an $\epsilon$-hermitian form $h'.$ On $V\oplus V'$ we have the $\epsilon$-hermitian form $\tilde{h}$ defined by
\[\tilde{h}((v,v'),(w,w'))=h(v,w)+h'(v',w').\]
\be
\item We have 
\[(\Lambda\oplus\Lambda')^\#=\Lambda^\#\oplus\Lambda'^\#\]
for two lattice functions $\Lambda\in\Lattone{o_D}{V}$
and $\Lambda'\in\Lattone{o_D}{V'}.$
\item The direct sum of self-dual lattice functions is self-dual.
\ee 
\end{remark}

\begin{assumption}
For the last part of this subsection let us assume that 
$h$ is isotropic and that 
\[V=W\oplus W'\]
with maximal totally isotropic subspaces of $V.$
We put $k:=\dim_DW.$ 
\end{assumption}

\begin{definition}
For $M\in\Gitter{o_k}{W}$ we define \textit{its dual in $W'$}\idxD{dual of a lattice}
by
\[M^{\#,W'}:=\{w'\in W'|\ h(w',M)\subseteq\tf{p}_D\},\]\idxS{$(\ )^{\#,W'}$}
and analogously $M'^{\#,W}$ for $M'\in\Gitter{o_k}{W'}.$
The \textit{dual of}\idxD{dual of a lattice} $\Lambda\in\Lattone{o_D}{W}$ \textit{in $W'$} is
defined by 
\[\Lambda^{\#,W}(t):=(\Lambda((-t)+))^{\#,W'},\]\idxS{$\Lambda^{\#,W'}$}
and we have an anologous definition for $o_D$-lattice functions of
$W'.$
\end{definition}

\begin{proposition}\label{propIDOLF}
For $M\in\Gitter{o_k}{W},$ $Q\in\Gitter{o_k}{W'},$ $\Lambda\in\Lattone{o_D}{W}$ and $\Lambda'\in\Lattone{o_D}{W'}$ we have:
\be
\item $(M\oplus Q)^\#=Q^{\#,W}\oplus M^{\#,W'}$
and $Q^{\#,W}$ and $M^{\#,W'}$ are full lattices in the corresponding vector spaces.
\item $(\Lambda\oplus\Lambda')^{\#}=\Lambda'^{\#,W}\oplus\Lambda^{\#,W'}$ and $\Lambda'^{\#,W}$ and $\Lambda^{\#,W'}$ are lattice functions in the corresponding vector spaces.
\ee
\end{proposition}

\begin{proof}
For 1.: The equality is a consequence of $h(W,W)=h(W',W')=\{0\},$
i.e. \[(w,w')\in(M\oplus Q)^\#\]
if and only if 
\[h(w,Q)+h(M,w')\subseteq\tf{p}_D\]
if and only if
\[h(w,Q)\cup h(M,w')\subseteq \tf{p}_D\]
if and only if 
\[w\in Q^{\#,W} \text{ and }w'\in M^{\#,W'}.\]
The set $(M\oplus Q)^\#$ is a full lattice in $V$ and by the 
equality we get that $Q^{\#,W}$ is a full lattice in $W$ and $M^{\#,W'}$ is a full lattice in $W'.$

For 2.: From 1. we get the equality. The left side is a lattice function. Thus both summands on the right side of the equation are lattice functions.
\end{proof}

\begin{proposition}\label{propIDInverse}
\be
\item The maps 
\[()^{\#,W'}:\Gitter{o_k}{W}\ra\Gitter{o_k}{W'}\]
and 
\[()^{\#,W}:\Gitter{o_k}{W'}\ra\Gitter{o_k}{W}\]
are inverse to each other.
\item The maps from 1 are affine.
\ee
\end{proposition}

\begin{proof}
\be
\item We take $M\in\Gitter{o_k}{W}$ and $Q\in\Gitter{o_k}{W'}$ and by 1. of proposition \ref{propIDOLF} we get
\[((M\oplus Q)^\#)^\#=(Q^{\#,W}\oplus M^{\#,W'})^\#
=(M^{\#,W'})^{\#,W}\oplus (Q^{\#,W})^{\#,W'}.\]
From 2. of proposition \ref{propDualLatt}
we get 
\[((M\oplus Q)^\#)^\#=M\oplus Q.\]
Both equalities together imply the first assertion. 
\item By proposition \ref{propAffDL} the map $( )^\#$ on $\Lattone{o_D}{V}$ is affine. By remark \ref{remDSOLF} and \ref{propIDOLF}[2.] we get the second assertion.
\ee
\end{proof}

\begin{definition}\label{defEmbGLUh}
For an endomorphism $a\in\End_D(W),$ there is a unique 
endomorphism of $W'$ denoted by $a^{\sigma,W'}$ such that
$(a\oplus 0)^\sigma=0\oplus a^{\sigma,W'}.$
We define an embedding \[i_{W,W'}:\GL_D(W)\ra \U(h)\] as follows
\[i_{W,W'}(g)(w,w'):=(g(w),(g^{\sigma, W'})^{-1}(w')).\]
The map $i_{W,W'}$ defines a $k$-morphism and the differential at identity 
is given by
\[di_{W,W'}(a):=a\oplus (-a^{\sigma,W'}).\]
Its image is a subset of $\Lie(\bU(h))(k_0).$
\end{definition}

\begin{proposition}\label{propAffEmb}
The map
\[\phi:\Lattone{o_D}{W}\ra\Lattone{h}{V}\]
defined by
\[\phi(\Lambda):=\Lambda\oplus\Lambda^{\#,W'}\]
is affine and $\GL_D(W)$-equivariant, i.e.
\[\phi(g\Lambda)=i_{W,W'}(g)\phi(\Lambda).\]
\end{proposition}

\begin{proof}
We have $\Lambda\oplus\Lambda^{\#,W'}\in\Lattone{h}{V}$ by  
\ref{propIDOLF}[2.] and \ref{propIDInverse}[1.].
The affineness follows from \ref{propIDInverse}[2.]
and remark \ref{remDSOLF}. For the equivariance we need
\[(g\Lambda)^{\#,W'}=(g^{\sigma,W'})^{-1}\Lambda^{\#,W'}\] which follows from 
\[h((0,w'),(g(w),0))=h((0,g^{\sigma,W'}(w')),(w,0)).\]
\end{proof}


\section{\texorpdfstring{The Bruhat-Tits building of $\bU(h)$}{The Bruhat Tits building of U(h)}}\label{secBruhatTitsBuildingOfUh}

We adopt assumption \ref{assSDLF}.

\begin{remark}
From \ref{thmDScharlau} we deduce that $D$ can only have an index $d$ which is 1 or 2, and if the index is 2 then $k_0$ equals $k.$
Without loss of generality we can assume $\epsilon=-1$ if $d=2$ by \cite[(22.a)]{bruhatTitsIV:87}.
\end{remark}

In this section we describe the Bruhat-Tits building of $\bSU(h)$ as a subset of $\Building(\tibG,k).$ It was done by Bruhat and Tits in terms of norms. We use  the concept of self dual lattice functions from the last section. This description was introduced in \cite{broussousStevens:09} based on \cite{bruhatTitsIV:87}. 

\begin{notation}
We denote by $\bO^{is}_{2,k}$ the $k$-split isotropic orthogonal group of rank $1,$ i.e. the unitary group given by an isotropic symmetric $k$-bilinear form on $k^2.$ 
\end{notation}

\begin{example}\label{exOmittedCase}
The connected component of $\bO^{is}_{2,k}$ is $k$-isomorphic to $\bGm$ and we can apply section \ref{secBTBuildingsGLDV}. 
Its Bruhat-Tits building over $k$ is a point and its enlarged
building is a line.  The $\bGm(k)$-action on $\Lattone{o_k}{k}$ is extended to an $\bO^{is}_{2,k}(k)$-action via 
\[\antidiag(1,1).(s\mapsto\tf{p}_k^{[s-y]+}):=(s\mapsto\tf{p}_k^{[s+y]+}).\]
\end{example}

\begin{proposition}\label{propCondO2is}
The following three assertions are equivalent.
\be
\item There is a $k_0$-isomorphism between $\bU(h)$ and $\bO^{is}_{2,k_0}.$ 
\item There is a $k_0$-isomorphism between $\bSU(h)$ and $\bGm.$ 
\item The following system of conditions is satisfied:
\[D=k=k_0 \text{ and }m=2\text{ and }\sigma \text{ is orthogonal}\]
\[\text{ and }h \text{ is isotropic.}\]
\ee
\end{proposition}

\begin{proof}
The implication $3\Ra 1$ and $1\Ra 2$ are obvious and $2\Ra 3$ is a direct consequence of 
\ref{propExEnlargedBTBClassGrps}.
\end{proof}
\vspace{1em}

\begin{assumption}
In the following introduction we assume that $\bU(h)$ is not $k_0$-iso\-mor\-phic to $\bO^{is}_{2,k_0}.$ 
\end{assumption}

We consider the building of the valuated root datum given in \cite[1.15.]{bruhatTitsIV:87}, denote it by $\building(\bSU(h),k_0)$\idxS{$\building(\bSU(h),k_0)$} and call it the \textit{Bruhat-Tits building of 
$\bSU(h)$ over $k_0$}. 

\begin{remark}
\be
\item From the definition of a building corresponding to a valuated root datum it follows that in the anisotropic case the building is a point and in the isotropic case the building is the geometric realisation of a thick Euclidean building.
\item The apartments are in one to one correspondence with the Witt decompositions.
\ee
\end{remark}

Let us now fix a Witt basis $(v_i)_{I\cup I_0}$ of $V$ with respect to $h,$ i.e. we have the Witt decomposition 
\[V_i:=v_iD,\ i\in I,\ V_0:=\sum_{i\in I_0} v_iD.\]  
Let $T$ be the torus defined over $k_0$ whose set of $k_0$-rational
points is given by the $k_0$-rational points $t$ of $\bSU(h)$
satisfying:
\be
\item $t.v_i\in k_0v_i$ for all $i\in I$ and 
\item $t.v=v$ for all $v\in V_0$
\ee
Then $T$ is a maximal $k_0$-split torus of $\bSU(h).$
We look at the characters $a_i$ defined on $T(K_0)$ by
\[tv_i=\alpha_i(t)^{-1}v_i,\ t\in T(K_0).\]
There is a bijection from $\Delta$ the apartment of $\building(\bSU(h),k_0)$ corresponding to the torus $T$ to the set of the MM-norms which split under the given Witt basis. 
\[x\in \Delta\mapsto \alpha_x\]
where 
\[\alpha_x(\sum_{i\in I}v_i\lambda_i+v_0):=\inf\{\frac{1}{2}\nu(q(v_0)),\inf_{i\in I}\{\nu(\lambda_i)-a_i(x)\}\}.\] 
Here $q$ is the pseudo-quadratic form corresponding to $h.$
(see \cite[1.2. (9)]{bruhatTitsIV:87})
This gives a map from $\Delta$ to $\Lattone{h}{V},$ whose image is the  set of self-dual lattice functions which are split under our 
Witt basis. This set is denoted by $\Lattone{h,(V_i)_i}{V}.$
We denote the self dual lattice function corresponding to $x\in \Delta$ 
by $\Lambda_x.$

\bdfn\label{defBuildingU(h)}
If $\bU(h)$ is connected then the Bruhat-Tits building $\building(\bU(h),k_0)$\idxS{$\building(\bU(h),k_0)$} of $\U(h)$ is defined analogously and canonically identifies with $\building(\bSU(h),k_0).$ If $\bU(h)$ is not connected we define 
$\building(\bU(h),k_0)$ to be $\building(\bSU(h),k_0).$ 
\edfn

\begin{remark}
For the case we consider there is no proper enlarged building of $\bSU(h)$ and $\bU(h)$ because 
$X^*(\bSU(h))_{k_0}$ and $X^*(\bU(h)^0)_{k_0}$ are trivial, i.e. 
 we have 
\[\building(\bSU(h),k_0)=\Building(\bSU(h),k_0)=\building(\bU(h),k_0)=\Building(\bU(h),k_0).\]
\end{remark}

Broussous and Stevens proved a reformulation of \cite[2.12]{bruhatTitsIV:87} for the case where $D=k.$ With minor changes their proof is valid for the case $D\neq k$
(see also \cite{lemaire:09} \S4).

\begin{theorem}\cite[Prop. 4.2.]{broussousStevens:09}\label{thmBS4.2}
There a unique $\bU(h)(k_0)$-equivariant affine map 
\[\building(\bU(h),k_0)\ra \Lattone{h}{V}.\]
It is bijective and an extension of the map
\[x\in \Delta\mapsto \Lambda_x.\]
\end{theorem}

\brem\label{remIdO2}
In the omitted case, see example \ref{exOmittedCase}, we have \[\Building(\bO^{is}_{2,k},k)\cong\Lattone{o_k}{V}\cong\bbR,\ (\tf{p}_k^{[s-y]+})_{s\in\bbR}\mapsto y.\]\idxS{$\Building(\bO^{is}_{2,k},k)$}
The affine space $\Lattone{h}{V}$ can also be identified with $\bbR$ if one fixes a Witt-basis $(v_1,v_2)$ for the unique Witt decomposition of $V,$ precisely
\[y\in\bbR\mapsto\ (\tf{p}_k^{[s-y]+}v_1\oplus\tf{p}_k^{[s+y]+}v_2)_s.\]
The identity of $\bbR$ induces the unique $\bO^{is}_{2,k}(k)$-equivariant affine bijection from $\Building(\bO^{is}_{2,k},k)$ to $\Lattone{h}{V}$ because the identity is the only affine  map $j$ of $\bbR$ which satisfies $j(y+1)=j(y)+1$  and $j(-y)=-j(y)$ for all $y\in\bbR.$ 
\erem

We also have a notion of a Lie algebra filtration here.

\begin{definition}\label{defLieAlgebraFiltrationUh}
The set of $k_0$-rational points of $\Lie(\bU(h))$ is the set 
\[\{A\in a |\ a+\sigma(a)=0\}\]
of skewsymmetric $D$-en\-do\-mor\-phisms of $V$ with respect to $\sigma.$ For a point $x$ of \\ $\Building(\bU(h),k_0)$ the following intersection
\[\LF(x,\bU(h),k_0):=\LF(x,\tibG,k)\cap\Lie(\bU(h))(k_0)\]\idxS{$\LF(x,\bU(h),k_0)$}
defines an $o_{k_0}$-lattice function in $\Lie(\bU(h))(k_0).$
It is called the \textit{Lie algebra filtration}\idxD{Lie algebra filtration in the case of unitary groups} of 
$x$ in $\Lie(\bU(h))(k_0).$
\end{definition}

\begin{theorem}\cite{lemaire:09}
The filtration $\LF(x,\bU(h),k_0)$ coincides with the\\ Moy-Prasad fil\-tra\-tion.
\end{theorem}

\chapter{Maps which are compatible with the Lie algebra filtrations}\label{chMapsWhichAreCLF}

\section{Compatibility with the Lie algebra filtrations}\label{secCLFProperty}

The notion of CLF-map was introduced in \cite{broussousLemaire:02}.
Let $F$ be non-Archimedean local field with valuation ring $o_F.$

\begin{definition}
Let $B$ be a finite dimensional Lie algebra over $F.$ An $o_F$-Lie algebra filtration of $B$ is an $o_F$-lattice function of $B.$
\end{definition}

\begin{definition}\label{defLFGroup}\idxD{Lie algebra filtration in general}
\be
\item A reductive $F$-group $G$ with the Bruhat-Tits buil\-ding\\ $\building(G,F)$ is said to be an \textit{LF-$F$-group}, if every point $x$ of $\building(G,F)$ is at\-tached to an $o_F$-Lie algebra filtration $\LF(x,G,F)$ of $\Lie(G)(F).$ If there is no confusion we skip the prefix $o_F.$\idxS{$\LF(x,G,F)$}
\item The Lie algebra filtration of a Lie algebra $B$ attached to a point $x$ is also denoted by  $\LF(x,B).$\idxS{$\LF(x,B)$}
\ee
\end{definition}

If $G$ is a connected LF-$F$-group we can attach a Lie algebra filtration to every element of 
the enlarged building $\Building(G,k)$ 
if we use the projection to the first component
\[\Building(G,F)\ra\building(G,F),\ (y,w)\mapsto y,\] see \ref{chapEnlargedBuildingOfAReductiveGroup},
i.e. we define
\[\LF((y,w),\bG,F):=\LF(y,\bG,F).\]

\begin{definition}\label{defExtOfAPoint}
Let $\bH$ and $\bG$ be LF-$F$-groups. Let $i:\bH\ra\bG$ be an $F$-ho\-mo\-mor\-phism. We call a point $y$ of $\building(\bG,F)$ an \textit{extension of $x\in\building(\bH,F)$ with respect to $i$}\idxD{extension of a point}  if 
\beq\label{eqCompLF}\LF(y,\bG,F)\cap \im(di)=di(\LF(x,\bH,F))\eeq
where $di:\Lie(\bH)\ra\Lie(\bG)$ is the differential of $i.$
We omit to mention $i$ if the choice of $i$ is clear. An analogous definition can be made using also enlarged buildings.
\end{definition}
 
\begin{definition}\label{defCLF}
Under the assumptions of definition \ref{defExtOfAPoint} a map
$j$ between subsets of the buildings $\building(\bH,F)$ and $\building(\bG,F)$
is \textit{compatible with the Lie algebra filtrations (CLF) with respect to $i$}\idxD{CLF-map}\idxD{compatible the Lie algebra filtrations} if we have that an element $y$ of $\building(\bG,F)$ is an extension of  $x\in\building(\bH,F)$ if $j$ maps $x$ to $y$ or $y$ to $x.$
We give analogous definitions for maps between subsets of enlarged buildings or between subsets of an enlarged and a non-enlarged building.
\end{definition}

\begin{example}
Assume we have given a local hermitian $k$-datum of residue characteristic not two with unitary involution $\sigma.$
We have the inclusion
\[\bU(h)\ra\Res_{k|k_0}(\bGL_k(V))\stackrel{/k}{\cong}\bGL_k(V)\times\bGL_k(V),\]
\[\U(h)=\bU(h)(k_0)\subseteq\Res_{k|k_0}(\bGL_k(V))(k_0)=\GL_k(V)=\bGL_k(V)(k)\]
and 
\[\Lie(\bU(h))(k_0)\subseteq\Lie(\Res_{k|k_0}(\bGL_k(V)))(k_0)=\End_k(V)=\Lie(\bGL_k(V))(k).\]
Thus we have a notion of CLF for maps
\[\Building(\bU(h),k_0)\ra\Building(\bGL_k(V),k)=\Building(\Res_{k|k_0}(\bGL_k(V)),k_0).\]
\end{example}

The aim of this work is to analyse how precise a map is determined by the CLF property. 


\section{Buildings of centralisers}\label{secBuildOfCentralisers}

\begin{assumption}\label{assCh3and4}
For the rest of part 1 we adopt assumption \ref{asstiG} and we assume that $k$ has residue characteristic not two.
\end{assumption}

From section \ref{secBuildOfCentralisers} to section \ref{secUniquenessInTheGeneralCase} we only 
consider centralisers of \textbf{separable} Lie algebra elements and separable field extensions. 
In section \ref{secGeneralisationToTheNonseparableCase} we explain how the results of chapter 
\ref{chMapsWhichAreCLF} and \ref{chUniquenessResults} generalise to the non-separable case.

\begin{notation}
For a group action 
\[G\times W\ra W\] we denote the fixator of an element $w\in W$ by $G_{w}$ and of a subset $S$ of $W$ by $G_S.$ \idxS{$G_S$}
If $W$ is the Lie algebra of an algebraic group and if we do not specify the action, we use the adjoint group action.  
\end{notation}

\subsection{\texorpdfstring{The case of $\bGL_D(V)$}{The case of GLD(V)}}\label{subsecForGLDV}

Let $E$ be a commutative separable $k$-sub\-alge\-bra of 
$\Lie(\tibG)(k),$ i.e. $E$ splits in\-to a pro\-duct of separable field ex\-ten\-sions
of $k:$
\[E=\prod_iE_i.\]
If $1_i$ is the idempotent corresponding to $1_{E_i}$ we put
$V_i:=1_iV.$ 
For every $i$ there is an $E_i$-algebra isomorphism 
\[\End_{\tens{E_i}{k}{D}}(V_i)\cong \End_{\Delta_i}(W_i)\]
for some skewfields $\Delta_i$ central over $E_i$ and some finite dimensional $\Delta_i$-vector spaces $W_i.$

By the separability of $E$ over $k$ there is a canonical $k$-isomorphism
\[\tibG_E\cong\prod_i\bGL_D(V_i)_{E_i}\]
and thus the centraliser $\tibG_E$\idxS{$\tibG_E$} is $k$-isomorphic to 
\[\prod_i\Res_{E_i|k}(\bGL_{\Delta_i}(W_i)).\]
 The building of $\tibG_E$ over $k$ is $\GL_D(V)_E$-equivariantly isomorphic to 
\beq\label{eqbuildProdGL}\prod_i\building(\bGL_{\Delta_i}(W_i),E_i)\eeq
and the enlarged building is $\GL_D(V)_E$-equivariantly isomorphic to 
\beq\label{eqBuildProdGL}\prod_i\Building(\bGL_{\Delta_i}(W_i),E_i).\eeq
We identify these products with $\building(\tibG_E,k)$\idxS{$\building(\tibG_E,k)$} and $\Building(\tibG_E,k)$\idxS{$\Building(\tibG_E,k)$} respectively, and we work with the lattice function models of the factors.

\begin{notation}\label{notHDelta}
Instead of $\bGL_{\Delta_i}(W_i)$ we write $\bGL_{\tens{E_i}{k}{D}}(V_i).$\idxS{$\bGL_{\tens{E_i}{k}{D}}(V_i)$}
\end{notation}

\begin{definition}\label{defLFInProduct}
The Lie algebra filtration of a point $x=(x_i)$ of $\building(\tibG_E,k)$ or the enlarged building $\Building(\tibG_E,k)$ is given by 
the direct sum of the Lie algebra filtrations of the points $x_i,$ i.e.
\[\LF(x,\tibG_E,k)(t):=\oplus_i\LF(x_i,\bGL_{\tens{E_i}{k}{D}}(V_i),E_i)(t),\ t\in\bbR,\]\idxS{$\LF(x,\bG_E,k)$}
the sum of the corresponding square lattice functions.
\end{definition}


\subsection{\texorpdfstring{The case of $\bU(h)$}{The case of U(h)}}\label{subsecForUh}

Here we use the same idea as in the previous subsection.
We take a local hermitian datum with the fixed simple datum of assumption \ref{assCh3and4}. 
We consider the unitary group $\bG:=\bU(h)\subseteq\Res_{k|k_0}(\tibG )$\idxS{$\bG$}

Let $\beta$ be an element of 
\[\Lie(\bG)(k_0)=\{a\in\End_D(V)|\ a^\sigma+a=0\}\]\idxS{$\Lie(\bG)(k_0)$}
which is \textbf{separable} over $k,$
and we put $\bH:=\bG_\beta.$\idxS{$\bH$}

The semisimplicity of $k[\beta]$ gives us the following decompositions:

\bi
\item $E:=k[\beta]=\prod_{i\in J}E_i,$\idxS{$E_i,\ 1_i,\ V_i,\ \beta_i$} a product of fields,
\item $1=\sum_{i\in J}1_i,$ the decomposition of $1$ into primitive idempotents,
\item $V:=\oplus_{i\in J}(V_i),$ $V_i:=1_iV$ and
\item $\beta=\sum_{i\in J}\beta_i,\ \beta_i=1_i\beta.$
\ei

On $J$ we choose a representation system $J_{un+}$ for the equivalence relation defined by 
\[i\sim j\text{ if } i=j \text{ or }\sigma(1_i)=1_j\]
and we put 
\[J_{un}:=\{j\in J|\ \sigma(1_j)=1_j\}\text{ }J_+:=J_{un+}\setminus J_{un}\] and \[J_-:=J\setminus J_{un+}.\]\idxS{$J_{un},\ J_{un+},\ J_{+},\ J_{-}$}
We define $-i:=j$ if $\sigma(1_i)=1_j.$
For $i\in J_{un}$ we denote $(E_i)_0$ to be the set of fixed points 
of $\sigma$ in $E_i.$

We obtain a lattice function model for the enlarged building of $\bH$ in the 
following way. The following polynomial isomorphism
\begin{align*}
\bH(k_0) &= \bG(k_0)_{\beta} \\
&= \prod_{i\in J_{un+}}U(h_{(V_i+V_{-i})\times (V_i+V_{-i})})_{\beta_i+\beta_{-i}} \\
&\cong \prod_{i\in J_{un}}U(h_{V_i\times V_{i}})_{\beta_i}\times \prod_{i>0}
\GL_{\tens{E_i}{k}{D}}(V_i)_{\beta_i}\\
&= \prod_{i\in J_{un}} \Res_{(E_i)_0|k_0}(\bU(\sigma|_{\End_{\tens
{E_i}{k}{D}}(V_i)}))(k_0)\times \prod_{i>0}\Res_{E_i|k_0}(\bGL_{\tens
{E_i}{k}{D}}(V_i))(k_0)
\end{align*}
extends to an algebraic isomorphism
\[\bH\cong \prod_{i\in J_{un}} 
\Res_{(E_i)_0|k_0}(\bU(\sigma|_{\End_{\tens{E_i}{k}{D}}(V_i)}))\times 
\prod_{i>0}\Res_{E_i|k_0}(\bGL_{\tens{E_i}{k}{D}}(V_i))\]
because $\bH(k_0)$ is Zariski-dense in $\bH$ because $\bH$ is reductive and defined over the infinite field $k_0$ 
by the separability of $\beta.$ For he definition of 
$\bU(\sigma|_{\End_{\tens{E_i}{k}{D}}(V_i)})$ see \ref{defbUsigmah}.
For a reductive group $\Gamma$ defined over a local field $L$, we have 
\[ \Building(\Gamma, L)\cong\Building(\Res_{L|F}(\Gamma),F)\]
for every finite separable field extension $L|F.$ Thus $\Building(\bH,k_0)$ is isomorphic to
\[\prod_{i\in J_{un}} 
\Building (\Res_{(E_i)_0|k_0}(\bU (\sigma|_{\End_{\tens{E_i}{k}{D}}(V_i)})),k_0)
\times \prod_{i>0}\Building (\Res_{E_i|k_0}(\bGL_{\tens{E_i}{k}{D}}(V_i)),k_0)\]
which by the method of restriction of scalars is isomorphic to 
\[\cong \prod_{i\in J_{un}} \Building (\bU (\sigma|_{\End_{\tens
{E_i}{k}{D}}(V_i)}),(E_i)_0)\times \prod_{i>0}\Building (\bGL_{\tens
{E_i}{k}{D}}(V_i),E_i).\]

\begin{definition}\label{defLFForProductsInTheUnitaryCase}
Analogously to definition \ref{defLFInProduct} the Lie algebra filtration of a point $x=(x_i)$ of $\Building(\bH,k_0)$ is 
defined to be the direct sum of the Lie algebra filtrations of the points $x_i,$ i.e. for $t\in\bbR$ we put 
$\LF(x,\bH,k_0)(t)$ to be \[\prod_{i\in J_{un}}\idxS{$\LF(x,\bH,k_0)=\LF(x,\bG_{\beta},k_0)$}
\LF(x_i,\Skew(\End_{\tens{E_i}{k}{D}}(V_i),\sigma))(t)\times\prod_{i\in J_+}
\LF(x_i,\bGL_{\tens{E_i}{k}{D}}(V_i),E_i)(t).\]
\end{definition}

\subsection{Notation and simplification for the unitary case}

We use the notation of subsection \ref{subsecForUh}.
Under the choice of $J_+$ we make the following simplification of the situation. We put $J_{GL}:=J_+\cup J_-$\idxS{$J_{GL}$} and we introduce the following notation.
For a symbol $a\in\{+,-,un,GL\}$ we put 
\[V_a:=\sum_{i\in J_a}(V_i),\ E_a:=\sum_{i\in J_a}E_i,\ 1_a:=\sum_{i\in J_a}1_i,\ \beta_a:=\sum_{i\in J_a}\beta_i,\ \tibG_a:=\bGL_D(V_a),\]
for $b\in\{un,GL\}$ we put 
\[h_b:=h|_{V_b\times V_b},\ \sigma_b:=\sigma|_{\End_D(V_b)},\ \bG_b:=\bU(h_b),\ \bH_b:=(\bG_b)_{\beta_b}\] and for 
$c\in\{un,+\}$ we put $\tibH_c:=(\tibG_c)_{E_c}.$\idxS{$\tibG_+,\ \tibG_-,\ \tibG_{un},\ \tibG_{GL}$}\idxS{$\bG_{un},\ \bG_{GL}$}\idxS{$\bH_{un},\ \bH_{GL}$}\idxS{$\tibH_{un},\ \tibH_{+}$}\idxS{$h_{un},\ h_{GL}$}\idxS{$\sigma_{un},\ \sigma_{GL}$}\idxS{$V_+,\ V_-,\ V_{un},\ V_{GL}$}\idxS{$E_+,\ E_-,\ E_{un},\ E_{GL}$}\idxS{$1_+,\ 1_-,\ 1_{un},\ 1_{GL}$}\idxS{$\beta_+,\ \beta_-,\ \beta_{un},\ \beta_{GL}$}
For example we have the following commutative diagrams.
\[\begin{matrix}
\tibG_{GL}(k)\times\tibG_{un}(k)&\ra & \tibG(k)\\
\ua & & \ua \\
\bG_{GL}(k_0)\times\bG_{un}(k_0)&\ra & \bG(k_0),
\end{matrix}
\]
\[
\begin{matrix}
\tibH_{un}(k)&\ra & \tibG_{un}(k)\\
\ua & & \ua \\
\bH_{un}(k_0)&\ra & \bG_{un}(k_0),
\end{matrix}
\]
\begin{remark} We have 
\[\Building(\bH,k_0)=\Building(\bH_{un},k_0)\times\Building(\bH_{GL},k_0).\]
\end{remark}


\section{\texorpdfstring{CLF-maps in the case of $\bGL_D(V)$}{CLF maps in the case of GLD(V)}}\label{secForGLDV}

In this section we recall results of \cite{broussousLemaire:02}.
We fix a separable field extension $E|k$ in $\Lie(\tibG)(k).$ 
 The group $E^{\times}$ acts on $\Latttwo{o_k}{A}$
by conjugation, i.e. there is an $E^{\times}$-action on $\building(\tibG,k).$

\begin{theorem}\label{thmBLII.1.1}\cite[II.1.1]{broussousLemaire:02}
There is a unique CLF-application 
\[j:\ \building(\tibG,k)^{E^\times}\ra\building(\tibG_E,k).\]
The map $j$ is bijective and $j^{-1}$ is the unique $\tibG_E(k)$-equivariant affine map from $\building(\tibG_E,k)$ to $\building(\tibG,k).$
\end{theorem}

\begin{notation}\idxS{$j_E$}\idxS{$j^E$}
We denote 
\[j_E:=j\text{ and }j^E:=j^{-1}.\]
In this part of the thesis we mainly consider $j^E.$
\end{notation}

In \cite{broussousLemaire:02} the authors describe $j^{E}.$ They
define a map between the enlarged buildings and apply the projection to the non-enlarged buildings. 
The projection from an enlarged to the non-enlarged building is 
\[(y,w)\mapsto y.\] 

\begin{theorem}\cite[II.3.1, II.4]{broussousLemaire:02}\label{thmBL02I.3.1II.4}
There is an affine $\tibG_E(k)$-equivariant CLF-map $\tilde{j}$ from 
$\Building(\tibG_E,k)$ to $\Building(\tibG,k),$
such that the first component of $\tilde{j}(y,w)$ is $j^{E}(y).$
\end{theorem}

\begin{proof}
The existence of $\tilde{j}$ is stated in lemma \cite[II.3.1]{broussousLemaire:02}.
The affineness is proven in section \cite[II.4]{broussousLemaire:02} for $j^{E},$ but the proof actually shows that $\tilde{j}$ is affine.
The $\bG_E(k)$-equivariance follows from the formula of $\tilde{j}$
given in \cite[II.3.1]{broussousLemaire:02}.
\end{proof}

\begin{corollary}\label{corImageOfTildejIsETimesFP}
The image of $\tilde{j}$ is the set of $o_D$-lattice functions of $V$
which are $o_E$-lattice functions.
\end{corollary}

\begin{proof}
If $\Lambda\in \Lattone{o_D}{V}$ is in the image of $\tilde{j}$ then
$\Lambda+l$ is in the image too for every integer $l,$ because of the $k^\times$-equivariance. The affineness implies that a lattice function is an element of $\im(\tilde{j})$ if and only if the whole class is a subset of $\im(\tilde{j}).$ The assertion follows now from 
\[\im(j^E)=(\Latt{o_D}{V})^{E^\times}.\]
\end{proof}
\vspace{1em}

One has a uniqueness result for $j^E$ on the level of non-enlarged buildings.

\begin{theorem}\cite[10.3]{broussousStevens:09}\label{thmBS10.3}
For two points $y\in\building(\tibG,k)$ and $x\in\building(\tibG_E,k)$ which satisfy
\[\LF(y,\tibG,k)\cap \End_{\tens{E}{k}{D}}(V)\supseteq\LF(x,\tibG_E,k)\]
we have $j^{E}(x)=y.$
\end{theorem}

In \cite[10.3]{broussousStevens:09} this theorem was proven for the case
$D=k$ and $E$ is generated by one element. The proof did not use the second assumption, and it goes over to $D\neq k$ without changes. 
The theorem generalises easily to semisimple subalgebras.

We assume now that $E$ is a semisimple $k$-subalgebra of $\End_D(V)$ and $E$ may not be a field.
We put $o_E:=\sum_io_{E_i}$.

As in theorem \ref{thmBLII.1.1} we describe $\building(\tibG_E,k)$ as a subset of $\building(\tibG,k).$ We can not use the action of $E^{\times}$ because there are no fixed points in  $\building(\tibG,k)$ if $E$ is not a field. The set of $o^{\times}_E$-fixed points is too big. We therefore introduce a new notion of lattice function. 

\bdfn
An \textit{$o_E$-$o_D$-lattice function}\idxD{lattice function} of $V$ is an $o_D$-lattice function of $V$ 
which splits under $(V_i)$ such that for every $i$ the function
\[t\mapsto \Lambda(t)\cap V_i\]
is an $o_{E_i}$-lattice function of $V_i.$ We denote the set of $o_E$-$o_D$-lattice functions
by\\ $\Lattone{o_E-o_D}{V}.$
\edfn

\begin{theorem}\label{thmGenThmBS10.3}
Under the assumptions and notation of subsection \ref{subsecForGLDV}
there is an affine and $\tibG_E(k)$-equivariant CLF-map
\[\Building(\tibG_E,k)\ra\Building(\tibG,k),\] whose image in terms of lattice functions is
$\Lattone{o_E-o_D}{V}.$ 
\end{theorem}

\begin{proof}
For every index $i$ theorem \ref{thmBL02I.3.1II.4} ensures the existence of an affine $\GL_{\tens{E_i}{k}{D}}(V_i)$-equivariant CLF-map \[\tilde{j}_i:\Building(\bGL_{\tens{E_i}{k}{D}}(V_i),E_i)\ra\Building(\bGL_D(V_i),k).\] Thus the product $\tilde{j}:=\prod_i\tilde{j}_i$ is an affine and $\tibG_E(k)$-equivariant CLF-map.
The map
\[\oplus_*:\prod_i\Building(\bGL_D(V_i),k)\ra\Building(\tibG,k)\]
defined by 
\[(\Lambda_i)_i\mapsto \oplus_i\Lambda_i\]
is affine and $\prod_i\GL_D(V_i)$-equivariant by 
remark \ref{remDSOLF} and CLF by proposition \ref{propCLFunderDS}.
Thus the map
\[j:=\oplus_*\circ \tilde{j}\] fulfils the asserted properties. The assertion about the image follows from corollary \ref{corImageOfTildejIsETimesFP}.
\end{proof}


\section{\texorpdfstring{CLF-maps in the case of $\bO_{2,k}^{is}$}{CLF maps in the case of O2kis}}\label{secIsotropicO2}

We consider a local hermitian $k$-data which satisfies

\begin{equation}\label{eqCondO2is}
D=k=k_0,\dim_kV=2,\rho=\id,\text{ Witt index}=1, \epsilon=1.
\end{equation}

\begin{remark}\label{remCalcisO2}
There is only one Witt decomposition and we have a corresponding $k$-basis $v_1,\ v_2$ such that 
\[h(v_1\lambda_1+v_2\lambda_2,v_1\mu_1+v_2\mu_2)=\lambda_1\mu_2+\lambda_2\mu_1.\]
The objects are the followings.
\be
\item $\bO^{is}_{2,k}(k)=\{\left(\begin{matrix} \alpha & 0 \\ 0 & \alpha^{-1}\end{matrix}\right),\ \left(\begin{matrix} 0 & \alpha \\ \alpha^{-1} & 0\end{matrix}\right)|\ \alpha\in k^\times\}$
\item We have $\sigma(B)=\tilde{B},\ B\in\GL_2(k),$ where $\tilde{B}$ is obtained from $B$ in permuting the diagonal entries.
\item The set $\Lie(\bO_{2,k}^{is})(k)$ is the set of diagonal matrices
$\diag(a,-a)$ where $a$ runs over $k.$ Any element of $\Lie(\bO_{2,k}^{is})(k)$ 
is therefore separable, e.g. for $a\neq 0$ we have
\[k[\diag(a,-a)]=k[\diag(1,-1)]=k\diag(1,0)\times k\diag(0,1).\]
\item The group $(\bO_{2,k}^{is})^0$ is canonically $k$-isomorphic to $\bGm$ via the following embedding:
\[g\in\bGm(k)\mapsto incl(g)\in \bO_{2,k}^{is}(k)\]
where $incl(g)(v_1)=v_1g$ and $incl(g)(v_2)=v_2g^{-1}.$ 
\ee
\end{remark}

\begin{remark}
Every element of $\Building(\bO_{2,k}^{is},k)$ has the same Lie algebra filtration, precisely
\[t\mapsto \{\diag(a,-a)|\ a\in\tf{p}_k^{[t]+}\}.\]
\end{remark}

\begin{proposition}\label{propExistenceO2Is}
There is an affine map 
\[j:\ \Building((\bO_{2,k}^{is})_\beta,k)\ra\Building(\bO_{2,k}^{is},k)\]
which is affine, $(\bO_{2,k}^{is})_\beta(k)$-equivariant and compatible with the Lie algebra filtrations.
\end{proposition}

\begin{proof}
In the case of $\beta=0$ we take for $j$ the identity of $\Building(\bO_{2,k}^{is},k).$
If $\beta$ is not zero the following map $j$ defined by
\[(\tf{p}_k^{[t+s]+})_{t\in\bbR}\mapsto 
(v_1\tf{p}_k^{[t+s]+}+v_2\tf{p}_k^{[t-s]+})_{t\in\bbR}\]
is compatible with the Lie algebra filtrations with respect to $incl,$
because on both sides there is only one filtration and we have
\[d(incl)(\tf{p}_k^{[t]+})=\{\diag(a,-a)\mid\ a\in\tf{p}_k^{[t]+}\}.\]
The affineness and the equivariance are obvious.
\end{proof}
\vspace{1em}

We consider the identifications of remark \ref{remIdO2}.

\begin{proposition}\label{propTransO2Is}
A map 
\[j:\ \Building((\bO_{2,k}^{is})_\beta,k)\ra\Building(\bO_{2,k}^{is},k)\]
which is $\bGm(k)$-equivariant and affine is a translation of $\bbR.$
The translations of $\bbR$ are in terms of lattice functions 
$\bGm(k)$-equivariant.
\end{proposition}

\begin{proof}
The linear part of an affine map $j$ on $\bbR$ must be 
the identity if $j$ satisfies
\[j(s+1)=j(s)+1\]
for all $s\in \bbR.$ Thus such a $j$ must be a translation. 
\end{proof}

\begin{remark}\label{remUniqO2isEquivTranslation}
The identity is the only $\bO_{2,k}^{is}(k)$-equivariant translation of $\bbR.$
\end{remark}


\section{\texorpdfstring{CLF-maps in the unitary case}{CLF maps in the unitary case}}\label{secTheUnitaryCase}

We use the notation of subsection \ref{subsecForUh}.

\begin{convention}\label{conbGnotO2is}
We omit the case where $\bG$ is $k_0$-isomorphic to $\bO_{2,k_0}^{is},$ i.e. the case  considered in section \ref{secIsotropicO2}. Therefore $\bG$ has no proper enlarged building over $k_0.$
\end{convention}

\begin{convention}\label{convFixedBuildingModels}
We work with the models of the buildings in terms of lattice functions and square lattice functions and we have thus fixed isomorphisms of the form given in theorem \ref{thmBLI1.4} and theorem \ref{thmBS4.2}. 
\end{convention}

\begin{theorem}\label{thmExistence}
There is an injective, affine and $\bH(k_0)$-equivariant CLF-map
\[j:\ \Building(\bH,k_0)\ra\building(\bG,k_0)\]
whose image in terms of lattice functions is the set of self-dual $o_E$-$o_D$-lattice functions of $V.$
\end{theorem}

We construct the map using the diagram
\[\begin{matrix}
\Building(\bH_{un},k_0)\times \Building(\bH_{GL},k_0) & \stackrel{\phi}{\ra} & \Building(\bG_{un},k_0)\times \Building(\bG_{GL},k_0)\\
\parallel & & \psi\da\\
\Building(\bH,k_0) & \stackrel{j}{\ra} & \building(\bG,k_0)\\
\end{matrix}\]

We have to construct $\phi$ and $\psi.$

\begin{lemma}\label{lemSplitting}
There is an injective, affine and $\bG_{un}(k_0)\times\bG_{GL}(k_0)$-equivariant CLF-map $\psi.$
\end{lemma}

\begin{proof}
We define $\psi$ to be the map
\[\Lattone{h_{un}}{V_{un}}\times\Lattone{h_{GL}}{V_{GL}}\ra \Lattone{h}{V}\]
given by
\[(\Lambda_{un},\Lambda_{GL})\mapsto \Lambda_{un}\oplus\Lambda_{GL}.\]
The map is welldefined by 2. of remark \ref{remDSOSDLF} because
$h(V_{un},V_{GL})=\{0\}.$ The affineness, the equivariance and the injectivity are obvious. The CLF-property follows from 
proposition \ref{propCLFunderDS} and 
\[\Skew(\End_D(V),\sigma_h)\cap (\End_D(V_{un})\oplus\End_D(V_{GL}))=\]
\[\Skew(\End_D(V_{un}),\sigma_{un})\oplus\Skew(\End_D(V_{GL}),\sigma_{GL}).\]
\end{proof}

\begin{lemma}\label{lemCaseJ0}
There is an injective, affine and $\bH_{un}(k_0)$-equivariant CLF-map
\[\phi_{un}:\ \Building(\bH_{un},k_0)\ra\Building(\bG_{un},k_0).\]
\end{lemma}

\begin{proof}
\be 
\item At first we assume that $E_{un}$ is a field.
We use diagram (\ref{diagNormh}) and in terms of lattice functions we consider $\Building(\bH_{un},k_0)$  and $\Building(\bG_{un},k_0)$ as em\-bed\-ded in $\building(\tibH_{un},k)$ and $\building(\tibG_{un},k)$ respectively.
We have to prove that the image of $\Building(\bH_{un},k_0)$ under 
$j^{E_{un}}$ of theorem \ref{thmBLII.1.1} is a subset of $\Building(\bG_{un},k_0).$
For an element $x$ of $\Building(\bH_{un},k_0)$ we have
\[\LF(x,\Lie(\tibH_{un}),k)=\LF(j^{E_{un}}(x),\Lie(\tibG_{un}),k)\cap\End_{\tens{E_{un}}{k}{D}}(V_{un})\] by the CLF-property of $j^{E_{un}}.$
The left hand side is invariant under $\sigma,$ and we 
obtain 
\beq\label{eqComp}
\LF(x,\Lie(\tibH_{un}),k)=\LF(j^{E_{un}}(x),\Lie(\tibG_{un}),k)^\sigma\cap\End_{\tens{E_{un}}{k}{D}}(V_{un}).
 \eeq
By \ref{propMapsInducedByInv}[4.] there is a point $y$ of $\building(\tibG_{un},k)$ whose Lie algebra filtration is
\[\LF(j^{E_{un}}(x),\Lie(\tibG_{un}),k)^\sigma.\] 
By theorem \ref{thmBS10.3} the equation (\ref{eqComp}) implies $y=j^{E_{un}}(x),$ i.e. the Lie algebra filtration of $j^{E_{un}}(x)$ in 
$\End_D(V_{un})$ is self-dual and therefore 
$j^{E_{un}}(x)$ lies in $\Building(\bG_{un},k_0).$ We define \[\phi_{un}:=j^{E_{un}}|_{\Building(\bH_{un},k_0)}.\]
\item If $E_{un}$ is not necessarily a field we get for every $i\in J_{un}$ a map $\phi_{un,i}$ constructed above. The image of $\phi_{un,i}$ is a subset of $\Lattone{h|_{V_i\times V_i}}{V_i}$ and we define $\phi_{un}$ to be the direct sum of the maps $\phi_{un,i}.$ The assertion follows now from 2. of
remark \ref{remDSOSDLF}. 
\item We now prove the injectivity of $\phi_{un}.$ If two tupel of self-dual lattice functions $(\Lambda_i)$ and $(\Lambda'_i)$ are in the same fiber of $\phi_{un}$ we obtain
\[j^{E_i}([\Lambda_i])=j^{E_i}([\Lambda'_i])\] for all indexes $i.$ The injectivity of $j^{E_i}$
implies $[\Lambda_i]=[\Lambda'_i]$ and the self-duality implies $\Lambda_i=\Lambda'_i.$
\ee
\end{proof}

\begin{lemma}\label{lemExCaseJ+}
There is an injective, affine $\bH_{GL}(k_0)$-equivariant CLF-map
\[\phi_{GL}:\ \Building(\bH_{GL},k_0)\ra\Building(\bG_{GL},k_0).\]
\end{lemma}

\begin{proof}
In terms of lattice functions we have 
\[\Building(\bH_{GL},k_0)=\Building(\tibH_+,k)\]
 and we take a map
\[\phi_{GL,1}:\Building(\tibH_+,k)\ra\Building(\tibG_+,k)\]
constructed as in the proof of theorem \ref{thmGenThmBS10.3}. The map 
\[\phi_{GL,2}:\ \Building(\tibG_+,k)\ra\Building(\bG_{GL},k_0)\]
is obtained by proposition \ref{propAffEmb}. The CLF property is an easy calculation knowing that
\[\Lie(\tibG_+)(k)\hra\Lie(\bG_{GL})(k_0)\]
is given by
\[a\mapsto a\oplus (-a^{\sigma,V_-}).\]
$\phi_{GL,2}$ is injective, affine and $\tibG_+(k)$-equivariant.
We put 
\[\phi=\phi_{GL,2}\circ\phi_{GL,1}.\]
\end{proof}
\vspace{1em}

The combination of all lemmas provides the proof of part one of theorem \ref{thmExistence}.

\begin{lemma}\label{lemExOfACLFMap}
There is an injective, affine and $\bH(k_0)$-equivariant CLF-map
\[j:\ \Building(\bH,k_0)\ra\building(\bG,k_0)\]
\end{lemma}

\begin{proof}
We define 
\[j:=\psi\circ (\phi_{un}\times \phi_{GL}),\]
where $\phi_{GL}:=\phi_{GL,2}\circ\phi_{GL,1}.$
\end{proof}
\vspace{1em}

In the proof above many choices have been made. The following lemma finishes the proof of theorem \ref{thmExistence}. 

\begin{lemma}\label{lemMapsWithImageETFP}
Let $j$ be a map from $\Building(\bH,k_0)$ to $\building(\bG,k_0)$
constructed as in the proof of \ref{lemExOfACLFMap}.
The image of $j$ is the set of self-dual $o_E$-$o_D$-lattice functions.
\end{lemma}

\begin{proof}
\textbf{Case 1:} At first we consider the case where $J_{un+}$ has exactly one element. 

\textbf{Case 1.1:} If the index lies in $J_{un}$ the map $j$ is a restriction of $j^{E}$ 
 whose image is the set of classes of $o_E$-$o_D$-lattice functions by  
\ref{thmBLII.1.1}. 
Let $y$ be an element of $\building(\bG,k_0)$ whose self-dual lattice function is an $o_E$-$o_D$-lattice function. By the surjectivity of
$j^{E}$ there is an $x\in\Building(\tibH,k)$ such that
$j^{E}(x)=y$ and therefore 
\[\LF(y,\End_D(V))\cap\End_{\tens{E}{k}{D}}(V)=\LF(x,\End_{\tens{E}{k}{D}}(V)).\] 
The self-duality of $\LF(y,\End_D(V))$ implies the self-duality of the Lie algebra filtration 
of $x,$ i.e. $x$ has to be an element of $\Building(\bH,k_0).$

\textbf{Case 1.2:} If the unique element of $J_{un+}$ is an element of 
$J_+,$ then the map $j$ is a composition of 
\[(\Lattone{o_D}{V_i})^{E_i^{\times}}\liso (\Lattone{h}{V})\cap(\Lattone{o_E-o_D}{V})\]
from proposition \ref{propAffEmb}
and 
\[\tilde{j}:\ \Building(\bH,k_0)=\Building((\tibG_i)_{E_i},k)\liso\Building(\tibG_i,k)^{E_i^\times}.\]
Thus the image of $j$ is the set of self-dual $o_E$-$o_D$-lattice functions of $V.$

\textbf{Case 2:} In the general case $j$ is the direct sum of maps of the kind of the two cases 1.1 and 1.2 which finishes the proof.
\end{proof}


\chapter{Uniqueness results}\label{chUniquenessResults}

\begin{notation}
We take the notation and assumptions of subsection 
\ref{subsecForUh} and convention \ref{conbGnotO2is}
\end{notation}

In theorem \ref{thmExistence}
we proved the existence of an affine and $\bH(k_0)$-equivariant 
CLF-map \[\Building(\bH,k_0)\ra\building(\bG,k_0).\]
In this chapter we observe in which sense the CLF property
determines such a map if we forget the affineness or weaken the equivariance. It
will force uniqueness if $\bH$ has no factor $k_0$-isomorphic to
$\bO_{2,k_0}^{is}$ and $J_{GL}$ is empty.
It will force a uniqueness up to translations of $\Building(\bH,k_0)$ in
general. For technical reasons we introduce further notation.

\begin{notation}
For $i\in J_{un+}$ we put
\bi
\item $h_i:=h|_{(V_i+V_{-i})\times(V_i+V_{-i})},$\idxS{$h_i$}
\item $\bG_i:=\bU(h_i)$ and
$\bH_i:=(\bG_i)_{E_i+E_{-i}},$\idxS{$\bG_i$}\idxS{$\bH_i$}
\item $\tibG_i:=\bGL_D(V_i)$ and
$\tibH_i:=(\tibG_i)_{E_i}.$\idxS{$\tibG_i$}\idxS{$\tibH_i$}
\ei
\end{notation}

To shorten the notation we write  
\bi
\item $\tilde{\tf{g}}$ for $\Lie(\tibG)(k),$ \idxS{$\tilde{\tf{g}}$}
\item $\tf{g}$ for $\Lie(\bG)(k_0),$ \idxS{$\tf{g}$}
\item $\tilde{\tf{h}}$ for $\Lie(\tibH)(k)$ and  \idxS{$\tilde{\tf{h}}$}
\item $\tf{h}$ for $\Lie(\bH)(k_0).$\idxS{$\tf{h}$}
\ei
For any index $i$ we denote
\bi
\item $\Lie(\tibG_i)(k)$ by $\tilde{\tf{g}}_i,$\idxS{$\tilde{\tf{g}}_i$}
\item $\Lie(\bG_i)(k_0)$ by $\tf{g}_i,$\idxS{$\tf{g}_i$}
\item $\Lie(\tibH_i)(k)$ by $\tilde{\tf{h}}_i$ and \idxS{$\tilde{\tf{h}}_i$}
\item $\Lie(\bH_i)(k_0)$ by $\tf{h}_i.$\idxS{$\tf{h}_i$}
\ei

We have to remark the following correspondence. For $i\in J_+$ we have 
\[(\bH_i)(k_0)=\{a+(a^{-1})^{\sigma}|\
a\in\Aut_{\tens{E_i}{k}{D}}(V_i)\}\cong\Aut_{\tens{E_i}{k}{D}}(V_i)=\tibH_i
(k),\]
\[\tf{h}_i=\{a-(a)^{\sigma}|\
a\in\End_{\tens{E_i}{k}{D}}(V_i)\}\cong\End_{\tens{E_i}{k}{D}}(V_i)=\tilde{\tf{h
}}_i\]
thus a Lie algebra filtration in $\tilde{\tf{h}}_i$ corresponds to a Lie algebra
filtration in $\tf{h}_i.$
The Lie algebra filtrations are defined in  \ref{defLieAlgebraFiltrationGLDV},
\ref{defLieAlgebraFiltrationUh} and \ref{defLFForProductsInTheUnitaryCase}.

\section{Factorisation}

\begin{lemma}\label{lem1betai0}
There is at most one index $i\in J$ such that $\beta_i=0$
and if such an index exists it has to be in $J_{un}.$
\end{lemma}

\begin{proof}
Assume that there are two different indexes $i,j\in J$ such that $\beta_i$ and
$\beta_j$ are zero. We take a polynomial $P$ with coefficients in $k$ such that
$1_i=P(\beta).$
We obtain firstly
\[1_i=1_i1_i=1_iP(\beta)=1_iP(0),\]
i.e.
\[(1-P(0))1_i=0,\]
and secondly
\[0=1_i1_j=P(\beta)1_j=P(0)1_j.\]
The element $P(0)$ lies in $k$ and therefore 
it must be $1$ by first and $0$ by the second equality which gives a
contradiction. 

If there is one index with $\beta_i=0$ then $-\beta_{-i}=\beta_i^\sigma$ is zero
too, and 
by the above argument $i$ equals $-i$ and thus $i\in J_{un}.$ 
\end{proof}
\vspace{1em}

For this section, \textbf{let $y\in\building(\bG,k_0)$ be an extension
of $x\in\Building(\bH,k_0)$}, i.e.
\beq\label{eqIntersectionProperty}
\LF(x,\tf{h})=\LF(y,\tf{g})\cap \tf{h}.
\eeq

We want to show that $\LF(y,\tf{g})$ is a direct sum of Lie algebra filtrations
of $\tf{g}_i$ where $i$ runs over $J.$
The element $x$ is a vector of elements $x_i\in\Building(\bH_i,k_0)$ for $i\in
J_{un+}.$

\begin{lemma}\label{lemId}
The idempotents $1_i$ are elements of $\LF(y,\tilde{\tf{g}})(0).$ 
\end{lemma}

\begin{proof}
\textbf{Case 1:} We firstly consider an index $i\in J_{+}.$
$1_i-1_{-i}$ is an element of $\LF(x_i,\tf{h}_i)(0),$ thus an element of
$\LF(y,\tf{g})(0)$ by (\ref{eqIntersectionProperty}) and therefore 
\[1_i+1_{-i}=(1_i-1_{-i})^2\in\LF(y,\tilde{\tf{g}})(0).\]
Hence $1_i$ and $1_{-i}$ are elements of $\LF(y,\tilde{\tf{g}})(0)$ since $2$ is
invertible in $o_k.$ 

\textbf{Case 2:} We take an index $i\in J_{un}$ and we assume that 
$\beta_i$ is not zero. Since $\beta_i\in E_i$ is skewsymmetric we have for all
$t\in \bbR$
\[\beta_i\LF(x_i,\tilde{\tf{h}}_i)(t)=\LF(x_i,\tilde{\tf{h}}
_i)(t+\nu(\beta_i))\]
and  
\[\beta_i\LF(x_i,\tf{h}_i)(t)=\LF(x_i,\tilde{\tf{h}}
_i)(t+\nu(\beta_i))\cap\Sym(\tilde{\tf{h}}_i,\sigma_i).\]
By the invertibility of $2$ in $o_k$ every element of $\LF(x_i,\tilde{\tf{h}}_i)(t)$ is a sum of a skewsymmetric and a symmetric
element of $\LF(x_i,\tilde{\tf{h}}_i)(t).$ Thus we obtain 
\begin{align*}\label{eqnLFIncl}
\LF(x_i,\tilde{\tf{h}}_i)(0)&= \LF(x_i,\tf{h}_i)(0)+\beta_i\LF(x_i,\tf{h}
_i)(-\nu(\beta_i))\\
&\subseteq \LF(y,\tf{g})(0)+ \LF(y,\tf{g})(\nu(\beta_i))\LF(y,\tf{g})(-\nu(\beta_i))\\
&\subseteq \LF(y,\tilde{\tf{g}})(0)
\end{align*}
and the $i$th idempotent $1_i$ is an element of $\LF(y,\tilde{\tf{g}})(0).$

\textbf{Case 3:} If there is an index $i_0$ such that $\beta_{i_0}=0,$
by lemma \ref{lem1betai0} it is unique, and the two cases above imply
\[1_{i_0}=1-\sum_{i\neq i_0}1_i\in \LF(y,\tilde{\tf{g}})(0).\]
\end{proof}
\vspace{1em}

The idea of the proof of case 2 is taken from 
\cite[11.2]{broussousStevens:09}..

\begin{corollary}\label{corPsiJ}
The $o_D$-lattice function of $y$ splits under $(V_i)_{i\in J}$ 
and $y$ is in the image of the injective, affine and 
$\prod_{i\in J_{un+}}\bG_i(k_0)$-equivariant CLF-map 
\[\psi_J:\prod_{i\in J_{un+}}\Building(\bG_i,k_0)\ra\building(\bG,k_0)\] which
is defined by taking the direct sum of the self-dual lattice functions. 
\end{corollary}

The fixed element $y$ is in the image of $\psi_J,$ i.e.
\[y=\psi_J((y_i)_{i\in J_{un+}})\]
 for some 
\[(y_i)\in\Building(\prod_{i\in J_{un+}}\bG_i,k_0)=\prod_{i\in
J_{un+}}\Building(\bG_i,k_0).\]
We now prove property (\ref{eqIntersectionProperty}) for the coordinates.

\begin{lemma}\label{eqIntersectionPropertyComp}
For all $i\in J_{un+}$ we have 
\[\LF(x_i,\tf{h}_i)=\tf{h}_i\cap\LF(y_i,\tf{g}_i).\]
\end{lemma}

\begin{proof}
Let $t\in\bbR.$ We have
\begin{align*}
\prod_i\LF(x_i,\tf{h}_i)(t)&= \LF(x,\tf{h})(t)\\
&= \LF(y,\tf{g})(t)\cap \tf{h}\\
&= (\LF(y,\tf{g})(t)\cap \prod_i\tf{g}_i)\cap \tf{h}\\
&= (\prod_i\LF(y_i,\tf{g}_i)(t))\cap (\prod_i\tf{h}_i)\\
&= \prod_i(\LF(y_i,\tf{g}_i)(t)\cap\tf{h}_i).
\end{align*}
Thus we have for all indexes $i\in J_{un+}$ and for all 
$t\in\bbR$ the property
\[\LF(x_i,\tf{h}_i)(t)=\LF(y_i,\tf{g}_i)(t)\cap\tf{h}_i.\]
\end{proof}
\vspace{1em}

The last two lemmatas lead to a factorization of a CLF-map.
More precisely we can prove the following proposition.

\begin{proposition}\label{propFactOfCLFMap}
If $j$ is a CLF-map from $\Building(\bH,k_0)$ to 
$\building(\bG,k_0)$ there is a unique map
\[\tau:\ \Building(\bH,k_0)\ra\Building(\prod_{i\in J_{un+}}\bG_i,k_0)\text{
such that }j=\psi_J\circ\tau.\]
The map $\tau$ is
\be
\item a CLF-map,
\item affine if $j$ is affine, and
\item $\bH(k_0)$-equivariant if $j$ is $\bH(k_0)$-equivariant.
\ee
\end{proposition}

\begin{proof}
The value $j(x')$ of a point $x'$ is an extension of $x'$
and lies in the image of $\psi_J$ by corollary \ref{corPsiJ}.
The injectivity of $\psi_J$ implies the unique existence of $\tau.$
In addition to the injectivity the map $\psi_J$ is affine and $\prod_{i\in
J_{un+}}\bG_i(k_0)$-equivariant which implies 2 and 3.
\end{proof}

\begin{remark}\label{remRedOfJun+}
The proposition allows us to reduce proofs to the case where
$J_{un+}$ has only one element, i.e. where $E$ is a field or 
a product of two fields which are switched by $\sigma.$
The first case corresponds to a non-empty $J_{un}$ and the 
second case to a non-empty $J_+.$
\end{remark}


\section{\texorpdfstring{Uniqueness if $J_{GL}$ is empty}{Uniqueness if JGL is empty}}

\begin{theorem}\label{thmUniquenessForJGLIsEmpty}
Assume that $J_{GL}$ is empty and that no $\bH_i$ is $k_0$-isomorphic to
$\bO_{2,k_0}^{is}.$ There is exactly one CLF-map 
$j^\beta$ from $\Building(\bH,k_0)$ to $\building(\bG,k_0).$
Indeed  we have the stronger result that
$j^\beta(x)=y$ if $y$ is an extension of $x.$
\end{theorem}

By remark \ref{remRedOfJun+} it is enough to prove the theorem for the case
where $E$ is a field. We only have to ensure that $\bO_{2,k_0}^{is}$ does not
occur among the $\bG_i.$

\begin{lemma}\label{lemO2isDoesNotOccur}
Under the assumptions of theorem \ref{thmUniquenessForJGLIsEmpty}
no group $\bG_i$ is $k_0$-isomor\-phic to $\bO_{2,k_0}^{is}.$
\end{lemma}

\begin{proof}
If $\bG_i$ is $k_0$-isomorphic to $\bO_{2,k_0}^{is}$ then $\beta_i$ has to be
zero by remark \ref{remCalcisO2}[3.] because $E_i$ is stable under $\sigma,$
i.e. $\bH_i$ equals $\bG_i$ and is $k_0$-isomorphic to $\bO_{2,k_0}^{is}$ which
is excluded by the assumption of the theorem.
\end{proof}
\vspace{1em}

Theorem \ref{thmUniquenessForJGLIsEmpty} will be proven by two steps. 

\begin{proposition}\label{propBS11.2}(Compare with
\cite[11.2]{broussousStevens:09})
Theorem \ref{thmUniquenessForJGLIsEmpty} is true if $E$ is a field and $\beta$
is not zero.
\end{proposition}

\begin{proof}
If $y$ is an extension of $x$ we have by the same argument as in 
case 2 in the proof of lemma \ref{lemId} that
\[ \LF(x,\tilde{\tf{h}})(t)\subseteq \LF(y,\tilde{\tf{g}})(t). \]
We now apply theorem \ref{thmBS10.3} and obtain 
\[y=j^E(x)\]
if we consider $x$ as an element of $\building(\tibH,E)$
and $y$ as an element of $\building(\tibG,k).$
Thus for every $x\in\Building(\bH,k_0)$ there is only one extension
in $\building(\bG,k_0).$
\end{proof}

\begin{lemma}\label{lemBeta=0}
The theorem is true if $\beta$ is zero.
\end{lemma}

For the proof we need the following operation on square matrices.

\begin{definition}\label{defTildeOpSqMatr}
For a square matrix $B=(b_{i,j})\in M_r(D)$ the matrix $\tilde{B}$ is defined to
be $(b_{r+1-j,r+1-i})_{i,j},$ i.e. $\tilde{B}$ is obtained from $B$ by a
reflection on the antidiagonal.
\end{definition}

\begin{proof}
If $\sigma$ is of the second type there is a skewsymmetric non-zero element
$\beta'$ in $k$ and we can replace $\beta$ by $\beta'$ and apply theorem
\ref{propBS11.2}. 
Thus we only need to consider $\sigma$ to be of the first kind.
We fix a point $y\in\building(\bG,k_0)$ and fix an apartment 
containing $y.$ This apartment is determined by a Witt decomposition
and thus determined by a Witt basis $(w_i)$ by corollary
\ref{corExistOfWittBasis}. The self-dual $o_D$-lattice function $\Lambda$
corresponding to $y$ is split by this basis and is thus described by its
intersections with the lines $w_iD,$
i.e. there are real numbers $\alpha_i$ such that 
\[\Lambda(t)=\bigoplus_iv_i\textfrak{p}_{D}^{[(t-\alpha_i)d]+}.\]
Thus the square lattice function of $\Lambda$ is 
\[\End(\Lambda)(t)=\bigoplus_{i,j}\textfrak{p}_{D}^{[(t+\alpha_j-\alpha_i)d]+}
E_{i,j}\] 
where $E_{i,j}$ denotes the matrix with a 1 in the intersection of the ith row
and the jth column and zeros everywhere else. See for example
\cite[I.4.5]{broussousLemaire:02}.

What we have to show is that $\End(\Lambda)$ is determined by the Lie algebra
filtration $\LF(y,\tf{g}).$ This is enough since the self-dual square lattice
function of a point determines the point uniquely.
The Gram matrix $\Gram_{(v_i)}(h)$ of the $\epsilon$-hermitian form $h$ has the
form 
\begin{displaymath}
\left(\begin{array}{ccc}
0 & M & 0\\
\epsilon M & 0 & 0\\
0 & 0 & N\\
\end{array}\right)
\end{displaymath}
with $M:=\antidiag(1,\ldots,1)$ and a diagonal regular matrix $N.$
The adjoint involution of $h$  
\[B\mapsto B^{\sigma}=\Gram_{(v_i)}(h)^{-1}(B^\rho)^T \Gram_{(v_i)}(h)\] on
$\Matr_m(D)$ has under this basis the form 
\begin{displaymath}
\left(\begin{array}{ccc}
B_{1,1} & B_{1,2} & B_{1,3}\\
B_{2,1} & B_{2,2} & B_{2,3}\\
B_{3,1} & B_{3,2} & B_{3,3}\\
\end{array}\right)\mapsto 
\left(\begin{array}{ccc}
\tilde{C}_{2,2} & \epsilon \tilde{C}_{1,2} & \epsilon M C_{3,2}^TN\\
\epsilon \tilde{C}_{2,1} & \tilde{C}_{1,1} & M C_{3,1}^T N\\
\epsilon N^{-1} C_{2,3}^TM & N^{-1} C_{1,3}^T M & N^{-1} C_{3,3}^T N\\
\end{array}\right).
\end{displaymath}
The matrices $B_{1,1},B_{1,2},B_{2,1}$ and $B_{2,2}$ are $r\times r$-matrices
and $C:=B^\rho$ where $r$ is the Witt index of $h.$
By the above calculation we obtain that 
$E_{i,j}^{\sigma}$ is $+E_{i,j},$ $-E_{i,j}$ or $\lambda E_{u,l}$ with
$(i,j)\neq (u,l)$ for some $\lambda\in D^\times.$ From the self-duality of 
$\End(\Lambda)$ and since 2 is invertible in $o_k$ we get:
\[\LF(y,\tf{g})(t)\cap k(E_{i,j}-E_{i,j}^{\sigma})
=\tf{p}_k^{[t+\alpha_j-\alpha_i]+}(E_{i,j}-E_{i,j}^{\sigma}).\]
For the calculation see the lemma below. Thus we can get the ex\-po\-nent
$\alpha_j-\alpha_i$ from the know\-ledge of the Lie algebra fil\-tration if
$E_{i,j}$ is not fixed by $\sigma.$ We now con\-si\-der two ca\-ses.\\
\textbf{Case 1:} We assume that there is an anisotropic part in the Witt
de\-com\-po\-si\-tion, i.e. $N$ occurs. The matrix $E_{i,m}$ is fixed by
$\sigma$ if and only if $i$ equals $m.$
Thus from the know\-ledge of the Lie al\-ge\-bra fil\-tra\-tion we know all
dif\-fe\-ren\-ces $\alpha_i-\alpha_{m}$ for all in\-de\-xes $i$ dif\-fe\-rent
from $m,$ and thus by sub\-trac\-tions we know the dif\-fe\-ren\-ces
$\alpha_i-\alpha_j$ for all $i$ and $j.$\\
\textbf{Case 2:} Now we assume that there is no an\-iso\-tro\-pic part in the Witt
de\-com\-po\-si\-tion. If $\epsilon$ is $-1,$ no $E_{i,j}$ is fixed and we can
de\-duce the differences $\alpha_i-\alpha_j$ for all $i$ and $j$ and, as a
consequence, we only have to consider the case where $h$ is hermitian and $D=k.$
Here the matrix $E_{i,j}$ is fixed by $\sigma$ if and only if $i+j=m+1.$ Thus we
can determine all differences $\alpha_i-\alpha_j$ where $i+j\neq m+1.$ If $m$ is
at least $4$ for an index $i$ there is an index $k\neq i$ with $i+k\neq m+1$ and
we can obtain $\alpha_i-\alpha_{m+1-i}$ if we substract
$\alpha_k-\alpha_{m+1-i}$ from $\alpha_i-\alpha_k.$
The only subcase left is when $m$ equals $2$ and $\epsilon$ is 1. Here the group
$\bG$ is 
$k$-isomorphic to $\bO^{is}_{2,k}$ which is excluded by the assumption of the
theorem. 
\end{proof}

To complete the proof we need the following lemma.

\begin{lemma}
For all $t\in\bbR$ we have
\[\tf{p}_D^{[td]+}\cap k=\tf{p}_k^{[t]+}.\]
\end{lemma}

\begin{proof}
For an element $x$ of $k$ we have:
\[x\in\tf{p}_D^{[td]+}\text{ if and only if }\]
\[\nu(x)\geq\frac{[td]+}{d}\]
There are integers $l$ and $k$ such that 
\[[td]+=ld+k\text{ and }1\leq k\leq d,\]
and thus $[t]+= l+1$ and we get that
\[ \nu(x)\geq\frac{[td]+}{d}\text{ if and only if }\nu(x)\geq [t]+.\]
The  "only if" follows from $\nu(x)\in\bbZ.$
\end{proof}

The proof of theorem \ref{thmUniquenessForJGLIsEmpty} follows now from 
lemmas \ref{lemO2isDoesNotOccur} and \ref{lemBeta=0} and proposition
\ref{propBS11.2}. 

\begin{corollary}
For an index $i\in J_{un}$ the following statements are equivalent.
\be
\item $\bH_i$ is $k_0$-isomorphic to $\bO_{2,k_0}^{is}.$
\item \[\beta_i=0,k=k_0=D,\dim_kV_i=2,\epsilon=1\]
and the Witt index of $h_i$ is 1.
\item $\bG_i$ is $k_0$-isomorphic to $\bO_{2,k_0}^{is}.$
\item There are at least two CLF-maps from $\Building(\bH_i,k_0)$
to $\Building(\bG_i,k_0).$ 
\item There are infinitely many CLF-maps from $\Building(\bH_i,k_0)$
to $\Building(\bG_i,k_0).$ 
\ee
\end{corollary}

\begin{proof}
That 1. follows from 3. is a consequence of lemma \ref{lemO2isDoesNotOccur}.
From 1. follows 5. because we have infinitely many translations of
$\Building(\bO^{is}_{2,k_0},k_0).$
5. implies 4.. We did not use that $\bG$ is not $\bO_{2,k_0}$ for the proofs of
lemma \ref{lemBeta=0} and theorem \ref{propBS11.2}. Thus we obtain from 4. the
statements 1., 2. and 3..
3. follows from 2. obviously. We summarise:
\[3.\Ra 1.\Ra 5.\Ra 4.\Ra 2.\Ra 3..\]
\end{proof}


\section{The image of a CLF-map}

\begin{proposition}\label{propCLFETFP}
The image of a CLF-map from $\Building(\bH,k_0)$ to $\building(\bG,k_0)$ is a
subset of the set of $o_E$-$o_D$-lattice functions.
\end{proposition}

\begin{proof}
By proposition \ref{propFactOfCLFMap} we can assume that 
\[J_{un+}=\{i\}.\] 

\textbf{Case 1:} We assume $i\in J_{un}.$ If $\beta$ is zero we have $E=k$
and therefore the ${E^\times}$-action is trivial. If $\beta$ is non-zero
there is only one CLF-map by theorem \ref{thmUniquenessForJGLIsEmpty}
and it fullfils the assertion by theorem \ref{thmExistence}.

\textbf{Case 2:} We assume $i\in J_+.$ Let $y\in\building(\bG,k_0)$ be an extension
of $x\in\Building(\bH,k_0).$ The lattice function $\Lambda$ of $y$ splits under
$(V_i,V_{-i})$ by corollary \ref{corPsiJ} and by the self duality we only have
to pove that $\Lambda\cap V_i$ is an $o_{E_i}$-lattice function. The building
\[\Building(\bH,k_0)=\Building(\bGL_{\tens{E_i}{k}{D}}(V_i),E_i)\]
is identified with the set of lattice functions over a skewfield whose center is
$E_i.$ Thus we get
\bi
\item $a-a^{\sigma}\in\LF(x,\tf{h})(0)\subseteq \LF(y,\tf{g})(0)$ for all $a\in
o_{E_i}^\times,$
\item $\pi_{E_i}-\pi_{E_i}^{\sigma}\in \LF(x,\tf{h})(\frac{1}{e})\subseteq
\LF(y,\tf{g})(\frac{1}{e})$ and 
\item $\pi_{E_i}^{-1}-(\pi_{E_i}^{-1})^{\sigma}\in
\LF(x,\tf{h})(-\frac{1}{e})\subseteq \LF(y,\tf{g})(-\frac{1}{e})$
\ei
where $e$ is the ramification index of $E_i|k$ and $\pi_{E_i}$ is a prime
element of $E_i.$ We conclude that $1_i\Lambda$ is an $o_{E_i}$-lattice
function.
\end{proof}


\section{Rigidity of Euclidean buildings}

\begin{definition}
Let $S$ be a set with affine structure. An \textit{affine
functional}\idxD{affine functional} $f$ on 
$S$ is an affine map from $S$ to $\bbR,$ i.e.
\[f(tx+(1-t)y)=tf(x)+(1-t)f(y)\]
for all $t\in [0,1]$ and $x,y\in S.$
\end{definition}

We analyse affine functionals on the buildings $\building(\bG,k_0)$ and
$\Building(\tibG,k).$ At first we give the general statement.

\begin{proposition}\label{propRigidityForAffineBuildings}
Let $\Omega$ be a thick Euclidean building and $|\Omega|$ be its geometric
realisation, then every affine functional $a$ on $|\Omega|$ is constant.
\end{proposition}

For the definition of a Euclidean building and its geometric realisation
see \cite[VI.3]{brown:89} or chapter \ref{chapterEuclideanBuildingOfGLDV} in
part 2.
 We use the following properties of a thick Euclidean building in the next
proof.

\begin{remark}\label{remPropertiesOfAThickEuclideanBuilding}
\be
\item A building is a chamber complex, especially two arbitrary chambers are
connected by a gallery.    
\item The thickness, i.e. at every corank $1$ face $S$ there are at least three
different chambers which have $S$ as a common subface.
\item The geometric realisation of a Euclidean building of rank $r$ has an
affine structure and the geometric realisation of an apartment is affine
isomorphic to $\bbR^{r-1}.$
\item For two arbitrary faces there is an apartment containing them.
\item If $\Sigma$ and $\Sigma'$ are two apartments containing a chamber $C$
there is an isomorphism of simplicial complexes from $\Sigma$ to $\Sigma'$ which
fixes the intersection of $\Sigma$ and $\Sigma'.$ It induces an affine
isomorphism between the geometric realisations.
\ee
\end{remark}

\begin{proof}(of \ref{propRigidityForAffineBuildings})
Assume that we are given three vertices $P_1,\ P_2$ and $P_3$ of adjacent
chambers $C_1,\ C_2$ and $C_3,$ more precisely the three chambers have a common
codimension 1 face $S$ and the vertex $P_i\in C_i$ does not lie on $S.$ 
The line segment $[P_1,P_2]$ meets $[P_1,P_3]$ and $[P_2,P_3]$ in a point $Q\in
S.$ This is a consequence of \ref{remPropertiesOfAThickEuclideanBuilding}[4.,
5.] as follows.  
We are working in three different apartments simultaneously. If $\Delta_{ij}$
denotes an apartment 
containing $C_i$ and $C_j,$  for different $i$ and $j,$ the affine isomorphism 
from $|\Delta_{12}|$ to $|\Delta_{13}|$ fixing $|\Delta_{12}\cap\Delta_{13}|$ sends
$[P_1,P_2]$ to $[P_1,P_3]$ and 
thus the unique intersection point in $[P_1,P_2]\cap |\bar{C_1}|\cap |\bar{C_2}|$
lies on $[P_1,P_3],$ and analogously on $[P_3,P_2].$  
Without loss of generality assume that $a(Q)$ vanishes. If $a(P_1)$ is negative
then $a(P_2)$ and $a(P_3)$ are positive by the affiness of $a.$ Thus 
$a(Q)$ is positive since it lies on $[P_2,P_3].$ A contradiction. Using
galleries we obtain that $a$ is constant on vertices of the same type. An
apartment is affinely generated by its vertices of a fixed type. Thus $a$ is
constant on every apartment and therefore on $|\Omega|.$ 
\end{proof}
\vspace{1em}

We remind again that $\bG$ is not $k_0$-isomorphic to $\bO_{2,k_0}^{is}.$

\begin{proposition}\label{propRigidityForClassicalGroups}
Every affine functional on $\building(\bG,k_0)$ is constant. 
\end{proposition}

\begin{proof}
If $\bG$ is totally isotropic then $\building(\bG,k_0)$ is a point and otherwise
it is the geometric realisation of a thick Euclidean building. Now we apply
proposition \ref{propRigidityForAffineBuildings}.  
\end{proof}

\begin{proposition}\label{propRigidityForGLnSO2}
\be
\item A $k^\times$-invariant affine functional $a$ on $\Building(\tibG,k)$ is
constant. 
\item Every $k^\times$-invariant affine functional on
$\Building(\bO_{2,k}^{is},k)$ is constant. 
\ee
\end{proposition}

\begin{proof}
\be
\item We can consider $a$ as a map on $\building(\tibG,k),$ because the fibers
of $a$ are unions of classes of $o_D$-lattice functions. Now we apply
proposition \ref{propRigidityForAffineBuildings}.
\item It follows from part 1, because
\[\Building(\bO_{2,k}^{is},k)\cong\Building(\bG_m,k)\] by a
$k^\times$-equivariant affine bijection. 
\ee
\end{proof}


\section{Uniqueness in the general case}\label{secUniquenessInTheGeneralCase}

In this section we want to generalise theorem \ref{thmUniquenessForJGLIsEmpty}
to the case where there are no restrictions on $J.$ CLF-maps can differ by
translations in the following sence.

\begin{definition}\label{defTranslation}
\be
\item Fix a natural number $n$ and a real number $s.$ A \textit{
translation}\idxD{translation of a Bruhat-Tits building} of
$\Building(\tibG,k)$ by $s$ is a map 
\[t:\ \Building(\tibG,k)\ra\Building(\tibG,k)\]
defined by \[t(\Lambda):=\Lambda+s\]
in terms of $o_D$-lattice functions of $V.$
 Here $\Lambda+s$ denotes the lattice function 
 \[r\mapsto \Lambda(r-s).\] This also defines translations
 on $\Building(\bO_{2,k}^{is},k)\cong\Building(\bG_m,k).$
\item We only call the identity of $\building(\bG,k_0)$ a \textit{translation} of
$\building(\bG,k_0).$
\item A \textit{translation} of $\Building(\bH,k_0)$ is a product
of translations $t_i$ of $\Building(\bH_i,k_0)$ where $i$ runs over
$J_{un+}.$
\ee
\end{definition}

\begin{remark}
A translation of $\Building(\bH,k_0)$ is $\bH(k_0)$-equivariant if there is no
$\bH_i$ $k_0$-iso\-mor\-phic to $\bO^{is}_{2,k_0}.$ Otherwise we get $k=k_0$ and  
such a translation is 
\[(\prod_{i\in J_{un+},
\bH_i\stackrel{/k}{\not\cong}\bO^{is}_{2,k}}\bH_i(k))\times
(\bO^{is}_{2,k})^0(k)-\] and especially $\bH^0(k)$-equivariant, but in general
not $\bH(k)$-equivariant, see \ref{remUniqO2isEquivTranslation}.
\end{remark}
	
Let $j$ be a map from $\Building(\bH,k_0)$ to $\building(\bG,k_0)$ constructed
as in the proof of theorem \ref{thmExistence}.

\begin{theorem}\label{thmCLFUnitaryCase}
If $\phi$ is an affine and $\Centr(\bH^0(k_0))$-equivariant CLF-map from\\
$\Building(\bH,k_0)$ to $\building(\bG,k_0)$ then $j^{-1}\circ\phi$ is a
translation of $\Building(\bH,k_0).$ In  terms of lattice functions the image of
$\phi$ is the set of self dual $o_E$-$o_D$-lattice functions on $V$ and $\phi$
is $\bH^0(k_0)$-equivariant. 
\end{theorem}

\begin{proof} The image of $\phi$ is a subset of the image of $j$ which is the
set of selfdual $o_E$-$o_D$-lattice functions by \ref{propCLFETFP} and
\ref{thmExistence}, especially $\tau:=j^{-1}\circ\phi$ is well-defined.
We prove in the lemmas below that $\tau$ is a translation. A translation is a
bijection and we conclude that $\phi$ and $j$ have the same image. 
The $\bH^0(k_0)$-equivariance of $\phi$ follows because $j$ and $\tau$ are
$\bH^0(k_0)$-equivariant. 
\end{proof}
\vspace{1em}

We work with the notation of the theorem and its proof. 
The coordinates of $\tau$ are denoted by $\tau_i,$ $i\in J_{un+}.$ 

\begin{lemma}\label{lemCoordinatesDependOnxi}
The coordinate $\tau_i$ only depends on $x_i.$
For all $i\in J_{un}$ for which $\bO_{2,k_0}^{is}$ is not $k_0$-isomorphic to
$\bH_i$ we have $\tau_i(x)=x_i$ for all $x\in\Building(\bH_i,k_0).$ 
\end{lemma}

\begin{proof}
We have to look at three cases.\\
\textbf{Case 1:} For the indexes $i$ in $J_{un}$ for which ${\bH}_i$ is not
$k_0$-isomorphic to $\bO_{2,k_0}^{is}$ we know by theorem
\ref{thmUniquenessForJGLIsEmpty} that $\tau_i(x)$ equals $x_i.$\\
\textbf{Case 2:} We assume that we have an index $i\in J_{un}$ such that $\bH_i$ is
$k_0$-isomorphic to $\bO_{2,k_0}^{is}.$ In this case we have $k=k_0.$  A lattice
function $\Lambda$ corresponding to a point of the building
$\Building(\bO_{2,k},k)$ is identified with a real number,
see remark \ref{remIdO2}.
 If we fix an index $t\in J\setminus\{i\}$ and coordinates $x_l$ for $l\in
J\setminus\{t\}$ then the map 
\[x_t\mapsto \tau_i(x)\]
is constant by proposition \ref{propRigidityForClassicalGroups} or
\ref{propRigidityForGLnSO2} and thus $\tau_i$ does not depend on $x_t.$\\
\textbf{Case 3:} In the case of $i\in J_+$ an analogous argument like in case 2
applies.
The affine map we use is the map $a_i$ defined by 
\[\Lambda_{\tau_i(x)}=\Lambda_{x_i}+a_i(x).\]
\end{proof}
\vspace{1em}

The last lemma allows us to define a map $\tilde{\tau}_i$ by
\[\tilde{\tau}_i(x_i):=\tau_i(x),\ x\in \Building(\bH,k_0).\] 

\begin{lemma}\label{lemCoordinatesAretranslations}
The map $\tilde{\tau}_i$ is a translation.
\end{lemma}

\begin{proof}
We firstly consider an index $i\in J_{un}$ such that $\bH_i$ is $k_0$-isomorphic
to $\bO_{2,k_0}^{is}.$ We identify $\Building(\bO_{2,k_0}^{is},k_0)$ with
$\bbR.$ In this case we have $k=k_0$ and the $\SO_2(k)$-equivariance of $\tau_i$
gives
\[\tilde{\tau}_i(x_i+1)=\tilde{\tau}_i(x_i)+1.\] 
The affineness property implies that $\ti{\tau}_i$ is a translation.
For $i\in J_+$ the map $a_i$ in case 3 of the preceding proof is an affine
functional and the $k^{\times}$-equivariance of $\tau_i$ implies the
$k^{\times}$-invariance of $a_i,$ because one gets in terms of lattice functions
\begin{align*}
\pi_k\Lambda+a_i(\pi_k\Lambda)&= \ti{\tau}_i(\pi_k\Lambda)\\
&= \pi_k\ti{\tau}_i(\Lambda)\\
&= \Lambda+a_i(\Lambda)-1\\
&= \pi_k\Lambda+a_i(\Lambda)
\end{align*}.
Thus $a_i$ is constant by proposition \ref{propRigidityForGLnSO2}.
\end{proof}


\section{Generalisation to the non-separable case}\label{secGeneralisationToTheNonseparableCase}

All theorems and propositions of the preceding sections of chapter
\ref{chMapsWhichAreCLF} and \ref{chUniquenessResults} work if we forget all the
separability assumptions, but we have to explain the definition of the enlarged
Bruhat-Tits building of the centraliser. This definition was introduced in
\cite{broussousStevens:09}.

\textbf{Case 1:} We firstly summarise changes of subsection \ref{subsecForGLDV}. We
assume that $E$ is commutative, semisimple and not separable over $k.$
We define 
\bi
\item $\building(\tibG_E,k)$\idxS{$\building(\tibG_E,k)$ for $E$ not separable
over $k$} (resp. $\Building(\tibG_E,k)$)\idxS{$\Building(\tibG_E,k)$ for $E$ not
separable over $k$} to be the 
product (\ref{eqbuildProdGL}) (resp. (\ref{eqBuildProdGL})),
\item $\tibG_E(k):=\tibG(k)\cap \tibG_E$ and\idxS{$\tibG_E(k)$ for $E$ not
separable over $k$}
\item $\Lie(\tibG_E)(k):=\Centr_{\Lie(\tibG)(k)}(E).$\idxS{$\Lie(\tibG_E)(k)$
for $E$ not separable over $k$}
\ei

\textbf{Case 2:}
We now come to the case of a unitary group, i.e. we come to subsection
\ref{subsecForUh}. 
Let us assume that $k[\beta]$ is semisimple but not separable over $k.$ 
In this case $\bH:=\bG_{\beta}$ is well defined but not reductive. We
define 
\bi
\item $\Building(\bH,k_0)$ to be the product \idxS{$\Building(\bH,k_0)$ for
$\beta$ not separable}
\beq\prod_{i\in J_{un}} \Building (\bU (\sigma|_{\End_{\tens
{E_i}{k}{D}}(V_i)}),(E_i)_0)\times \prod_{i>0}\Building (\bGL_{\tens
{E_i}{k}{D}}(V_i),E_i).
\eeq
\item $\bH(k_0):=\bG(k_0)\cap \bH,$\idxS{$\bH(k_0)$ for $\beta$ not separable}
$\bH^0(k_0):=\bG(k_0)\cap \bH^0$ and 
\item $\Lie(\bH)(k_0):=\Centr_{\Lie(\bG)(k_0)}(\beta).$\idxS{$\Lie(\bH)(k_0)$
for $\beta$ not separable}
\ei

As in \ref{defLFInProduct} and \ref{defLFForProductsInTheUnitaryCase} we define
the Lie algebra filtration of a point $x=(x_i)_i$ as the direct sum of the Lie
algebra filtrations of the $x_i.$ Also for the non-separable case we have the
definition of a CLF-map. A map $j$ between a subset of the (enlarged) building
of $\tibG_E(k)$ and a subset of the (enlarged) building of $\tibG(k)$ is a \textit{
CLF-map}\idxD{CLF-map} if for every element $x$ of the first and $y$
of the second building with $j(x)=y$ or $j(y)=x$ the equality 
\[\LF(y,\tibG,k)(t)\cap\Lie(\tibG_E)(k)=\LF(x,\tibG_E,k)(t)\]
holds for all $t\in\bbR.$ Analogously for the unitary case.

\begin{theorem}\label{thmGenToTheNonSeparableCase}
The theorems \ref{thmExistence}, \ref{thmUniquenessForJGLIsEmpty} and
\ref{thmCLFUnitaryCase} are still valid 
if one assumes $k[\beta]$ to be semisimple but not necessarily separable over
$k.$
\end{theorem}

\begin{proof}
The proofs of the mentioned theorems are valid without changes. 
\end{proof}

\chapter{Torality}\label{chapterTorality}

In this chapter we use the notation of subsection \ref{subsecForUh} but we skip the assumption
that $\beta$ is separable. We only assume that $E$ is semisimple over $k$ and we apply the conventions and definitions of section \ref{secGeneralisationToTheNonseparableCase} in the case that $\beta$ is not separable. 

\begin{definition}
A map 
\[f:\ \Building(G_1,k_0)\ra\Building(G_2,k_0)\]
between two enlarged buildings of reductive groups defined over $k_0$ is called \textit{toral}\/ if for each maximal $k_0$-split
torus $S$ of $G_1$ there is a maximal $k_0$-split torus $T$ of $G_2$ containing $S$ such that
$f$ maps the apartment corresponding to $S$ into the apartment corresponding to $T.$ 
An analogous definition applies to maps between non-enlarged buildings.
\end{definition}

\begin{proposition}\label{propToralityOfj}
The map $j$ constructed in the proof of theorem \ref{thmExistence} maps apartments into apartments. Further $j$ is toral if $\beta$ is separable.
\end{proposition}

The proof is divided into two cases. Because of the construction of $j$ as a direct sum of maps it is enough to restrict 
to the following two cases.
\be
\item case 1: $J_{0+}=J_+=\{i\}$ and
\item case 2: $J_{0+}=J_0=\{i\}.$
\ee

\begin{proof}[Case 1]
We assume that $J_{0+}=J_+=\{i\}.$ By \cite[5.1]{broussousLemaire:02} the map $\phi_{GL,1}$ from $\Building(\bH,k_0)$ to 
\[\Building(\bGL_D(V_i),k)=\Building(\Res_{k|k_0}(\bGL_D(V_i)),k_0)\] mentioned in the proof of lemma \ref{lemExCaseJ+} maps apartments into apartments and further is toral if $E_i|k$ is separable.
We prove that the map $\phi_{GL,2}$ from $\Building(\Res_{k|k_0}(\bGL_D(V_i)),k_0)$ to $\Building(\bG,k_0)$ is toral. 
A maximal $k$-split torus $S$ of $\bGL_D(V_i)$ corresponds to a decomposition of $V_i$ in one-dimensional $k$-subspaces, i.e. there is a decomposition $V_i=\oplus_lV_{i,l}$ such that 
\[ S(k)=\{g\in\GL_D(V_i)|\ g(V_{i,l})\subseteq V_{i,l} \text{ for all } l\}.\]
Let $V_{-i,j}$ be the subspace of $V_{-i}$ dual to $V_{i,j}$ and let $T$ be the torus given by the decoomposition 
\[V=\oplus_j(V_{i,l}\oplus V_{-i,l}).\]
Under the canonical embedding of $\GL_D(V_i)$ into $\GL_D(V)$ 
\beq\label{eqEmbedGroup}
g\mapsto g\oplus (g^{\sigma,V_{-i}})^{-1}
\eeq
the set $S(k)$ is mapped into $T(k)$ and the image of the apartment of $S$ under 
\beq\label{eqEmbedBuild}
\Lambda\in\Lattone{o_D}{V_i}\mapsto\Lambda\oplus\Lambda^{\#,V_{-i}}\Lattone{h}{V}.
\eeq
is a subset of the apartment of $T.$ Let $S'$ and $T'$ be the maximal $k_0$-split 
subtori of $\Res_{k|k_0}(S)$ and $\Res_{k|k_0}(T)$ respectively. The set $S'(k_0)$ is mapped into \mbox{$(T'\cap\bG)(k_0)$} under 
(\ref{eqEmbedGroup}).  The image of 
(\ref{eqEmbedBuild}) only consists of selfdual lattice functions. Hence 
$\phi_{GL,2}$ seen as a map 
from $\Building(\Res_{k|k_0}(\bGL_D(V_i)),k_0)$ to $\Building(\bG,k_0)$ is toral.  
\end{proof}
\vspace{1em}

We make the following definition for the proof of proposition \ref{propToralityOfj} in case two.

\begin{definition}
Assume we have given a decomposition 
\[V=V'^+\oplus V'^-\oplus V'^{0}\]
such that $V'^+$ and $V'^-$ are maximal totally isotropic and $V'^+\oplus V'^-$ is orthogonal to $V'^{0}$ with respect to $h.$
A maximal $k_0$-split torus $T$ of $\bU(h)$ is \textit{adapted to } $(V'^+,V'^-,V'^0)$ if there is a Witt decomposition $(V'^k)$ corresponding to $T$ with anisotropic part $V'^0$ such that 
\[\oplus_{k>0}V'^k=V'^+\text{ and } \oplus_{k<0}V'^k=V'^-.\]
An apartment of $\Building(\bG,k_0)$ is adapted to 
$(V'^+,V'^-,V'^0)$ if every lattice function in this apartment is split by $(V'^+,V'^-,V'^0).$
\end{definition}

\begin{proof}[Case 2]
Here we assume $J_{0+}=J_0=\{i\}.$ Thus we have $E=E_i.$ There is a central division algebra $\Delta$ over $E$ and
 a finite dimensional right vector space $W$ such that  $\End_{\tens{E}{k}{D}}(V)$ is $E$-algebra isomorphic 
to $\End_{\Delta}(W).$ We identify the $E$-algebras $\End_{\tens{E}{k}{D}}(V)$ and $\End_{\Delta}(W)$ via a fixed isomorphism and we fix a signed hermitian form $h_E$ which corresponds to the restriction $\sigma_E$ 
of $\sigma$ to the $E$-algebra $\End_{\Delta}(W).$ Let $r$ be the Witt index of $h_E.$ We fix a decomposition 
\beq\label{eqDirSumW}
W=(W^+ \oplus W^{-})\oplus W^0
\eeq
such that $W^+$ and $W^{-}$ are maximal isotropic subspaces of $W$ contained in the orthogonal complement of $W^0.$
 Let $e_+, e_-$ and $e_0$ be the projections to the vector spaces $W^+,W^-$ and $W^0$ via the direct sum (\ref{eqDirSumW}). We define 
\[ V^+:=e^+V,\ V^-:=e^-V\text{ and }V^0:=e^0V.\]
Consider the following diagram.
\[\begin{matrix}
\Building(\bH,k_0) & \stackrel{j}{\ra} & \Building(\bG,k_0) \\
\ua &  & \ua \\
\Building(\bU((h_E)|_{W^0\times W^0}), E_0)\times\Building(\bGL_{\Delta}(W^+),E) & \stackrel{\alpha}{\ra} & \Building(\bU(h|_{V^0\times V^0}),k_0)\times\\ 
 & & \Building(\bGL_D(V^+),k)\\ 
\da & & \da \\
\building(\bGL_{\Delta}(W^+),E) & \ra & \building(\bGL_D(V^+),k)\\
\end{matrix}\]
where the rows are induced by $j.$ The lower horizontal arrow fulfils the CLF-property and its image only consists of $E^{\times}$ fixed points of $\building(\bGL_D(V^+),k),$ both properties inherited from $j.$ Thus the map in the last row is $j^E,$ i.e. the inverse of $j_E,$
because otherwise we could construct a CLF-map from $\building(\bGL_D(V^+),k)^{E^{\times}}$ to $\building(\bGL_{\Delta}(W^+),E)$ different from $j_E,$ but such a CLF-map is unique by \cite[II.1.1.]{broussousLemaire:02}. Now $j^E$ maps apartments into apartments which implies that $j$ maps apartments adapted to 
$(W^+,W^-,W^0)$ into apartments adapted to $(V^+,V^-,V^0).$ 

We now prove that $j$ is toral if $E|k$ is separable. Let us assume that $E|k$ is separable. This implies that the last row $j^E$ is toral by \cite[5.1]{broussousLemaire:02} which implies the torality of $\alpha$ because the only maximal $E_0$-split torus of $\bU((h_E)|_{W^0\times W^0})$ is the trivial group. The torality of $\alpha$ implies the torality of $j$ on tori adapted to 
$(W^+,W^-,W^0).$ Hence $j$ is toral because the triple $(W^+,W^-,W^0)$ was choosen arbitrarily.
\end{proof}

\chapter{Summary of the main theorems}\label{chapterSummaryOfTheTheorems}

In this section we just want to summarise the main results of part 1
in one theorem. See definition \ref{defLocalHermkDatum} of a local hermitian
datum and see \ref{defTranslation} for the definition of a translation.

\begin{theorem}
Let
\[((A,V,D),\rho,k_0,h,\epsilon,\sigma)\]
be a hermitian datum over a local non-Archimedean field $k$ of residue characteristic
different from 2. Let $\beta$ be an element 
of $\Lie(\bU(h))(k_0)$ such that $E:=k[\beta]$ is semisimple over $k.$ We put $\bG:=\bU(h)$ and
 $\bH:=\bG_{\beta}.$ 
\be
\item There is an injective, affine and $\bH(k_0)$-equivariant CLF-map 
\[j:\Building(\bH,k_0)\ra\Building(\bG,k_0)\] such that:
\be
\item $j$ maps apartments into apartments and is toral if $\beta$ is separable,
\item in terms of lattice functions the
image of 
$j$ is the set of selfdual $o_E$-$o_D$-lattice functions,
\ee
\item If $j$ and $j'$ are two affine, $\Centr(\bH^0(k_0))$-equivariant CLF-maps
then there is a translation $\tau$ of $\Building(\bH,k_0)$ such that 
\[j=j'\circ\tau.\] Both maps are $\bH^0(k_0)$-equivariant and their image is the
set of selfdual $o_E$-$o_D$-lattice functions.
\item The $k$-algebra $E$ is a product of fields $E_i.$ Assume further that every $E_i$
is invariant under $\sigma$ (i.e. $J=J_{un}$). Let $1_i$ be the one element of $E_i,$ $V_i:=1_iV$
and let $\bH_i$ be the centraliser of $1_i\beta$ in $\bU(h|_{V_i\times V_i}).$
If no $\bH_i$ is $k_0$-isomorphic to $\bO^{is}_{2,k_0}$ then there is exactly
one CLF-map $j$ from $\Building(\bH,k_0)$ to $\Building(\bG,k_0).$ This map is
denoted by $j^{\beta}.$  
\ee
\end{theorem}

This theorem follows from the theorems \ref{thmExistence},
\ref{thmUniquenessForJGLIsEmpty}, \ref{thmCLFUnitaryCase} and
\ref{thmGenToTheNonSeparableCase} and the propositions \ref{propExistenceO2Is}
, \ref{propTransO2Is}, \ref{propToralityOfj}.

\part{Embedding types and their geometry}

\chapter{Introduction and notation}

\section{First remark}

This part answers a question that naturally arises from the papers of M.Grabitz and P. Broussous (see \cite{broussousGrabitz:00}) and P. Broussous and B. Lemaire (see \cite{broussousLemaire:02}). 
For an Azumaya-Algebra $A$ over a non-archimedean local field $F,$ M. Grabitz and P. Broussous have introduced embedding invariants for field em\-beddings, that is for pairs $(E,\textfrak{a})$, where $E$ is a field exten\-sion of $F$ in $A$ and $\textfrak{a}$ is a hereditary order which is normalised by $E^{\times}.$ On the other hand if we take such a field extension $E$ and define $B$ to be the centraliser of $E$ in $A,$ then $G:=A^{\times}$ and $G_{E}:=B^{\times}$ are sets of rational points of reductive groups $\bG$ and $\bH$ defined over $F$ and $E$ respectively. P. Broussous and B. Lemaire have defined a map $j_E:\ \building(\bG,F)^{E^\times}\rightarrow \building(\bH,E)$, i.e. between the Bruhat-Tits buildings of $\bG$ over $F$ and $\bH$ over $E,$ 
see section \ref{secBTBuildingsGLDV} and theorem \ref{thmBLII1.1}. 
The task Prof. Zink has given to me was to relate the embedding invariants to the behavior of the map $j_E$ with respect to the simplicial structures of $\building(\bG,F)$ and $\building(\bH,E).$ 

\section{Notation}
\be
\item The letter $\nu$ denotes the valuation on $F$ with $\nu(\pi_F)=1.$ 
\item We assume $D$ to be a finite dimensional central division algebra over $F$ of index $d.$ 
\item We fix an $m$ dimensional right $D$ vector space $V$, $m\in\mathbb{N},$ and put $A:=\End_D(V).$ In particular $V$ is a left $\tens{A}{F}{D^{op}}$-module.
\item The letter $L$ denotes a maximal unramified field extension of $F$ in $D$ and we assume that $\pi_D$ is a uniformizer of $D$ which normalizes $L,$ i.e. the map \[x\mapsto\sigma(x):=\pi_D x\pi_D^{-1},\ x\in D,\] generates $\Gal (L|F)$. 
\item For a positive integer $f|d$ we denote by $L_f$ the subfield of degree $f$ over $F$ in $L.$
\item In this part all $o_{F'}$-lattice functions on a vector space over a field $F'$ have period 1, i.e. we have \[\pi_{F'}\Lambda(x)=\Lambda(x+1).\] This is different to part 1. 
\ee

\chapter{Preliminaries}

\section{Vectors and Matrices up to cyclic permutation}\label{secVectorsMatrices}\idxD{vector classes}

\begin{remark}
All invariants which are considered in this part are vectors or matrices modulo cyclic permutation. 
\end{remark}

\begin{definition}
Let $s$ be a positive integer and $R$ be an arbitrary non-empty set. Two vectors $w,w'\in R^{1\times s}$ are said to be \textit{equivalent} if 
$w'$ can be obtained from $w$ by cyclic permutation of the entries of $w,$ i.e.
\[w'=(w_k,\ldots,w_s,w_1,\ldots,w_{k-1})\text{ for a }k\in\mathbb{N}_s.\]
The equivalence class is denoted by $\langle w\rangle.$
\end{definition}

\textbf{Vectors:} We denote by $\MopRow(s,t)$\idxS{$\MopRow(s,t)$} the set of all vectors $w\in\mathbb{N}_0^s$ whose sum of entries is $t,$ 
where $s$ and $t$ are natural numbers, i.e.
\[\sum_{i=1}^sw_i=t.\]
One can represent the class $\langle w\rangle$ of a vector $w\in\MopRow(s,t)$ by pairs 
\[\pairs(\langle w\rangle):=\langle(w_{i_0},i_1-i_0),(w_{i_1},i_2-i_1),\ldots,(w_{i_{k}},i_0+s-i_k)\rangle ,\]\idxS{$\pairs(\ )$}
where $(w_{i_j})_{0\leq j\leq k}$ is the subsequence of the non-zero coordinates.   
Given a vector $w$ with \[\pairs(\langle w\rangle)=\langle (a_0,b_0),\ldots,(a_k,b_k)\rangle \]
we define the \textit{complement of} $\langle w\rangle$, denoted by $\langle w\rangle^c$ to be the class $\langle w'\rangle$,
such that \[\pairs(\langle w'\rangle)=\langle (b_0,a_1),(b_1,a_2),(b_2,a_3),\ldots,(b_k,a_0)\rangle .\] 
This is a bijection \[(\ )^c:\ \MopRow(s,t)/\sim\ra \MopRow(t,s)/\sim.\]\idxS{$(\ )^c$} \vspace{3mm} 

\textbf{Matrices:} For $r,s,t\in\mathbb{N},$ $\Mint{r}{s}{t}$\idxS{$\Mint{r}{s}{t}$} denotes the set of $r\times s$-matrices with non-negative integer entries, such that
\bi
 \item in every column there is an entry greater than zero, and
 \item the sum of all entries is $t$.
 \ei 
For a matrix $M=(m_{i,j})\in\Mint{r}{s}{t},$ we define the vector $\row(M)\in\MopRow(rs,t)$ to be 
\[(m_{1,1},m_{1,2},\ldots,m_{1,s},m_{2,1},\ldots,m_{2,s},\ldots,m_{r,s}).\]
Two matrices $M,N\in\Mint{r}{s}{t}$ are said to be \textit{equivalent} if $\row(M)$ and 
$\row(N)$ are. The equivalence class is denoted by $\langle M\rangle .$

\begin{example}
\[\left(\begin{array}{cc}2&0\\1&3\\0&1\end{array}\right)\backsim\left(\begin{array}{cc}1&2\\0&1\\3&0\end{array}\right) 
\]
\end{example}

\section{Hereditary orders and lattice chains}

In the next section we need the concept of hereditary orders and lattice chains. As references we recommand \cite{reiner:03} for hereditary orders and \cite{broussousLemaire:02} for lattice chains. We use definition \ref{defODLattice} of a full $o_D$-lattice. We omit the word full. 

\begin{definition}
A unital subring $\textfrak{a}$ of $A$, is called an $o_F$-\textit{order of} $A$ if  $\textfrak{a}$ is an $o_F$-lattice of $A.$
We call an $o_F$-order $\textfrak{a}$ \textit{hereditary}\idxD{hereditary order} if  the Jacobson radical $\rad{\textfrak{a}}$ is a projective right-$\textfrak{a}$ module. The set of all hereditary orders is denoted by $\Her{A}.$\idxS{$\Her{A}$}
For $\textfrak{a}\in\Her{A}$ we denote by $\lattices{\textfrak{a}}$\idxS{$\lattices{\textfrak{a}}$} the set of all $o_D$-lattices $\Gamma$ of $V$ such that $a\Gamma\subseteq\Gamma$ for all $a\in\textfrak{a}.$
\end{definition}

\begin{definition}
\be
\item Let $R$ be a non-empty set, and take $r\in\mathbb{N}$. Given non-empty subsets $R_{i,j}$ of $R,$ $(i,j)\in{\mathbb{N}}_{r}^2,$ and natural numbers $n_1,\ldots,n_r$, we denote by $(R_{i,j})^{n_1,\ldots,n_r}$ the set of all block matrices in $\Matr_{\sum_{i=1}^rn_i}(R)$, such that for all $(i,j)$ the \mbox{$(i,j)$}-block lies in $\Matr_{n_i,n_j}(R_{i.j}).$ 
\item Given $r\in\mathbb{N},$ $\bar{n}=(n_1,\ldots,n_r)\in\mathbb{N}^r,$ we get a hereditary order \[\textfrak{a}^{\bar{n}}:=(R_{i,j})^{n_1,\ldots,n_r},\text{ where }\]  \[R_{i,j}:=\left\{\begin{array}{ll}o_D,&\text{if}\ j\leq i\\\pD,&\text{if}\ i<j\end{array}\right. .\] 
\item A hereditary order of $\Matr_{m}(D)$ of this form is called \textit{in standard form.}\idxD{standard form of a hereditary order} The class $\langle \bar{n}\rangle $ is called the \textit{invariant}\idxD{invariant of a hereditary order} and $r$ the \textit{(simplicial) rank}\idxD{rank of a hereditary order} of $\textfrak{a}^{\bar{n}}.$
\ee
\end{definition}

If we say that sets are conjugate to each other, we mean conjugate by an element of $A^\times.$ The proof of the next theorem is given in \cite{reiner:03}.

\begin{theorem}
We fix a $D$-basis of $V$ and identify $A$ with $\Matr_m(D).$
\be
\item Two hereditary orders in standard form of $A$ are conjugate to each other if and only if they have the same invariant.
\item Every $\textfrak{a}\in\Her{A}$ is conjugate to a hereditary order in standard form. 
\ee
\end{theorem}

By this theorem the notion of \textit{invariant}\idxD{invariant of a hereditary order} and \textit{rank}\idxD{rank of a hereditary order} carries over to every element of $\Her{A}$ and they do not depend on the choice of the basis. 

\begin{definition}
A sequence $(\Gamma_i)_{i\in\mathbb{Z}}$ of lattices of $V$ is called an \textit{$o_D$-lattice chain in}\idxD{lattice chain} $V$ if  
\be
\item for all integers $i,$ we have $\Gamma_{i+1}\subsetneqq\Gamma_{i},$ and
\item there exists a natural number $r$ such that for all integers $i$ we have $\Gamma_i\pi_D=\Gamma_{i+r}.$  
\ee
We call $r$ the \textit{rank of the lattice chain.} \idxD{rank of a lattice chain}
For a lattice chain $\Gamma$ we put
\[\lattices{\Gamma}:=\{\Gamma_i|\ i\in\mathbb{Z}\}.\]\idxS{$\lattices{\Gamma}$}
Two lattice chains $\Gamma,\ \Gamma'$ are called \textit{equivalent}\/ if  $\lattices{\Gamma}$ and $\lattices{\Gamma'}$ are equal. We write $[\Gamma]$ for the equivalence class. We define an order by $[\Gamma]\leq [\Gamma']$ if  $\lattices{\Gamma}$ is a subset of $\lattices{\Gamma'}.$ The set of all lattice chains in $V$ is denoted by $\LC_{o_D}(V).$\idxS{$\LC_{o_D}(V)$} 
\end{definition}

\begin{remark}
For every lattice chain $\Gamma$ in $V,$ the set 
\[\textfrak{a}_{\Gamma}:=\{a\in A|\ \forall i\in\mathbb{Z}:\ a\Gamma_i\subseteq\Gamma_i\}\]\idxS{$\textfrak{a}_{\Gamma}$} 
is a hereditary order of $A.$
\end{remark}

\begin{theorem}\cite[(1.2.8)]{bushnellFroehlich:83}
$[\Gamma]\mapsto \textfrak{a}_{\Gamma}$ defines a bijection between the set of equivalence classes of lattice chains in $V$ and the set of hereditary orders of $A.$ We have:
\[[\Gamma]\leq[\Gamma']\ \Leqa\ \textfrak{a}_{\Gamma}\supseteq\textfrak{a}_{\Gamma'}\]
\end{theorem}

In this part we need the definition of $\Latt{o_D}{V}$ given in 
\ref{defLattoDV} of part 1. We put $\tf{a}_{\Lambda}:=\End(\Lambda)$\idxS{$\tf{a}_{\Lambda}$} to emphasize that we mainly are interested in a filtration "around" a hereditary order than a lattice function of $A,$ see \ref{subsecSquarelatticeFct}.
We also put
\[\rank([\Lambda]):=\rank(\Lambda),\]\idxS{$\rank([\Lambda])$}
see \ref{defLatticeFunction}, and
\[\lattices{\Lambda}:=\image(\Lambda)\]\idxS{$\lattices{\Lambda}$}
for \[\Lambda\in\Lattone{o_D}{V}.\]

\section{Embedding types}

For a field extension $E|F$ we denote by $E_D|F$ the maximal field extension in $E|F,$ which is $F$-algebra isomorphic to a subfield of $L.$ Its degree is the greatest common divsor of $d$ and the residue degree of $E|F.$  

\begin{definition}
An \textit{embedding}\idxD{embedding} is a pair $(E,\textfrak{a})$\idxS{$(E,\textfrak{a}),$ an embedding} satisfying 
\be
\item $E$ is a field extension of $F$ in $A$,
\item $\textfrak{a}$ is a hereditary order of $A$, normalised by $E^{\times}.$ 
\ee
Two embeddings $(E,\textfrak{a})$ and $(E',\textfrak{a}')$ are said to be \textit{equivalent}\idxD{equivalent embeddings} if 
there is an element $g\in A^{\times}$, such that $gE_Dg^{-1}=E'_D$ and $g\textfrak{a}g^{-1}=\textfrak{a}'$.
\end{definition}

\begin{remark}
In each equivalence class of embeddings there is a pair such that the field can
be embedded in $L$. 
\end{remark}

The definition of $\Mint{r}{s}{t}$ is in the previous section.
Until the end of this section we fix a $D$-basis of $V$ and identify $A$ with $\Matr_m(D).$
\begin{definition}
Let $f|d$ and $r\leq m$.
A matrix with $f$ rows and $r$ columns is called an \textit{embedding datum}\idxD{embedding datum} if it belongs to $\Mint{f}{r}{m}$. Given an embedding datum $\lambda$, we define the pearl embedding as follows. 
The \textit{pearl embedding}\idxD{pearl embedding} of $\lambda$ (with respect to the fixed $D$-basis of $V$) is the embedding $(E,\textfrak{a})$, with the
following conditions:
\be
\item $[E:F]=f$,
\item $E$ is the image of the monomorphism \[x\in L_f\mapsto \diag(M_1(x),M_2(x),\ldots,M_r(x))
\in\Matr_m(D)\]
where
\[M_j(x)=\diag(\sigma^0(x)\EMatr{\lambda_{1,j}},\sigma^1(x)\EMatr{\lambda_{2,j}},\ldots,\sigma^{f-1}(x)\EMatr{\lambda_{f,j}})\]
\item $\textfrak{a}$ is a hereditary order in standard form according to the partition $m=n_1+\ldots +n_r$ where $n_j:=\sum_{i=1}^f\lambda_{i,j}.$
\ee 
\end{definition}

\begin{theorem}\cite[2.3.3 and 2.3.10]{broussousGrabitz:00}\label{thmPearlEmb}
\be
\item Two pearl embeddings are equivalent if and only if the embedding
  data are equivalent.
\item In any class of embeddings lies a pearl embedding.
\ee 
\end{theorem}

\begin{definition}
Let $(E,\textfrak{a})$ be an embedding. By the theorem it is equivalent to a
pearl-embedding. The class of the corresponding matrix $(\lambda_{i,j})_{i,j}$
is called the \textit{embedding type}\idxD{embedding type} of $(E,\textfrak{a}).$ This definition does not depend on the choice of the basis by the theorem of Skolem-Noether.  
\end{definition}


\chapter{\texorpdfstring{The simplicial structure of  $\building(\bGL_D(V),F)$}{The simplicial structure of the building of  GLD(V)}}\label{chapterEuclideanBuildingOfGLDV}

In section \ref{secBTBuildingsGLDV} we gave the definition of $\Building(\bGL_D(V),F),$ i.e. of the Bruhat-Tits building of $\bGL_D(V)$ over $F.$ In this part of the thesis we are interested in its simplicial structure. 

\section{Definitions}

Here we give the basic definitions in order to be able to state precisely the description of the Euclidean building of $\bGL_D(V)$ over $F$ with lattice chains.
Basic definitions of the notions of simplicial complex and chamber complex are given in \cite[Ch. I App.]{brown:89}.
For the definition of a Coxeter complex see \cite[Ch. III]{brown:89}.

\begin{definition}
A \textit{building} \idxD{building} is a triple $(\Omega,\mathcal{A},\leq)$, such that
$(\Omega,\leq)$ is a simplicial complex and $\mathcal{A}$ is a set of subcomplexes of $(\Omega,\leq)$ which cover $\Omega,$ i.e.
\[\bigcup \mathcal{A} =\Omega,\]
(The elements of $\mathcal{A}$ are called \textit{apartments.}) statisfying the following \textit{building axioms}\/:
\bi
\item \textbf{B0} Every element of $\mathcal{A}$ is a Coxeter complex.
\item \textbf{B1} For \textit{faces} (also called simplicies), i.e. elements, $S_1$ and $S_2$ of $\Omega$ there is an apartment $\Sigma$ containing them. 
\item \textbf{B2} If $\Sigma$ and $\Sigma'$ are two arpartments containing $S_1$ and $S_2$ then there is a poset isomorphism from $\Sigma$ to $\Sigma'$ which fixes $\bar{S}_1$ and $\bar{S}_2$ where $\bar{S}$ for a face $S$ is defined to be the set of all faces $T\leq S.$ 
\ei
The minimal faces are the \textit{vertices}\idxD{vertex} and the rank of a face $S$ is the number of vertices in $\bar{S}.$ The maximal faces are the \textit{chambers}.\idxD{chamber} Faces of rank two are \textit{edges}.\idxD{edge} A building is said to be \textit{thick} \idxD{thick} if every codimension 1 face is attached to at least three chambers. 
\end{definition}

\begin{remark}
A building in this part of the thesis consist either only of one element or is thick. 
\end{remark}

A \textit{Euclidean Coxeter complex} \idxD{Euclidean Coxeter complex} is a Coxeter complex $(\Sigma,\leq)$ 
which is poset-isomorphic to a simplicial complex $\Sigma(W,V)$ defined by an essential irreducible infinite affine reflection group $(W,V).$
For a face $S$ of a simplicial complex $(\Omega,\leq)$ the set of all formal sums
\[\Sigma_{v\leq S,rk(v)=1}\lambda_vv\] with positive real coefficients such that $\Sigma_{v\leq S,rk(v)=1}\lambda_v=1$ is denoted by $|S|.$ The set 
\[|\Omega|:=\bigcup_{S\in\Omega}|S|\] 
is called the geometric realisation of $\Omega.$
A \textit{morphism} of simplicial complexes from $(\Omega,\leq)$ to $(\Omega',\leq')$ is a
map $f:(\Omega,\leq)\ra (\Omega',\leq'),$ such that for every face $S\in\Omega$ the restriction 
$f:\ \bar{S}\ra \bar{f(S)}$ is a poset isomorphism. In \cite{brown:89} the notion of non-degenerate simplicial map is used instead of morphism. A morphism $f$ induces a map $|f|$ between the geometric realisations, by 
\[|f|(\sum_v\lambda_vv):=\sum_v\lambda_vf(v).\]
\begin{definition}
Given two buildings $(\Omega,\mathcal{A},\leq)$ and $(\Omega',\mathcal{A}',\leq')$ a \textit{morphism}
from the first to the latter is a morphism of simplicial complexes such that the image of an apartment of $\mathcal{A}$ is contained in an apartment of $\mathcal{A}'.$
\end{definition}

As described in \cite{brown:89} VI.3 there is a canonical way to define a metric, up to a scalar, on the geometric realisation of a Euclidean building by pulling back the metric from an affine reflection group to the apartment and this defines a canonical affine structure on the geometric realisation of the building.
The map $|\phi|$ between the geometric realisations of two Euclidean buildings induced by an isomorphism $\phi$ is affine. 

\section{The description with lattice chains}\label{secDiscrWithLatticeChains}

Let $\Omega$\idxS{$\Omega$} be the simplicial structure of $\building(\bGL_D(V),F).$ We denote 
\[\calli{I}:=|\Omega|=\building(\bGL_D(V),F).\]\idxS{$\calli{I}$}
By theorem \ref{thmBLI1.4} there is a unique affine and $A^{\times}$-equivariant bijection from 
$\calli{I}$ to $\Latt{o_D}{V}.$ 
We describe the Euclidean building $\Omega$ of $A^{\times}$ in terms of lattice chains and hereditary orders as it is done in \cite[I.3]{broussousLemaire:02}.

\begin{proposition}
\be
\item The posets $(\LC_{o_D}(V),\leq)$ and $(\Her{A},\supseteq)$ are simplicial complexes  of rank $m.$ They are isomorphic via $\Psi([\Gamma]):=\textfrak{a}_{\Gamma}$ as simplicial complexes. 
\item A hereditary order is a vertex (resp. a chamber) if and only if its rank is 1 (resp. m).
\ee
\end{proposition}

\begin{definition}
A \textit{frame of} $V$ is a set of lines $v_1D,\ldots,v_mD$, where $v_i, i\in\mathbb{N}_m,$ is a $D$-basis of $V.$ If $\textfrak{R}$ is a frame we say that a lattice $\Gamma$ \textit{is split by}\/ $\textfrak{R}$ if  \[\Gamma=\bigoplus_{W\in\textfrak{R}}(\Gamma\cap W).\] 
A lattice chain $\Gamma$, lattice function $\Lambda,$ hereditary order $\textfrak{a}$ is \textit{split by}\/ $\textfrak{R}$ if  every element of $\lattices{\Gamma}$, $\lattices{\Lambda},$ $\lattices{\textfrak{a}}$ resp. is split by $\textfrak{R}$. An equivalence class is \textit{split by}\/ $\textfrak{R}$ if  every element of the equivalence class is split by $\textfrak{R}.$ The set of these classes split by $\textfrak{R}$ is called the \textit{apartment corresponding to \textfrak{R}} and is denoted by ${\LC}_{o_D}(V)_{\textfrak{R}},$  $\Her{A}_{\textfrak{R}},$ $\Latt{o_D}{V}_{\textfrak{R}}$ resp.. For the set of these apartments we write
\[\textfrak{A}({\LC}_{o_D}(V)),\ \textfrak{A}(\Her{A})\ \&\  \textfrak{A}(\Latt{o_D}{V}).\]
\end{definition}

\begin{definition}
The left action of $A^{\times}$ on the set of $o_D$-lattices of $V,$ i.e.
\[g.\Gamma:=\{g\gamma|\ \gamma\in\Gamma\},\]
defines an $A^{\times}$-action on ${\LC}_{o_D}(V),\ \Latt{o_D}{V}$ and $\Her{A}.$ 
\end{definition}

\begin{proposition}
\be
\item The two triples \[({\LC}_{o_D}(V),\textfrak{A}({\LC}_{o_D}(V)),\leq)\ \&\  (\Her{A},\textfrak{A}(\Her{A}),\supseteq)\] are isomorphic Euclidean buildings via $\Psi.$
\item $\Psi$ is $A^{\times}$-equivariant.
\item For every frame $\textfrak{R}$ the image of  ${\LC}_{o_D}(V)_{\textfrak{R}}$ under $\Psi$ is $\Her{A}_{\textfrak{R}}.$
\ee
\end{proposition}

For steps and calculations for the proof see for example \cite{reiner:03}.

\begin{remark}
Every $\textfrak{a}\in\Her{A}$ has a rank as a face in the chamber complex $(\Her{A},\supseteq),$ and this coinsides with the simplicial rank, but we never mean the $o_F$-rank of $\textfrak{a}$.
\end{remark}


From now on we need the affine structure on $\Latt{o_D}{V},$
see definition \ref{defLattoDV}. Now the next proposition explains why one can replace $\Omega$ by the building of classes of lattice functions. The geometric realisation of ${\LC}_{o_D}(V)$ can be identified with $\Latt{o_D}{V}$ in the following way. We put \[ [x]+:=\inf\{z\in\mathbb{Z}|\ x\leq z\},\ x\in\mathbb{R},\]\idxS{$[x]+$}
and we define a bijective map \[\tau:\ |{\LC}_{o_D}(V)|\ra \Latt{o_D}{V}\] as follows.
A convex barycenter 
\[\sum \beta_i[\Gamma^i],\ \beta_i\geq 0\text{ and }\sum_i\beta_i=1,\] with vertices $[\Gamma^i]$ of a chamber of ${\LC}_{o_D}(V)$ is mapped to $\sum_i\beta_i[\Lambda^i]$, where $\Lambda^i(t):=\Gamma^i_0\textfrak{p}_D^{[td]+}.$  

\begin{remark}
The definition of $\tau$ and proposition \ref{propBTII2.16} im\-ply that $\Latt{o_D}{V}$ in\-herits the same sim\-pli\-cial struc\-ture from $\Omega$ and from ${\LC}_{o_D}(V).$ 
\end{remark}

\begin{proposition}[\cite{broussousLemaire:02} sec. I.3]
The composition of the bijection from\\ $\Latt{o_D}{V}$ to $\calli{I}$ with $\tau$ in\-du\-ces an $A^\times$-equi\-va\-riant iso\-mor\-phism from \[({\LC}_{o_D}(V),\textfrak{A}({\LC}_{o_D}(V)),\leq)\]
to the buil\-ding $\Omega.$
\end{proposition}

\begin{notation}
By the propositions above we can identify $\Omega$ with \[(\Her{A},\textfrak{A}(\Her{A}),\supseteq).\]
\end{notation}

\bdfn
We call $\Omega$ the \textit{Euclidean building of $A^{\times}.$}\idxD{Euclidean building of $A^{\times}$}
\edfn

\chapter{\texorpdfstring{The map $j_E$}{The map jE}}\label{chThemapjE}

\begin{notation}
For this section let $E|F$ be a field extension in $A$ and we set $B$ to be the centraliser of $E$ in $A,$ i.e.
\[B:=\Centr_A(E):=\{a\in A|\ ab=ba\ \forall b\in B\}.\]
We denote the Euclidean building of $B^{\times}$ by $\Omega_E$ and its geometric realisation by $\mathcal{I}_E.$
\end{notation}

The next results are taken from \cite{broussousLemaire:02}.
We restate the following theorem in the notation of this part.

\begin{theorem}\cite[Thm. II.1.1.]{broussousLemaire:02}\label{thmBLII1.1}
There exists a unique application $j_E:\mathcal{I}^{E^{\times}}\ra \mathcal{I}_E$ such that for any $x\in \mathcal{I}^{E^{\times}}$ and $t\in\bbR$ we have $\tf{a}_{j_E(x)}(t)=B\cap\tf{a}_x(e(E|F)t).$
The map $j_E$ satisfies the following properties:
\be
\item it is bijecive,
\item it is a $B^\times$-equivariant and
\item it is affine.
\ee
Moreover its inverse $j_E^{-1}$ is the only map $\mathcal{I}_E\ra \mathcal{I}$ such that 2. and 3. hold.
\end{theorem}

We briefly give Broussous and Lemaire's description of $j_E$ in terms of lattice functions but only in the case, where $E|F$ is isomorphic to a subextension $L_f|F$ of $L|F.$ Then $\tens{E}{F}{L}\cong\bigoplus_{k=0}^{f-1}L$ coming from the decomposition $1=\sum_{k=0}^{f-1}1_k$ labeled such that the
$\Gal(L|F)$-action on the second factor gives $\sigma(1_k)=1_{k-1}$ for $k\geq
1$ and $\sigma(1_0)=1_{f-1}$.
Applying it on the $\tens{E}{F}{L}$-module $V$, we get $V=\bigoplus_kV_k$, $V_k:=1_kV.$

\begin{remark}
\be
\item $B\cong\End_{\Delta_E}(V_0)$ and
\item $B\cong \Matr_{m}(\Delta_E)$ 
\ee
where $\Delta_E:=\Centr_{D}(L_f).$
\end{remark}

\begin{theorem}\label{thmDescrjE}\cite[II 3.1.]{broussousLemaire:02}
In terms of lattice functions $j_E$ has the form
$j_E^{-1}([\Theta])=[\Lambda]$, with 
\[\Lambda(s):=\bigoplus_{k=0}^{f-1}\Theta(s-\frac{k}{d})\pi_D^k,\ s\in \mathbb{R}\] 
where $\Theta$ is an $o_{\Delta}$-lattice function on $V_0.$
\end{theorem}

\chapter{Embedding types through barycentric coordinates}\label{chapterIdea}

In this chapter we keep the notation from chapter \ref{chThemapjE}.
We repeat that $E_D|F$ denotes the big\-gest field extension of $E|F$ which can be embedded in $L|F.$ 
The centraliser of $E_D$ in $A$ is denoted by $B_D.$ 
We need a notion of orientation on $\Omega_{E_D}$ to order the barycentric co\-or\-di\-nates of a point in $\mathcal{I}_{E_D}.$ 

\begin{definition}
An edge of $\Omega$ with vertices $e$ and $e'$ is \textit{oriented towards} $e',$
if there are lattices $\Gamma\in\lattices{e}$ and $\Gamma'\in\lattices{e'},$ such that 
$\Gamma\supseteq\Gamma'$ with the quotient having $\kappa_D$-dimension 1, i.e. 
$\kappa_F$-dimension $d.$ We write $e\ra e'.$ If $x$ is a point in $\mathcal{I}$ then there is a chamber $C\in\Omega$ such that $x$ lies in the closure of $|C|,$ i.e. in
\[\bigcup_{S\leq C}|S|.\] The vertices of $C$ can be given in the way
\[e_1\ra e_2\ra\ldots\ra e_{m}\ra e_1.\] 
If $(\mu_i)$ are the barycentric coordinates of $x$ with respect to $(e_i),$ i.e.  
\[x=\sum_i\mu_ie_i,\]
then the class $\langle \mu\rangle $ is called the \textit{local type of} $x.$
\end{definition}

This definition applies for $\mathcal{I}_{E_D}$ as well. The skewfield is then $\Centr_D(E_D)$ instead of $D$ and one has to substitute $d$ by $\frac{d}{[E_D:F]}.$ 

\begin{proposition}
The notion of local type does not depend on the choice of the chamber $C$ and the starting vertex $e_1.$
\end{proposition}

For the definition of $\langle  \rangle ^c$ see section \ref{secVectorsMatrices}.

\begin{theorem}\label{thmconnection}
Let $(E,\textfrak{a})$ be an embedding of $A$ with embedding type $\langle \lambda\rangle $ and suppose $\textfrak{a}$ to have rank $r.$ If $M_{\textfrak{a}}$ denotes the barycenter of $\textfrak{a}$ in $\mathcal{I}$ and $\langle \mu\rangle $ the local type of $j_{E_D}(M_{\textfrak{a}}),$ then the following holds.
\be
\item $r[E_D:F]\mu\in \mathbb{N}_0^m,$ and
\item $\langle \row(\lambda)\rangle =\langle [E_D:F]r\mu\rangle ^c.$ 
\ee
\end{theorem}

\begin{remark}
With theorem \ref{thmconnection} we can calculate the embedding type from the local type. For example take $r=2,\ f=6,\ m=7$ and  assume that $j_{E_D}(M_{\textfrak{a}})$ is
\[\frac{3}{12}b_0+\frac{2}{12}b_1+\frac{1}{12}b_2+\frac{0}{12}b_3+\frac{0}{12}b_4+\frac{4}{12}b_5+\frac{2}{12}b_6.\]
and thus \[\langle 12\mu\rangle =\langle 3,2,1,0,0,4,2\rangle \equiv\langle (3,1),(2,1),(1,3),(4,1),(2,1)\rangle .\]
From the complement 
\[\langle 12\mu\rangle ^c\equiv \langle (1,2),(1,1),(3,4),(1,2),(1,3)\rangle \equiv \langle 1,0,1,3,0,0,0,1,0,1,0,0\rangle \] 
applying theorem \ref{thmconnection} we can deduce  the embedding type of $(E,\textfrak{a}):$ 
\begin{equation*}
\left(\begin{array}{cc}1&0\\1&3\\0&0\\0&1\\0&1\\0&0\end{array}\right).
\end{equation*}
\end{remark}

For the proof we can restrict to the case where $E=E_D$ and thus $B=B_D.$ We put $f:=[E:F],$ i.e.
\[E\cong L_f\subseteq L\]
and 
\[F\subseteq E\subseteq B\subseteq A.\]
 Firstly we need some lemmas. The actions of $G$ on square lattice functions by conjugation induces maps  
\[m_g:\ \Omega \ra \Omega,\ x\mapsto g.x \]
and 
\[c_g:\ \mathcal{I}_E\ra \mathcal{I}_{gEg^{-1}},\ y\in \Latttwo{o_E}{B}\mapsto gyg^{-1}\in\Latttwo{o_{gEg^{-1}}}{gBg^{-1}} \]
for $g\in G.$ 

\begin{lemma}\label{lemmaCommutativeDiagram}
$|m_g|$ and $c_g$ induce isomorphisms on the simplicial structures of the Euclidean buildings, which preserve the orientation, i.e. an oriented  edge is mapped to an oriented edge such that the direction is preserved. In particular $|m_g|$ and $c_g$ are affine bijections, $m_g$ preserves the embedding type, $c_g$ the local type, and the following diagram is commutatve:
\begin{equation*}
\xymatrix{
\mathcal{I}^{E^\times} \ar[r]^{ |m_g| } \ar[d]^{j_E} & \mathcal{I}^{(gEg^{-1})^\times} \ar[d]^{j_{gEg^{-1}}}\\
\mathcal{I}_{E} \ar[r]^{c_g} & \mathcal{I}_{gEg^{-1}}
}
\end{equation*}
\end{lemma}  

The following lemma gives a geometric interpretation of the map
\[\{\text{embedding types}\}\ra\{\text{embedding types of vertices}\} \] 
\[ \langle \lambda\rangle\mapsto \langle \row(\lambda)^T \rangle. \]

\begin{lemma}[rank reduction lemma]\label{LemRankReduction}
Assume there is a field extension $K|F$ of degree $s$ in $E|F,$ where $2\leq s\leq m.$ Let $\textfrak{a}$ be a vertex in $\Omega^{E^\times}$ such that $\textfrak{a}\cap \Centr_A(K)$ is a face of rank $s$ in $\Omega^{E^\times}_{K}$ and assume $(E,\textfrak{a})$ has embedding type $\langle \lambda\rangle $ and \mbox{$(E,\textfrak{a}\cap \Centr_{K}(A))$} has embedding type $\langle \lambda'\rangle .$ Then we get 
\[\row(\lambda)\sim\row(\lambda'),\text{ i.e. }\lambda\sim\trans{\row(\lambda')}.\]
\end{lemma}

\begin{proof}
By lemma \ref{lemmaCommutativeDiagram} it is enough to show the result only for one embedding equivalent to $(E,\textfrak{a}).$ For simplicity we can restrict ourself to the case of $s=2.$ The argument for $s>2$ is similar. We fix a $D$-basis of $V.$ It is $(E,\textfrak{a})$ equivalent to the pearl embedding $(E_\lambda,\textfrak{a}_\lambda)$ of  $\lambda,$ moreover $\textfrak{a}_\lambda$ is $\Matr_{m}(o_D).$  Now we apply a permutation $p$ on $(E_\lambda,\textfrak{a}_\lambda)$ such that the odd exponents of $\sigma$ in $pE_\lambda p^{-1}$ are behind all even exponents, i.e. $pE_\lambda p^{-1}$ is the image of 
\[x\in L_{f}\mapsto \diag(M_{n_1}(x),M_{n_2}(x)),\ n_1:=\sum_{i\text{ odd}}\lambda_i,
\ n_2:=\sum_{i\text{ even}}\lambda_i\]
where
\[M_{n_1}(x)=\diag(\sigma^0(x)\EMatr{\lambda_{1}},\sigma^2(x)\EMatr{\lambda_{3}},\ldots,\sigma^{f-2}(x)\EMatr{\lambda_{f-1}})\]
and
\[M_{n_2}(x)=\diag(\sigma^1(x)\EMatr{\lambda_{2}},\sigma^3(x)\EMatr{\lambda_{4}},\ldots,\sigma^{f-1}(x)\EMatr{\lambda_{f}}).\]
For the embedding $(E',\textfrak{a}')$ obtained by conjugating $p(E_\lambda,\textfrak{a}_\lambda) p^{-1}$ with the matrix 
\[\diag(\EMatr{n_1},\pi_D^{-1}\EMatr{n_2})\] we have the following properties. Let $K'|F$ be the field extension of degree two in $E'|F.$
\bi
\item $K'$ is the image of the diagonal embedding of $L_2$ in $\Matr_m(D)$ and its centraliser is $\Matr_m(\Delta_{K'}),$ where $\Delta_{K'}:=\Centr_{D}(L_2).$ This follows because even powers of $\pi_D$ commute with $L_2.$
\item The intersection of $\textfrak{a}'$ with $\Matr_m(\Delta_{K'})$ is a herditary order in standard form with invariant $\langle n_1,n_2\rangle .$ The positivity of the integers $n_i$ follows from the assumption that this intersection is a face of rank $2.$
\ei
Since $\pi_{\Delta_{K'}}:=\pi_D^2$ is a prime element of $\Delta_{K'}$ which normalises $L$ and since the powers of $\sigma$ occuring in the description of $E'$ are even we can read the embedding type of $(E',\textfrak{a}'\cap \Matr_m(\Delta_{K'}))$ directly. It is  
the class of
\begin{equation*}
\left(\begin{array}{cc}
\lambda_1 & \lambda_2\\
\lambda_3 & \lambda_4\\
\vdots &\vdots \\
\lambda_{f-1} & \lambda_{f}
\end{array}\right).
\end{equation*}
Thus the result follows.
\end{proof}
\vspace{1em}

The next lemma shows that changing the skewfield does not change the embedding type.

\begin{lemma}[changing skewfield lemma]\label{lemmaChangingSkewfield}
Let $D'$ be a central skewfield over a local field $F'$ of index $d$ with a maximal unramified extension field $L'$ normalized by a prime element $\pi_{D'}$ and assume that $V'$ is an $m$ dimensional right vector space over $D'.$ Denote the Euclidean building of $GL_m(D')$ by $\mathcal{I}'$ and let $\Sigma,\ \Sigma'$ be an apartment of $\mathcal{I},\ \mathcal{I}'$ corresponding to a basis $(v_i),\ (v'_i)$ respectively. Then $\Sigma'$ is fixed by the image $E'$ of the diagonal embedding of $L'_f$ in $\Matr_{m}(D').$ Assume further that $E$ is the image of the diagonal embedding of $L_f$ in $\Matr_m(D).$ Under these assumptions the map $\equiv$ from $|\Sigma|$ to $|\Sigma'|$ defined by
\[ [x\mapsto\bigoplus_iv_i\textfrak{p}_D^{[d(x+\alpha_i)]+}]\mapsto [x\mapsto\bigoplus_iv'_i\textfrak{p}_{D'}^{[d(x+\alpha_i)]+}] \]
is the geometric realisation from an isomorphism $\phi$ of simplicial complexes which preserves the orientation and the embedding type. The latter means that if $\textfrak{a}'$ is the image of a hereditary order $\textfrak{a}$ under $\phi$ then the embedding types of 
$(E,\textfrak{a})$ and $(E',\textfrak{a}')$ equal.
\end{lemma}

\begin{proof}
We define $\phi$ to map the class of a lattice chain $\textfrak{L}$ with 
\[\textfrak{L}_j=\bigoplus_iv_i\textfrak{p}_D^{\nu_{i,j}}\]
to the class of $\textfrak{L}'$ with
\[\textfrak{L}'_j=\bigoplus_iv'_i\textfrak{p}_{D'}^{\nu_{i,j}}.\]
We only show the preserving of the embedding type. The other properties are verified easely. We take the two lattice chains $\textfrak{L}$ and $\textfrak{L}'$ with corresponding hereditary orders $\textfrak{a}$ and  $\textfrak{a}'.$
Applying from the left an appropriative permutation matrix $P$ and an apropritative diagonal matrix $T$ (resp. $T'$), whose entries are powers of the corresponding prime element, we obtain simultanously lattice chains corresponding to hereditary orders $\textfrak{b},\ \textfrak{b}'$ in the same standard form.  More precisely $T'$ is obtained from $T$ if $\pi_D$ is substituted by $\pi_{D'}.$ 
Thus $(TPEP^{-1}T^{-1},\textfrak{b})$ and $(T'PE'P^{-1}T'^{-1},\textfrak{b}')$ have the same embedding type and thus by conjugating back $(E,\textfrak{a})$ and $(E',\textfrak{a}')$ have the same embedding type.
\end{proof}
\vspace{1em}

We now fix a $D$-basis $v_1,\ldots,v_m$ of $V$ and therefore a frame 
\[\textfrak{R}:=\{R_i:=v_iD |\ 1\leq i\leq m\}\]
and an apartment $\Sigma=\Her{A}_{\textfrak{R}}$ of $\Omega.$
The algebra $A$ can be identified with $\Matr_m(D).$
By the affine bijection $|\Sigma| \cong\mathbb{R}^{m-1} $
which maps 
\[ [\Lambda] \text{ with }\Lambda(x)=\bigoplus_i\textfrak{p}_D^{[d(x+\alpha_i)]+}\]
to
\[ d(\alpha_1-\alpha_2,\ldots,\alpha_{m-1}-\alpha_m), \]
we can introduce affine coordinates on $|\Sigma|$ where the points of $|\Sigma|$ corresponding to the vectors $0,\ (f,0,\ldots,0),\ (0,f,0,\ldots,0),\ldots,\ (0,\ldots,0,f)$ are denoted by $Q_1,Q_2,\ldots,Q_m.$ 

\begin{remark}
The vertices of $\Sigma$ are exactly the points of
\[ Q_1+\sum_{i=2}^m\frac{1}{f}\mathbb{Z}(Q_i-Q_1).\]
\end{remark}

\begin{remark}\label{remarkDiagonalEmbedding}
For an element $g\in \cap_{i=1}^m(\End_D(R_i))^\times,$ i.e. a diagonal matrix,  $|m_g|$ induces an affine bijection of $|\Sigma |.$
If $g$ is $\diag(1,\ldots,1,\pi_D^k,1,\ldots,1),$ with $\pi_D^k$ in the i-th row, the map $|m_g|$ is of the form 
\[Q\mapsto Q+\frac{k}{f}(Q_{i+1}-Q_i),\]
where we set $Q_{m+1}:=Q_1.$
\end{remark}

\begin{example}\label{exampleSimplification}
Let us assume $E$ is the image of the diagonal embedding of $L_f$ in $\Matr_m(D),$ i.e.
\[E=\{(x,\ldots,x)|\ x\in L_f\}.\]
Then $B$ and $j_E$ simplify, i.e.
\be
\item $B=\End_{\Delta}(W)$ with $\Delta:=\Centr_{D}(L_f)$ and $W:=\bigoplus_iv_i\Delta$
\item The geometric realisation of $\Sigma$ is a subset of $\calli{I}^{E^\times}.$
\item For $[\Lambda]\in\calli{I}$ we have
\[ j_E([\Lambda])=[\Lambda\cap W]\]
where $\Lambda\cap W$ denotes the lattice function 
\[ x\mapsto \Lambda(x)\cap W.\]
\item The image of $j_E|_{|\Sigma|}$ is the geometric realisation of the apartment $\Sigma_E$ which belongs to the frame $\{v_i\Delta|\ 1\leq i\leq m\}$ and in affine coordinates the map has the form
\[x\in\mathbb{R}^{m-1}\mapsto \frac{1}{f}x\in\mathbb{R}^{m-1}.\]
\item The vertices of $\Sigma_E$ are the points of $|\Sigma_E|$ with affine coordinate vectors in $\mathbb{Z}^{m-1}.$ Specifically the points $P_i:=j_E(Q_i)$ are vertices of a chamber of $\Sigma_E.$
\item  The edge from $P_i$ to $P_{i+1}$ is oriented to $P_{i+1}.$ 
\ee
\end{example}

\begin{proof}[example]
To prove the statements of the example it is enough to calculate $j_E$ in terms of lattice functions, i.e to show 3. The statements then follow by similar and standard calculations.\\
For 2: We have $|\Sigma|\subseteq \calli{I}^{E^{\times}}$ because for an $o_D$-lattice function $\Lambda$ split by $\textfrak{R}$ the action of an element of $E^{\times}$ on $\Lambda$ is the multiplication of every lattice $\Lambda(t)$ by a fixed element $x\in D^{\times}.$ \\
For 3: We use the decomposition 
\[ V=\tens{W}{\Delta}{D}=W\oplus W\pi_D\oplus W\pi_D^2\oplus\ldots\oplus W\pi_D^{f-1}, \]
the function 
\[[\Gamma]\in\calli{I}^{E^{\times}}\mapsto [\Lambda]\in\calli{I}\]
with 
\[\Lambda(x):=\bigoplus_{i=0}^{f-1}\Gamma(x-\frac{i}{d})\pi_D^{i}\]
is affine and $B^{\times}$-equivariant. By \ref{thmBLII1.1} it has to be $j^{-1}_E$ and thus 
\[j_E([\Lambda])=[\Lambda\cap W].\]  
The appearence of $j_E$ in terms of coordinates follows now from
\[\textfrak{p}_D^{[x]+}\cap \Delta = \textfrak{p}_\Delta^{[\frac{[x]+}{f}]+}=\textfrak{p}_\Delta^{[\frac{x}{f}]+}.\]  
\end{proof}

\begin{proof}(of theorem \ref{thmconnection})
By lemma \ref{lemmaCommutativeDiagram} and by theorem \ref{thmPearlEmb}
we can assume that we are in the situation of the example \ref{exampleSimplification} above and that there is a diagonal matrix $h$ consisting of powers of $\pi_D$  with exponents in $\mathbb{N}_{f-1}\cup\{0\}$ such that 
\[(hEh^{-1},h\textfrak{a}h^{-1})\]
is the pearl embedding of $\lambda.$
We consider two cases for the proof.\\
\textbf{Case 1:} $\textfrak{a}$ has rank 1, i.e. 
\[h\textfrak{a}h^{-1}=\Matr_m(o_D)=Q_1\]
and $\lambda$ is only one column.
We get $\textfrak{a}$ from $Q_1$ by applying $m_{h^{-1}}$ which is a composition of 
maps $m_g$ where $g$ differs from the identity matrix by only one diagonal entry $\pi_D^k.$ Now remark \ref{remarkDiagonalEmbedding} gives
\[\textfrak{a}=Q_1-\sum_{j=1}^{m}\frac{a_j}{f}(Q_{j+1}-Q_j)\]
where $a_j:=k-1$ if 
\[\sum_{i=1}^{k-1}\lambda_{i}<j\leq\sum_{i=1}^{k}\lambda_{i}.\]
Thus in barycentric coordinates $j_E(M_\textfrak{a})$ has the form
\[\frac{f-a_m+a_1}{f}P_1+\frac{a_2-a_1}{f}P_2+\ldots +\frac{a_{m}-a_{m-1}}{f}P_m.\]
and therefore the vector
\[\mu:=(\frac{f-a_m+a_1}{f},\frac{a_2-a_1}{f},\ldots,\frac{a_{m}-a_{m-1}}{f})\]
fullfils part one of the theorem.
If $(\lambda_{i_l})_{1\leq l\leq s}$ is the subsequence of non-zero entries we define the indexes 
\[j_l:=\lambda_1+\ldots+\lambda_{i_{l-1}}+1\]
and $j_1:=1.$
This are the indexes where the $\mu_j$ are non-zero, more precisely from 
\[j_l=\sum_{i=1}^{i_l-1}\lambda_i+1\leq \sum_{i=1}^{i_l}\lambda_i\]
we obtain for $a_j$ the following values:
\[a_j=a_{j_l}=i_l-1,\ j_l\leq j<j_{l+1}\]
and 
\[a_j=a_{j_s}=i_s-1,\ j_s\leq j\leq m,\]
and thus the subsequence of non-zero entries of $f\mu$ is 
\[(f\mu_{j_l})=(f-i_{s}+i_1,i_2-i_1,i_3-i_2,\ldots,i_{s}-i_{s-1}).\]
Therefore we get for $\pairs(\langle f\mu\rangle )$ the expression
\[\langle (f-i_{s}+i_1,\lambda_{i_1}),(i_2-i_1,\lambda_{i_2}),(i_3-i_2,\lambda_{i_3}),\ldots,(i_{s}-i_{s-1},\lambda_{i_s})\rangle \]
and this is precisely $\langle \row(\lambda)\rangle ^c.$\\
\textbf{Case 2:} Assume the rank $r$ of $\textfrak{a}$ is not 1. Here we want to use rank reduction. We fix an unramified field extension $L'|F$ of degree $rd$ in an algebraic closure of $F.$
Denote by $D'$ a skewfield which is a central cyclic algebra over $F$ with maximal field $L'$ and an $L'$-normalising prime element $\pi_{D'},$ i.e.
\[D'=\bigoplus_{i=0}^{dr-1}L'\pi_{D'}^i,\] \[\pi_{D'}L'\pi_{D'}^{-1}=L',\ \text{and } \pi_{D'}^{d+r}=\pi_F.\]
The images of $L'_r,$ $L'_{rf}$ under the diagonal embedding of $L'$ in $\Matr_{m}(o_{D'})$ are denoted by $F',$ $E'$ respectively and the apartment of the Euclidean building $\Omega'$ of $GL_m(D')$ corresponding to the standard basis is denoted by $\Sigma',$ i.e.
we have a field tower
\[ E'\supseteq F'\supseteq F\]
and apartments $\Sigma',\ \Sigma'_{E'},\ \Sigma'_{F'}$
in the buildings $\mathcal{I}',\ \mathcal{I}'_{E'},\ \mathcal{I}'_{F'}$ respectively.
We then obtain a commutative diagram of bijections, where the lines are induced by isomorphisms of chamber complexes which preserve the orientation. 
\begin{equation*}
\xymatrix{
|\Sigma| \ar[r]^{{\equiv}_F} \ar[d]^{j_E} & |\Sigma'_{F'}| \ar[d]^{j_{E'}} \\ |\Sigma_E|\ar[r]^{{\equiv}_E} & |\Sigma'_{E'}|
}
\end{equation*}
The map ${\equiv}_F$ is given by
\[[x\mapsto \bigoplus_{i=0}^{m-1}v_i\textfrak{p}_D^{[d(x+\alpha_i)]+}]\mapsto [x\mapsto \bigoplus_{i=0}^{m-1}e_i\textfrak{p}_{\Centr_{D'}(L'_r)}^{[d(x+\alpha_i)]+}]\]
and ${\equiv}_E$ analogously. Here $(e_i)$ is the standard basis of $D'^{m}.$ Because of lemma \ref{lemmaChangingSkewfield} the map ${\equiv}_F$ preserves the embedding type and thus 
we can finish the proof by applying lemma \ref{LemRankReduction} on
\[\Sigma'\ra\Sigma'_{F'}\ra\Sigma'_{E'}.\]
More precisely, let $S_r$ be a face of rank $r$ in $\Sigma'_{F'}.$ Its barycenter has affine coordinates in $\frac{1}{r}\mathbb{Z}^{m-1}$ and therefore the preimage of it under $j_{F'}$ is a point $S_1$ with integer affine coeffitients, i.e. it corresponds to a vertex of $\mathcal{I}'.$ 
To emphasise the base field we write field extensions as the index of $j.$ Because of 
\[j_{E'|F'}(M_{S_r})=j_{E'|F'}(j_{F'|F}(S_1))=j_{E'|F}(S_1)\]
the theorem follows now from the rank reduction lemma and case 1.
\end{proof}
\part{Appendix, references and indexes}
\appendix
\chapter{The Weil-restriction}\label{appWeilRestriction}

Good references are \cite[1.3.]{weil:82} and \cite[20.5]{knus:98}

Let $E|F$ be a finite separable field extension, and let $V$ be an affine variety defined over $E.$
The functor
\[B\mapsto V(\tens{E}{F}{B})\]
from the category of commutative $F$-algebras to the category of sets is represented by an absolutely reduced finitely generated $F$-algebra $\tilde{A},$
see for example the proof in \cite[20.6]{knus:98}. The affine variety corresponding to $\tilde{A}$ is called \textit{Weil-restriction}\idxD{Weil-restriction} of $V$ from $E$ to $F$ and denoted by $\Res_{E|F}(V).$\idxS{$\Res_{E\mid F}(V)$}

Another way to construct the Weil-restriction is the following.
One introduces coordinates in choosing an $F$-basis in $E$ and the polynomial equations defining $V$ become polynomial equations with coeffitients in $F.$ The set of solutions of these equations is the Weil-restriction of $V$ from $E$ to $F$ and the map
\[e\in E\mapsto (\sigma(e))_\sigma\]
induces an isomorphism
\[\Res_{E|F}(V)\cong\prod_{\sigma:E\hra\bar{F}} V^\sigma\]
defined over the normal hull of $E.$ Here $\sigma$ runs over the set of $F$-algebra monomorphisms from $E$ into $\bar{F}.$
\chapter{The building of a valuated root datum}\label{appChBuildingOfAValuatedRootDatum}

The aim of this section is to give the definition of a valuated root datum and its building as it is done in \cite[chapter 6 and 7]{bruhatTitsI:72}. For the theory of root systems see
\cite[chapter IV]{bourbaki:81}.

Let $V$ be a finite dimensional $\mathbb{R}$-\vs\ and let $\Phi$ be a root system in $V^*.$ We denote its dual root system in $V$ by $\hat{\Phi}.$ The reflection $r_a$ of $V$ corresponding to $a\in \Phi$ is defined by
\[r_a(v):=v-a(v)\hat{a}.\] The Weyl-group of $\Phi,$ i.e. the group generated by all $r_a,\ a\in\Phi,$ is denoted by $\leftexp{v}{W}.$ We take an invariant positive definite inner product on $V$ and we get a canonical isomorphism from $V$ to $V^*$ via 
\[v\mapsto (v,*).\] It transfers the action of $\leftexp{v}{W}$ to $V^*.$ The Weyl-group stabilizes $\Phi$ and $\hat{\Phi}.$ The fixed point sets of the orthogonal reflections $r_a$ give a cell decomposition of $V$, see for example \cite[chapter 1]{brown:89}. The chambers of $V$ are in one to one correspondence to the bases of $\Phi.$ We fix a basis of $\Phi.$ Let $\Phi^+$ be the set of positive roots of $\Phi$ corresponding to this basis. 

\section{Valuated root datum}\label{appSecValuatedRootDatum}

\bdfn\cite[6.1.1]{bruhatTitsI:72}
Let $G$ be an arbitrary group. A system \[(T,(U_a,M_a)_{a\in\Phi})\] is a \textit{root} \textit{datum} \textit{of} \textit{type} $\Phi$ \textit{in} $G$\idxD{root datum} if the following holds.
\bi
\item (DR 1) The sets $T$ and $U_a$ are subgroups of $G.$ The groups $U_a$ are non-trivial.
\item (DR 2) For all roots $a,b$ the commutator group $[U_a,U_b]$ is a subset of the group generated by all $U_{na+mb}$ where $n$ and $m$ run over all natural numbers for which $na+mb$ is a root. 
\item (DR 3) If $a$ and $2a$ are elements of $\Phi$ then $U_{2a}$ is proper subset of $U_a.$
\item (DR 4) The set $M_a$ is a right coset of $T$ in $G$ and it holds
\[U_{-a}^*:=U_{-a}\setminus \{1\}\subseteq U_aM_aU_a.\]
\item (DR 5) For all roots $a$ and $b$ and all $m\in M_a$ on has $mU_{b}m^{-1}\subseteq U_{r_a(b)}.$
\item (DR 6) The group $U^+$ generated by all $U_a$ with positive root $a$ and the group $U^{-}$ generated by all the other $U_a$ have the property that the intersection of $T.U^{+}$ with $U^{-}$ is $\{1\}.$
\ei
A root datum of type $\Phi$ is \textit{generative} if $G$ is generated by the union of $T$ and all $U_a,\ a\in\Phi,$ i.e.
\[G=<T\cup \bigcup_{a\in\Phi} U_a>.\]
\edfn

\brem\label{remRootDatum}\cite[6.1.2(4),(9),(10)]{bruhatTitsI:72}
Given a root datum the cosets $M_a,$ $M_{-a}$ and $M_a^{-1}$ equal and are determined by $(T,(U_a)_{a\in\Phi}).$ 
Let $N$ be the group generated by $T$ and all $M_a.$ There is a group epimorphism
\[\leftexp{v}{\mu}:\ N\ra \leftexp{v}{W}\] such that 
\[nU_an^{-1}=U_{\leftexp{v}{\mu}(n)}.\]
One has for example $\mu(M_a)=\{r_a\}.$
\erem

The initials DR stand for "donn\'ee radicielle" the name given in \cite{bruhatTitsI:72}. 
Such a root datum can be defined for any reductive group defined over $k.$ 

\bex \cite[6.1.3]{bruhatTitsI:72}\label{ExRootDatum}
Let $k$ be a field.

Step 1: The group $\SL_2(k)$ has a generative root datum of type $A_1$
\[(T,M_1,M_{-1},U_1,U_{-1})\]
where $T$ is the set of diagonal matrices, $U_1$ \/ (resp. $U_{-1}$) the set of unitary upper (resp. lower) triangular matrices and $M_1$ the set of antidiagonal matrices in $\SL_2(k).$

Step 2: Let $G$ be a split, affine, connected and simple group defined over $k.$ We fix a maximal $k$-split torus $T$ of $G.$ Let $\Phi$ be the set of roots $\Phi(G,T)_k$ of the action of $T(k)$ on $\Lie(G)(k).$ By \cite[18.7]{borel:91} there is a unique family of unipotent connected closed $k$-subgroups $(U_a)_{a\in\Phi}$ of $G$ such that there are $k$-isomorphisms
\[\theta_a:\ \bA^1(\bar{k})\ra U_a\] satisfying
\[\Inn(t)\circ\theta_a(x)=\theta_a(a(t)x)\text{ for all } x\in\bA^1(\bar{k}),\ t\in T.\]
One can choose the maps $\theta_a$ such that following two assertions hold.
\bi
\item For every root $a$ there is an isomorphism from $SL_2(\bar{k})$ to the subgroup generated by $U_a$ and $U_{-a}$ which maps the upper and the lower triangular unitary matrix with non-diagonal entry $u$ to $\theta_a(u)$ and $\theta_{-a}(u)$ respectively.  
\item For every pair of roots $a$ and $b$ such that $a\neq -b$ there is a family of integers $(C_{a,b,n,m})$ such that 
\[[\theta_a(u),\theta_b(u')]=\prod_{n,m}\theta_{na+mb}(C_{a,b,n,m}u^nu'^m),\ u,\ u'\in \bA^1(\bar{k})\]
where $n$ and $m$ run over the natural numbers with $na+mb \in \Phi.$ 
\ei
The system $(T(k),(U_a(k))_{a\in \Phi})$ is part of a generative root datum of type $\Phi$ in $G(K).$

Step 3: Let $G$ be a split semi-simple connected group defined over $k.$ Then by \cite[22.10]{borel:91} $G$ is 
an almost direct product of the minimial closed connected normal $k$-subgroups $G_i$ of $G$ of strictly positive dimension.
A maximal $k$-split torus $T$ of $G$ is the image of a maximal $k$-split torus $\prod T_i$
of $\prod_iG_i$ where $T_i$ is a maximal $k$-split torus of $G_i$ because the canonical isogeny
\[f:\prod_iG_i\ra G\] is central, see \cite[22.9, 22.6]{borel:91}. 
$f$ is also separable, i.e. the differential $d_ef$ is surjective, and therefore we have 
\[f(\prod_iG_i(k))=G(k)\]
The map $d_ef$ is in fact the isomorphism 
\[\oplus_i\Lie(G_i)\cong\Lie(G)\]
and taking $\prod_i(T_i)$-fixed points on the left and $T$-fixed points on the right side we obtain 
\[\oplus_i\Lie(T_i)\cong\Lie(T).\] Thus 
\[\prod_iT_i\ra T\] is separable and we obtain
\[f(\prod_iT_i(k))=T(k).\] We now take for every $i$ a root datum 
\[(T_i(k),(M_a(k),U_a(k))_{a\in\Phi(G_i,T_i)})\] as done in $G_i$ by step 2.
We now apply $f$ on the product of the root data and we obtain a generative root datum
\[(T(k),(M_a(k),U_a(k))_{a\in\Phi(G,T)})\]
of type $\Phi(G,T)$ in $G(k).$

Step 4: We assume now that $G$ is an affine reductive group defined over $k.$ Then the group 
$G^0/\Rad(G)$ is affine, connected, semisimple and defined over $k$ by \cite[18.2(ii),Prop. 11.21, 6.8]{borel:91}. Thus we can assume that $G$ is semisimple and connected.
Let $T$ be a maximal $k$-split torus of $G.$ One can choose a maximal torus $T'$ of $G$ which is defined over $k$ and contains $T$ by the remark below. $T'$ is split over a finite separable extention $k'$ of $k.$ We take $\Phi':=\Phi(G,T')$ and the groups $U'_a,\ a\in\Phi',$ obtained from step 2. For $a\in\Phi:=\Phi(G,T)_k$ we define $U_a$ as the closed subgroup of $G$ generated by all
$U_{a'}$ where $a$ is the restriction of $a'$ to $T.$ The tupel
\[(\Centr_G(T)(k),(U_a(k))_{a\in\Phi})\]
is part of a generative root datum of type $\Phi$ of $G(k).$
\eex

\brem
If $G$ is a connected reductive $k$-group and $T$ is a maximal $k$-split torus of $G.$ We can choose a maximal torus of $G$ containing $T$ which is defined over $k.$
\erem

\bproof
A maximal torus $S$ containing $T$ is split over $k^{sep}$ and is a subset of $H:=\Centr_G(T)^0,$
which is normalized by $S$ and $k$-closed. Thus by \cite[20.3]{borel:91} $H$ is defined over $k^{sep}.$ The separability of $k^{sep}|k$ implies that $H$ is defined over $k.$ The theorem 
\cite[18.2(i)]{borel:91} ensures the existence of a maximal torus of $H$ which is defined over $k.$ This torus contains $T$ and it is a maximal torus of $G$ because it is conjugated to $S$
in $H.$
\eproof
\vspace{1em}

We now come to the definition of a valuation of a root datum.
\bass\label{assRootDatum}
Let 
\[RD:=(T,(U_a,M_a)_{a\in\Phi})\]
be a generative root datum of type $\Phi$ of a group $G.$ We put $U_{2a}:=\{1\}$
if $a\in\Phi$ and $2a\not\in\Phi.$
\eass

\bdfn\cite[6.2.1]{bruhatTitsI:72}\label{defValuationOfARootDatum}
A \textit{valuation}\idxD{valuation of a root datum} of $RD$ is a family $\phi=(\phi_a)_{a\in\Phi}$ of maps 
\[\phi_a:\ U_a\ra \mathbb{R}\cup \{\infty \} \]
such that the following contitions hold.
\be
\item (V0) For every $a$ the image of $\phi_a$ has at least 3 elements.
\item (V1) For every $a$ and for every $r\in\mathbb{R}\cup\{\infty\}$ the set 
\[U_{a,r}:=\phi_a^{-1}([r,\infty])\] is a subgroup of $U_a$ and $U_{a,\infty}$ is trivial.
For the images one writes
\[\Gamma_a:=\phi_a(U^*_a)\text{ and }\Gamma_a' :=\{\phi_a(u)\mid u\in U^*_a,\phi_a(u)=\sup\phi_a(uU_{2a})\}.\]
\item (V2) For every $a$ and for every $m\in M_a,$ the function 
\[u\mapsto \phi_{-a}(u)-\phi_a(mum^{-1})\] is constant on $U_{-a}^*.$
\item (V3) For $a,b\in\Phi$ with $b\not\in -\mathbb{R}_{+}a$ and $k,l\in\mathbb{R}$ the commutator group
$[U_{a,k},U_{b,l}]$ is contained in the group generated by $U_{pa+qb,pk+ql},\ p,q\in \mathbb{N}$ with
$pa+qb\in \Phi.$
\item (V4) If $a$ and $2a$ are in $\Phi$ the map $\phi_{2a}$ is the restriction of $2\phi_a$ on $U_{2a}.$
\item (V5) If $a\in \Phi,\ u\in U_a$ and $u',u''\in U_{-a}$ such that $u'uu''\in M_a$ then 
\[\phi_a(u)=-\phi_{-a}(u').\]
\ee
\edfn

\brem
One has $\Gamma'_a=\Gamma_a$ if $2a$ is not in $\Phi.$
\erem

\bdfn
A valuation $\phi$ is \textit{discret}\idxD{discret valuation} if every group $\Gamma_a$ is a discret subgroup of $\bbR.$
\edfn

If $(k,\nu)$ is a non-Archimedian local field there is a valuation of the root datum of  
\ref{ExRootDatum} by \cite[6.2.3 and chapter 10]{bruhatTitsI:72}.
For $\SL_2(k)$ it is given by
\[\phi_1\left(\zzmatrix{1}{u}{0}{1}\right):=\phi_{-1}\left(\zzmatrix{1}{0}{u}{1}\right):=\nu(u).\]
and in the case of a split, semisimple $k$-group by
\[\phi_a(\theta_a(t)):=\nu(t).\]
These valuations and the valuations considered in part 1 of this thesis are discret.

\brem
Let $\phi$ be a valuation of $RD$ and let $\lambda:\ \Phi\ra \mathbb{R}_+^*$ be a function which is constant on the irreducible components of $\Phi.$ For a vector $v\in V$
the family $\psi:=\lambda \phi+v$ defined by
\[u\mapsto \lambda(a)\phi_a(u)+a(v)\]
is again a valuation of the root datum. If $\lambda$ is the constant function 1, then one says that 
$\psi$ is a \textit{translation}\idxD{translation of a valuation of a root datum} of $\phi$ by $v.$
\erem

\bdfn\cite[6.2.5]{bruhatTitsI:72}
Two valuations are \textit{equipollent}\idxD{equipollent valuations of a root datum}\/ if the second is a translation of the first by a vector of $V.$
The group $N$ acts on the set of valuations of $RD$ in the following way.
If $w$ is the element $\leftexp{v}{\mu}(n)$ for some $n\in N$ one puts
\[(n.\phi)_a(u):=\phi_{w^{-1}(a)}(n^{-1}un).\]
For $n\in N,\ v\in V$ and $\lambda:\Phi\ra \mathbb{R}$ on has 
\[n.(\lambda\phi+v)=\lambda (n.\phi)+\leftexp{v}{\nu}(n)(v).\]
\edfn 

\section{Building of a valuated root datum}\cite[chapter 7]{bruhatTitsI:72}\label{secBuildingOfAValuatedRootDatum}

\bass
In addition to assumption \ref{assRootDatum} we fix a valuation $\phi$ of $RD.$
\eass

Let $\Delta$ be the set of valuations of $RD$ equipollent to $\phi.$ It is an affine space over $V$ and we identify $\Delta$ and $V$ in choosing $\phi$ as the zero of $\Delta.$ The action of $N$ on the set of valuations of $RD$ restricts to an action of $\Delta$ and it defines a map $\nu$ from $N$ to the group of affine automorphisms of $\Delta.$ Its kernel is denoted by $H.$
The set of affine roots of $\Phi$ in $\Delta$ is the collection of the closed halfspaces 
\[\alpha_{a,k}:=\{x\in \calli{A}\mid \ a(x)+k\geq 0\}, a\in\Phi, k\in\Gamma_a'.\] The set of affine roots is denoted by $\Sigma.$ We put $U_{\alpha_{a,k}}:=U_{a,k}$ 
For a non-empty subset $S$ of $\Delta$ one defines 
\bi
\item $U_S$ to be the group generated by the $U_{\alpha}$ where $\alpha$ runs over the affine roots containing $S,$ and
\item $P_S:=HU_S.$
\ei
  
The \textit{Bruhat-Tits building $\calli{F}$ of $G$ with respect to $RD$ and $\phi$} is the set of equi\-valence classes of $G\times \Delta$ under the relation
\[(g,x)\sim (h,y)\text{ if and only if there exists an }n\in N\]
such that 
\[y=\nu(n)(x)\text{ and }g^{-1}hn\in P_{\{x\}}.\]
It is a $G$-set under the action on the first coordinate.
A subset $\Delta'$ of $\calli{F}$ is said to be an apartment of $\calli{F}$ if there is an element $g$ of $G$ such that $\Delta'=g\Delta.$

This definition does not need that $\phi$ is discrete. For the description of the faces we assume that $\phi$ is discrete, for the description in the general case see 
\cite[7.2]{bruhatTitsI:72}.
The faces of $\Delta$ are the cells of the cell decomposition given by the hyperplanes which are boundaries of affine roots, see \cite[chapter 1]{brown:89}. A subset $S$ of $\calli{F}$ is a \textit{face}\idxD{face} if there is a face $S'$ of $\Delta$ and an element $g$ of $G$ such that $S'=gS.$
\chapter{The enlarged building of a reductive group}\label{chapEnlargedBuildingOfAReductiveGroup}

\bass
In this section $k$ is a non-Archimedean local field with residual characteristic different from 2.
\eass

We follow the explaination in \cite[4.2.16]{bruhatTitsII:84}
The buildings introduced in \cite{corvallisTits:79} are already enlarged. 

Let $G$ be a connected affine reductive group defined over $k$ together with a Bruhat-Tits building $\calli{F}.$ Let $X^*(G)_k$ be the group of characters of $G$ defined over $k$ and let $V^1$ be the dual $\bbR$-vector space of $\tens{X^*(G)_k}{\bbZ}{\bbR}.$ The \textit{enlaged affine building of $(G,\calli{F})$}\idxD{enlarged building} is the set $\calli{F}\times V^1$ equipped with the $G(k)$-action 
\[g.(x,v):=(g.x,v+\theta(g))\]
where $\theta(g)$ is defined by 
\[\theta(g)(\chi):=-\nu(\chi(g)).\]
Apartments and faces carry over from $\calli{F}$ to $\calli{F}^1$ in the 
natural way. 

\begin{definition}
We say that there exists a proper enlarged building over $k$ if $X^*(G)_k$ is not trivial.
\end{definition}

We now discuss the cases where an enlarged building occurs in 
the case of the classical groups considered in this thesis. We fix a hermitian $k$-datum
\[((A,V,D),\rho,k_0,h,\epsilon,\sigma)\]
and we analyse below when $X^*(\bSU(h))_{k_0}$ or $X^*(\bU(h)^0)_{k_0}$ are trivial. 

\begin{theorem}\cite[Cor. 14.2]{borel:91}\label{thmtrkCharGp}
Let $G$ be an affine reductive group.
Then the following conditions are equivalent.
\be
\item The group is semisimple, i.e. the maximal normal connected solvable subgroup $\calli{R}(G)$ is trivial.
\item The connected component equals its commutator subgroup.
\item The center of $G^0$ is finite.
\ee 
\end{theorem}

\begin{remark}
A semisimple connected group equals to its commutator subgroup which implies the triviality of the character group. Examples for semisimple connected groups are 
$\bSL_n(\bar{k}),$ $\bSp_{2n}(\bar{k}),$ and $\bSO_{n'}(\bar{k})$ for $n,n'\geq 1$ except $n'= 2.$ The connectivity is seen using transvections and the semisimplicity follows because these groups are generated by the images of connected subsets of $\bSL_2(\bar{k}).$
By proposition \ref{propHermForms} we obtain a trivial character group for $\bSU(h)$
if $\sigma$ is
\bi
\item symplectic, or
\item orthogonal and $md\neq 2$, or 
\item unitary.
\ei
\end{remark}

We firstly analyse the unitary case.

\blem\label{lemXUnIsTrivial}
The group $X^*(\bU(h))_{k_0}$ is trivial if $\sigma$ is unitary.
\elem

\bproof
We have $D=k$ by theorem \ref{thmDScharlau}. Using the isomorphism 
\[(\tens{\End_k(V)}{k_0}{\bar{k}},\tens{\sigma}{k_0}{\bar{k}})\cong (\Matr_m(\bar{k})\times\Matr_m(\bar{k}),\tilde{\sigma})\]
with 
\[\tilde{\sigma}(B,C)=(C^T,B^T)\]
we obtain for a $k_0$-rational character $\chi$ of $\bU(h)$ that its restriction to $\U(h)$ must be 
a power of the determinant. The involution $\sigma$ is conjugated to the transpostion which implies 
\[\chi(x)=\chi(\sigma(x))\] for all 
$x\in\U(h).$ 
In addition $\sigma(x)$ is the inverse of $x$ for $x\in\U(h).$ Thus
the only possible values of $\chi$ on $\U(h)$ are 1 and $-1.$
Thus $\chi$ is trivial because $\bU(h)$ is connected and $\U(h)$ is Zariski-dense in $\bU(h)$ by \cite[18.3]{borel:91}. 
\eproof

\blem
Let $\sigma$ be orthogonal and assume $dm=2.$
There exist a proper enlarged Bruhat-Tits-building for $\bSU(h)$ over $k$ if and only if  $d=1$ and $h$ is isotropic. 
\elem

Before we start the proof we recall that the \textit{$k$-rank}\idxD{$k$-rank of a reductive group} of a reductive connected $k$-group is the dimension of a maximal $k$-split torus.

\bproof
We have $\bSU(h)\cong \bGm(\bar{k})$ defined over $k$ if $d$ is one and $h$ is isotropic, i.e. all characters are $k$-rational and the character group is free of rank one.

If $d=2$ there is an isomorphism from $\bSU(h)$ to $\bGm(\bar{k})$ defined over $\bar{k}$ but not over $k$ because of the different $k$-ranks. There is an element $a\in\SU(h)\setminus\{1,-1\}$ because $\SU(h)$ is Zariski-dense in $\bSU(h)$ by \cite[18.3]{borel:91}. The degree of $D$ over $k$ is $2$ and therefore the centralizer of $k[a]$ in $D$ is $k[a],$ especially the commutative group $\SU(h)$ is a subset of $k[a]$. In addition $k[a]$ is invariant under $\sigma.$ Thus we can apply lemma \ref{lemXUnIsTrivial} and we obtain that there is no polynomial multiplicative map from $\SU(h)$ to $\bGm(k).$

For the last part of the proof we assume that $d=1$ and $\bSU(h)$ is anisotropic. There is a $k$-basis of $V$ such that the Gram-matrix of 
$h$ is of the form 
\[\zzmatrix{e}{0}{0}{f}\]
and we identify $A$ with $\Matr_2(k).$ A short calculation shows that 
\[\bSU(h)=\left\{\zzmatrix{a}{cf}{-ce}{a}\mid a,c\in\bar{k}\ s.t.\ a^2+efc^2=1\right\}.\]
We fix square roots $\sqrt{e}$ and $\sqrt{-f}.$
The conjugation with
\[\zzmatrix{\sqrt{e}}{\sqrt{-f}}{\frac{\sqrt{e}}{2}}{-\frac{\sqrt{-f}}{2}}\]
maps $\bSU(h)$ to $\SU(\tilde{h})$ where the Gram-matrix of $\tilde{h}$ under the standart basis is 
\[\zzmatrix{0}{1}{1}{0}.\]
The explicit formula for the map is 
\[\zzmatrix{a}{cf}{-ce}{a}\mapsto \zzmatrix{a-c\sqrt{-ef}}{0}{0}{a+c\sqrt{-ef}}.\]
Thus a character of $\bSU(h)$ is of the form 
\[\zzmatrix{a}{cf}{-ce}{a}\mapsto (a+c\sqrt{-ef})^z\]
for some integer $z.$ The inverse of $(a+c\sqrt{-ef})$ is $(a-c\sqrt{-ef}).$ If $z$ is positive in the binomial expansion of $(a+c\sqrt{-ef})$ the coeffitient in front of $\sqrt{-ef}$ is zero because $\sqrt{-ef}\notin k$ because $h$ is anisotropic. Thus a $k$-rational character $\chi$ of $\bSU(h)$ fulfils 
\[\chi(x)=\chi(x^{-1})\]
for all $x\in \SU(h).$ The density of $\SU(h)$ in $\bSU(h)$ and the connectivity of $\bSU(h)$ imply that $\chi$ is trivial.
\eproof
\vspace{1em}

If we summarize the two lemmas and the remark we obtain the following proposition.

\bprop\label{propExEnlargedBTBClassGrps}
$X^*(\bSU(h))_{k_0}\neq 1$ if and only if $m=2$ and $d=1$ and $\sigma$ is orthogonal and $h$ is isotropic. If $\sigma$ is unitary there is no nontrivial $k$-rational character of $\bU(h).$ 
\eprop \cleardoublepage
\phantomsection
\addcontentsline{toc}{chapter}{References}
\renewcommand{\bibname}{References}
\bibliographystyle{alphadin}
\bibliography{bibliography}

\chapter*{Indexes}
\phantomsection
\addcontentsline{toc}{chapter}{Indexes}
\printindex*
\end{document}